\documentclass[11pt]{amsart}
\usepackage{setspace}
\usepackage{aliascnt}
\usepackage{amsmath, float}
\usepackage{eucal}
\usepackage{amssymb,amsmath,amsthm,latexsym, mathtools}
\usepackage{xcolor}
\usepackage[shortlabels]{enumitem}
\usepackage{microtype}
\usepackage{longtable,booktabs}
\usepackage{tikz-cd}
\usepackage{graphicx}
\usepackage{hyperref}
\usepackage{comment}
\usepackage{wrapfig}
\usepackage{pifont}
\usepackage[capitalise, nameinlink]{cleveref}

\theoremstyle{definition}
\newtheorem{thm}{Theorem}[section]

\crefname{subsection}{Subsection}{Subsections}
\Crefname{subsection}{Subsection}{Subsections}

\newaliascnt{lemma}{thm}

\aliascntresetthe{lemma}

\newaliascnt{corollary}{thm}
\newtheorem{corollary}[corollary]{Corollary}
\aliascntresetthe{corollary}

\newaliascnt{proposition}{thm}
\newtheorem{proposition}[proposition]{Proposition}
\aliascntresetthe{proposition}

\newaliascnt{conjecture}{thm}

\aliascntresetthe{conjecture}

\newaliascnt{remark}{thm}
\newtheorem{remark}[remark]{Remark}
\aliascntresetthe{remark}

\newaliascnt{definition}{thm}
\newtheorem{definition}[definition]{Definition}
\aliascntresetthe{definition}

\newaliascnt{example}{thm}

\aliascntresetthe{example}

\newaliascnt{question}{thm}

\aliascntresetthe{question}

\newaliascnt{notation}{thm}

\aliascntresetthe{notation}

\usepackage{todonotes}

\def\c{\mathbb{C}}
\def\e{\mathbb{E}}
\def\f{\mathbb{F}}

\def\m{\mathbb{M}}

\def\r{\mathbb{R}}
\def\z{\mathbb{Z}}

\def\ca{\mathcal{A}}
\def\ce{\mathcal{E}}

\def\Ext{\operatorname{Ext}}

\def\ker{\operatorname{ker}}

\def\Sq{\operatorname{Sq}}

\newcommand{\ddash}{/\!\!/}

\author{Prerna Dhankhar}\address{University of Kentucky}\email{prerna.dhankhar@uky.edu}

\title{On the Motivic \(K\)-Theories of Classifying Spaces of Cyclic \(2\)-Groups}

\begin{document}

\begin{abstract}
We compute the $2$-completed integral motivic homology, effective algebraic $K$-theory, and very effective Hermitian $K$-theory of the geometric classifying spaces $B\mu_{2^n}$, $n\geq 2$, over algebraically closed and real closed fields. We first carry out the corresponding classical calculations and, then obtain the motivic results using motivic spectral sequences and comparison techniques. The results reveal the classical $2$-torsion patterns together with motivic $\tau$-periodicity and $\rho$-torsion.\end{abstract}
\maketitle
\vspace{-1.5em}
\tableofcontents
\newpage

\section{Introduction}
The study of motivic invariants such as motivic cohomology, algebraic $K$-theory, and Hermitian $K$-theory sits at a crossroads of algebraic geometry and homotopy theory. These invariants of motivic spaces encode information about Chow groups and have proven especially powerful in the construction of counterexamples to the integral Hodge conjecture \cite{AH62}. 

The development of motivic stable homotopy theory by Morel and Voevodsky \cite{MV99} established a framework needed to translate questions in algebraic geometry into questions in homotopy theory, making it possible to apply familiar tools—such as spectral sequences and generalized cohomology theories—to motivic spaces. In this sense, these motivic invariants often refine their classical counterparts, encoding additional structures that become visible only in the motivic setting.
This perspective was instrumental in Voevodsky’s proof of the Milnor and Bloch--Kato conjectures, which established profound connections between motivic cohomology, étale cohomology, and Milnor $K$-theory.

From the topological perspective, generalized (co)homology theories of classifying spaces have been studied extensively. In particular, Bruner and Greenlees \cite{BG03,BG10} computed connective complex and real topological $K$-theory for a large family of finite groups, including cyclic groups. Their computations highlight both rich algebraic structure and the increasing complexity as the group order grows. These results naturally motivate analogous questions in motivic homotopy theory.

Over the past decade, computational motivic homotopy theory has undergone rapid development. Isaksen and Shkembi \cite{IS11} introduced $\c$-motivic analogs of connective $K$-theories and described their cohomology in terms of the motivic Steenrod algebra \cite{Voe10}, a pivotal result for the present paper. Further advances by Hill \cite{Hil11}, Dugger and Isaksen \cite{DI16a}, and others developed computational tools for the motivic Adams spectral sequence over $\r$ and established important convergence and structural results. More recently, Isaksen, Wang, and Xu \cite{IWX23} introduced powerful comparison techniques between classical and motivic Adams spectral sequences, leading to substantial new computations of motivic stable stems and significantly expanding the computational reach of motivic stable homotopy theory.

Building on these developments, in \cite{DFNQJ25} we, together with Field, Nigam, Quigley, and Yang, applied motivic Adams and Atiyah--Hirzebruch spectral sequence techniques to compute the $2$-completed integral motivic homology, algebraic $K$-theory, and Hermitian $K$-theory of geometric classifying space of the cyclic group of order $2$ over algebraically closed fields, the real numbers, and finite fields. These computations illustrated the interplay between classical and motivic phenomena and provided a foundation for the investigations carried out in the present work.

Extending this line of investigation, the present paper computes the classifying spaces of cyclic groups of order $2^n$, where $n>1$.
Let $F$ be either an algebraically closed field of characteristic not equal to $2$ or a real closed field. 
\begin{thm}\label{thm:Main Theorem}
Let $n>1$.  
The $2$-completed $F$-motivic homology of the motivic geometric classifying space of the cyclic group of order $2^n$ is explicitly computed as a module over the respective coefficient rings in each of the following cases:
\begin{enumerate}
\item Integral motivic homology (see \cref{thm:c-HZ} and \cref{thm:r-HZ});
\item Effective algebraic $K$-theory (see \cref{thm:c-kgl} and \cref{thm:r-kgl});
\item Very effective Hermitian $K$-theory (see \cref{thm:c-kq} and \cref{thm:r-kq}).
\end{enumerate}
These computations are complete in all stems for integral motivic homology and through stem $14$ for the 
$K$-theories. For $n=2$, all three computations are
complete in all stems.
\end{thm}

Although our computations are carried out over the complex numbers $\c$ and $\r$, they are not essentially tied to this choice of base field. The Lefschetz principle suggests that phenomena observed over $\c$ and $\r$ persist over any algebraically closed field of characteristic not equal to $2$ and real closed fields, respectively, as discussed in \cref{subsec:extention}.

We begin with the classical computations and then adapt the arguments to the $\c$- and $\r$-motivic settings. As in the case of $BC_2$ \cite{DFNQJ25}, over $\c$, the resulting charts closely parallel the classical ones, enriched by motivic $\tau$-periodicity. Over $\r$, however, additional $\rho$-torsion phenomena produce genuinely new structures invisible classically and over $\c$.
Nevertheless, there are natural surjections
\[H\z^{motivic}_{**}(B\mu_{2^n}) \twoheadrightarrow H\z^{classical}_{*}(BC_{2^n})\]
\[kgl^{motivic}_{**}(B\mu_{2^n}) \twoheadrightarrow ku^{classical}_{*}(BC_{2^n})\]
\[kq^{motivic}_{**}(B\mu_{2^n}) \twoheadrightarrow ko^{classical}_{*}(BC_{2^n}).\]
For $n>1$, new phenomena distinguish the computations for $B\mu_{2^n}$ from the case $n=1$. In particular, the action of $\Sq^1$ on the motivic cohomology of $B\mu_{2^n}$ for $n>1$ is trivial. This requires a different approach for computing the $E_2$-page and differentials of the motivic Adams spectral sequence from the $n=1$ case, where the nontrivial $\Sq^1$-action plays a central role in the computations.

Our computations proceed in two stages. First, we determine the order of the groups via the Atiyah--Hirzebruch spectral sequence associated to the cellular filtration of $B\mu_{2^n}$. The attaching maps govern the differentials in this spectral sequence and, in most cases, force it to collapse at a finite page. Second, we resolve extension problems and determine the module structure over coefficient rings using the motivic Adams spectral sequence which takes the form:
\[
E_2^{s,(t,w)} = \Ext_{\ca^F}^{s,(t,w)}\!\Big(H^{**}(X), \m_2^F\Big)
\;\Rightarrow\;
\pi_{t-s,w}(X)^{\wedge}_{(2,\eta)}.
\]
Over $F=\c$ and $\r$, the $(2,\eta)$-completion agrees with $2$-completion, so this spectral sequence computes the desired groups. We discuss this framework in detail in \cref{sec:strategy}.

A key feature of our approach is that we first compute the $ko$-homology of $BC_{2^n}$ and then use it to compute the $ku$-homology. 
The $ko$-homology has a richer structure and it is more accessible to our methods.  The periodicization map
\[
p: ko_*(BC_{2^n}) \to KO_*(BC_{2^n}),
\]
and studying the eta invariant on periodic $KO$-homology, we obtain information about the connective $ko$-homology. This approach is discussed in detail in \cref{app:eta-invariant}.
Then, we describe a map between the $E_2$-pages of the corresponding Adams spectral sequences that allows us to transfer information from $ko$ to $ku$.

Beyond the explicit calculations, we highlight the framework combining Atiyah--Hirzebruch and motivic Adams spectral sequences to analyze these motivic homology theories.

\subsection{Future Work}
First, the author is currently working toward completing the computation of higher differentials in the Adams spectral sequence for $n>2$. Our approach is to exploit the correspondence between differentials in the Bockstein spectral sequence and $h_0$-towers in the motivic Adams spectral sequence, as described by the \emph{May--Milgram theorem}. Bruner and others have used a similar approach for quaternion groups in their work in \cite{BB96}.

Then, we plan to extend these computations \emph{ to finite fields $\f_q$} for odd characteristics, $q>1$. These calculations reveal rich and intricate structures, and are currently in progress.

Another natural direction is to investigate equivariant refinements of these motivic homology theories. The equivariant Betti realization functor connects the motivic stable homotopy category $Ho(SH(\r))$ to $C_2$-equivariant stable homotopy theory \(Ho(Sp^{C_2})\), whose homotopy groups are graded over \(RO(C_2)\). 
From a computational perspective, this relationship allows one to push $\r$-motivic computations to $C_2$-equivariant data.

The $C_2$-equivariant analogues of algebraic and Hermitian $K$-theory are denoted by $k\r$ and $ko_{C_2}$, respectively.
Motivated by my $\r$-motivic calculations, I plan to investigate $k\r_{**}(B_{C_2}\mu_{2^n})$ and
$(ko_{C_2})_{**}(B_{C_2}\mu_{2^n})$.

\subsection{Organization}

\cref{sec:background} gathers the essential background, beginning with the motivic analogs of cohomology of a point, steenrod algebra and connective $K$-theories. Then we continue with the structure of the mod–two motivic cohomology of the geometric classifying space $B\mu_{2^n}$ for $n>1$, viewed as a module over a suitable subalgebra of the motivic Steenrod algebra.

In \cref{sec:strategy} we describe the general framework that guides all of our computations, emphasizing how the Atiyah--Hirzebruch and Adams spectral sequences fit together in this setting. \cref{sec:classical} begins with the classical calculation, allowing us to highlight the familiar patterns before turning to the richer $\mathbb{C}$- and $\mathbb{R}$-motivic analogues. \cref{Integral Homology}, \cref{kgl}, and \cref{kq} form the heart of the paper. There, we compute the integral homology, effective algebraic $K$-theory, and very effective hermitian $K$-theory, respectively. 

\subsection{Acknowledgments}

This project grew out of a summer collaborative workshop with Rebecca Field, Arjun Nigam, J.\,D. Quigley, and Albert Jinghui Yang. I am also grateful for many insightful conversations with Bert Guillou and Arun Debray, which clarified and strengthened several parts of this work. I also want to thank Jackson Morris for his thoughtful feedback on earlier drafts.

\section{Background} \label{sec:background}

In this section, we recall some definitions and basic facts from motivic homotopy theory that will be used throughout this paper.

Throughout this section, let $F$ be a field, either $\c$ or $\r$. Here $\c$ and $\r$ serve as representative base fields for the algebraically closed and real
closed fields. The motivic Lefschetz principles recalled in \cref{subsec:extention} extend the calculations considered here to algebraically
closed fields of characteristic not equal to $2$ and to real closed
fields, respectively.

\subsection{Cohomology of a point}

Define
\(
\m_2^F := H^{**}(\mathrm{pt}; \f_2)
\)
to be the bigraded \(F\)-motivic mod 2 cohomology ring of a point. Here the first grading denotes the cohomological degree and the second denotes the weight.

By Voevodsky (\cite{Voe03}, \cite{Voe10}),
\[
	\m_2^F \;\cong\;
	\begin{cases}
		\f_2[\tau], & F=\c,\; \\
		\f_2[\tau, \rho], & F=\r
	\end{cases}
	\]
where \(\tau\) has bidegree \(|\tau|=(0,1)\) and \(\rho\) has bidegree \( |\rho| = (1,1).\)

\subsection{The Motivic Steenrod Algebra}

The mod two motivic Steenrod algebra over \(F\), denoted \(\ca^F\), is the ring of bistable cohomology operations generated by Steenrod squares (as an algebra over \(\m_2^F\)),
\[
\Sq^i \colon H^{p,q}(-;\f_2) \longrightarrow H^{p+i,q+\lfloor i/2 \rfloor}(-;\f_2),
\]
subject to motivic Adem relations; see \cite{Voe03}.

\(\ca^\c\) closely resembles the classical Steenrod algebra \(\ca\). One of the major differences is the presence of the class \(\tau\). In fact, there is a map
\[
\ca^{\c} \otimes_{\m_2} \m_2[\tau^{-1}] \;\to\; \ca \otimes_{\m_2} \m_2[\tau^{-1}]
\]
defined by sending
\[
\Sq^{2k} \longmapsto \tau^{-k} \Sq^{2k}
\qquad\text{and}\qquad
\Sq^{2k-1} \longmapsto \tau^{-k} \Sq^{2k-1},
\]
which is an isomorphism of bigraded rings, see \cite[Theorem~2.9]{DI10} for details.

\(\ca^\r\) contains additional structure arising from the class \(\rho \in H^{1,1}(\mathrm{pt};\f_2)\). In particular, the Adem relations and coproduct formulas involve \(\rho\)-terms, and this extra structure plays a central role in computations via the \(\rho\)-Bockstein and motivic Adams spectral sequences. A detailed description of \(\ca^{\r}\), including its dual Hopf algebra and coproduct structure, can be found in \cite{DI16a}. 

\begin{remark}
\(\ca^F\) acts naturally on the motivic cohomology \(H^{**}(X;\f_2)\), for a motivic spectrum or space \(X\), endowing it with the structure of a left \(\ca^F\)-module. 
\end{remark}
In particular, \(\m_2^F\) carries a canonical \(\ca^F\)-module structure. For example, the action of \(\ca^{\c}\) on \(\m_2^{\c} \cong \f_2[\tau]\) is trivial, since \(\m_2^{\c}\) is concentrated in cohomological degree zero. In other words, all positive-degree $\Sq^i$ vanish for degree reasons.
In contrast, the action of \(\ca^{\r}\) on \(\m_2^{\r} \cong \f_2[\tau,\rho]\) is nontrivial: for example, \(\Sq^1(\tau) = \rho\),\; \(\Sq^2(\tau^2) = \tau \rho^2\), and there are additional higher nontrivial Steenrod operations.

We now introduce motivic analogs of some classical Steenrod subalgebras.

\begin{definition}
\label{def:subalgebras}
Let 
\begin{enumerate}
\item \(\ca^F(0)\) be the \(\m_2^F\)-subalgebra of \(\ca^F\) generated by
\( Q_0 = \Sq^1.\)

\item \(\ce^F(1)\)  be the \(\m_2^F\)-subalgebra of \(\ca^F\) generated by
\(
Q_0 \text{ and }  Q_1 = \Sq^1\Sq^2 + \Sq^2\Sq^1.
\)

\item \(\ca^F(1)\) be the \(\m_2^F\)-subalgebra of \(\ca^F\) generated by \(\Sq^1\) and \(\Sq^2\).
\end{enumerate}
\end{definition}

\begin{definition}
For a subalgebra \(B \text{ of } \ca\), define
\[
\ca\ddash B := \ca \otimes_B \m_2^F,
\]
where \(\m_2^F\) is the trivial \(B\)-module with generator concentrated in degree zero.
In this paper, \(B\) will always refer to one of \(\ca^F(0)\), \(\ce^F(1)\), or \(\ca^F(1)\).
\end{definition}

\subsection{Motivic ``Connective'' \texorpdfstring{$K$-theories}{K-theories}}\label{theories}

In the classical stable homotopy category, the connective \(K\)-theory spectra \(ku\) and \(ko\) arise as the connective covers of the topological K-theory spectra \(KU\) and \(KO\), respectively. These spectra play a central role in stable homotopy theory, and their cohomology admits well-known descriptions in terms of the classical Steenrod algebra.

Motivic homotopy theory admits closely analogous constructions. In the
$F$-motivic stable homotopy category $\mathbb{SH}(F)$, the full
subcategory $\mathbb{SH}(F)^{\mathrm{eff}}$ of effective motivic spectra
is the localizing subcategory generated by the suspension spectra of
smooth $F$-schemes. It carries the effective homotopy $t$-structure,
whose nonnegative part is the subcategory
$\mathbb{SH}(F)^{\mathrm{veff}}$ of very effective motivic spectra. These
subcategories give rise to the effective and very effective slice
filtrations, respectively.

The effective algebraic $K$-theory spectrum
\(
kgl^F
\)
is obtained as the effective cover of algebraic $K$-theory and serves as
the motivic analogue of $ku$, while the very effective Hermitian
$K$-theory spectrum
\(
kq^F 
\)
is obtained as the very effective cover of Hermitian $K$-theory and
serves as the motivic analogue of $ko$. These connective motivic
$K$-theory spectra are central to our study, both because they reproduce
the structural relationship between $ko$ and $ku$ and because their
motivic cohomology is amenable to explicit computation. The effective slice filtration was introduced by Voevodsky
\cite{Voe01}, while its very effective refinement was introduced by
Spitzweck and \O stv{\ae}r \cite{SØ12} and further developed by Bachmann
\cite{Bac17}.

A key feature of very effective Hermitian $K$-theory is its relationship with effective algebraic $K$-theory via the motivic Hopf map, as described for example in \cite{IS11,ARO17} and recalled in the following proposition.
\begin{proposition}
\cite[Proposition 2.11]{ARO17}
Over a field $F$ of characteristic not equal to $2$, there is a cofiber sequence of motivic spectra
\[
\Sigma^{1,1} kq^F \xrightarrow{\ \eta\ } kq^F \longrightarrow kgl^F,
\]
where the map is induced by multiplication by $\eta$. This cofiber sequence expresses effective algebraic $K$-theory as the $\eta$-cofiber of very effective Hermitian $K$-theory and may be regarded as the motivic analogue of the classical relationship between $ko$ and $ku$.
\end{proposition}

Let us recall the classical descriptions of the mod \(2\) cohomology of \(H\z\), \(ku\), and \(ko\) as modules over \(\ca\); see, for example, \cite{adams74}. These descriptions play a fundamental role in computations with the classical Adams spectral sequence.

\begin{proposition}\label{prop: classical description}
There are isomorphisms of $\ca$-modules
\begin{enumerate}
\item $H^*(H\z;\f_2) \cong \ca \ddash \ca(0)$,
\item $H^*(ku;\f_2) \cong \ca \ddash \ce(1)$,
\item $H^*(ko;\f_2) \cong \ca \ddash \ca(1)$.
\end{enumerate}
\end{proposition}

In the motivic setting, analogous descriptions hold for the motivic spectra $H\z^F$, $kgl^F$, and $kq^F$ described in the following proposition for the base fields $F=\c$ and $F=\r$.
\begin{proposition}\label{prop: motivic description}
Let $F=\c$ or $\r$. There are isomorphisms of $\ca^F$-modules
\begin{enumerate}
\item $H^{**}(H\z^F;\m_2^F) \cong \ca^F \ddash \ca^F(0)$,
\item $H^{**}(kgl^{F};\m_2^F) \cong \ca^{F} \ddash \ce^{F}(1)$,
\item $H^{**}(kq^{F};\m_2^F) \cong \ca^{F} \ddash \ca^{F}(1)$.
\end{enumerate}
The case $F=\c$ is proven in \cite[Thm.~5.7]{IS11}, while the case $F=\r$ is described in \cite[Section~2]{ARO17}.
\end{proposition}

\subsection{Motivic classifying space and its mod two motivic cohomology}\label{sec:coho}

As our goal is to compute motivic homology theories of motivic analogues of classifying spaces, we briefly recall the notion of the motivic classifying space used throughout this paper. We work with the \emph{geometric motivic classifying spaces}
 $B_{\mathrm{gm}}G$
constructed by Morel and Voevodsky \cite{MV99}. For simplicity, we write $BG$ for $B_{\mathrm{gm}}G$.  It is the motivic analogue of the topological classifying space: just as $BG=EG/G$, it is constructed as the quotient of a motivically contractible space by a free $G$-action.

Let $\mu_{2^n}$ denotes the group scheme of $2^n$-th roots of unity. It is defined by
\[
\mu_{2^n} := \ker\!\left(\mathbb{G}_m \xrightarrow{(\,\cdot\,)^{2^n}} \mathbb{G}_m\right), \text{ see for example,\cite[\S 6]{Voe03}}).
\]
Over algebraically closed fields, this coincides with the cyclic group of order $2^n$, i.e.\ $BC_{2^n}$.
The motivic classifying spaces of these group schemes have many nice and similar formal properties as the classical topological classifying spaces; see, for example, \cite{Tot14}.

Moreover, there is a realization functor from $\c$-motivic spectra to classical spectra, characterized by preserving weak equivalences and homotopy colimits and by sending the motivic suspension spectrum of a smooth scheme $X$ to the suspension spectrum of $X(\c)$. In particular, it carries $B\mu_{2^n}$ to the topological classifying space $BC_{2^n}$.

This realization functor is symmetric monoidal \cite{HO16} and induces natural maps
\[
H^{p,q}(X) \longrightarrow H^{p}(X(\c)),
\]
where the target is viewed as a bigraded object concentrated in weight $0$, as described in \cite{DI10}. After standard technical adjustments, these maps extend to a comparison between the $\c$-motivic Adams spectral sequence and the classical Adams spectral sequence, which we will use repeatedly throughout the paper.

\vspace{-0.3cm}
\begin{figure}[ht]
\hspace{-2.6em}
 \begin{minipage}{5.4in}
 For our computations, we care about the motivic cellular structure of these spaces. Analogous to the classical CW-structure, $B\mu_{2^n}$ admits a motivic cell structure with one cell, denoted by $x_i$, in each bidegree
 \[
(i, \lfloor \tfrac{i+1}{2} \rfloor), \quad i \ge 0.
\]
 The attaching maps are similarly structured: each even-dimensional
 cell is attached to the preceding odd-dimensional cell via a map of
 degree $h^n$. Here $h\in \pi_{0,0}$ denotes a class detected by $h_0$ in the motivic Adams spectral
 sequence, so $h^n$ is detected by $h_0^n$. Over $\c$, multiplication by $h$ agrees with the degree-$2$ map. Over $\r$, however, it is given by $2-\rho\eta$, see \cite[Definition~8.1]{DI16a}. Each cell in dimension
$i\equiv 0,1\pmod 4$ is attached to the preceding cell in dimension
$i\equiv 2,3\pmod 4$ via multiplication by the Hopf map $\eta$, which
is detected by $\Sq^2$. There are higher attaching maps as well, but we
will not need them, since they are not detected over
$\ca(0)$, $\ce(1)$, or $\ca(1)$, although they may be detected over the whole Steenrod algebra $\ca$.

\end{minipage} \hspace{0.08em}
\begin{minipage}{0.1in}
	\includegraphics[scale=0.6]{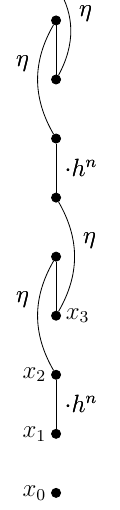}
\end{minipage}
\label{fig:cellstructure}
\end{figure}
\vspace{-0.2cm}

Recall,
\begin{proposition}\label{prop: classicala0e1a1} The mod--$2$ cohomology of the classifying space of $C_{2^n}$, $n\geq 2$ is given by
\[
H^*(BC_{2^n}; \f_2) \cong \f_2[x,y]/(x^2),
\]
where the degrees are $|x| = 1$ and $|y| = 2$ (see for example \cite[Proposition~4.5.1]{CTVEZ03}). 
The non-trivial action of the subalgebra $\ca(1)$ is completely determined by
\( \Sq^2(y) = y^2, \;\)
together with the Cartan formula and the Adem relations. 
\end{proposition}

\begin{remark} $H^{*}(BC_{2^n})$ contains a single copy of $\f_2$ in degree $0$, corresponding to the 0-cell in the cell decomposition. So, for convenience we work with \textbf{reduced mod 2 cohomology} which we still denote by $H^{*}(BC_{2^n})$.
\end{remark}

\begin{corollary}\label{cor: classicala0e1}
As a module over \(\ca(0)\) and \(\ce(1)\),
\[
H^{*}(BC_{2^n}) \cong \bigoplus_{i \geq 1} \Sigma^{i} \f_2 \cong \bigoplus_{i \geq 1} \f_2\{x_i\}.
\]
\end{corollary}

\begin{proof}
Since \(\Sq^1(x)=0\) and \(\Sq^1(y)=0\) by the above proposition, the actions of \(\ca(0)\) and \(\ce(1)\) are trivial, giving the claimed module structure. For convenience in computations, we relabel \(xy^{i} = x_{2i+1}\) and \(y^{i} = x_{2i}\) for \(i \ge 1\).
\end{proof}

\begin{corollary}\label{cor: classicala1}
As a module over $\ca(1)$, 
\[ H^{*}(BC_{2^n}) \cong \Sigma \f_2 \;\bigoplus_{i \ge 1} \Big( \Sigma^{4i+2}\ca(1)\ddash\ce(1) \;\oplus\; \Sigma^{4i+3}\ca(1)\ddash\ce(1)\Big),\]
where $\ca(1) \ddash \ce(1)$ is isomorphic to $\f_2 \oplus \Sigma^2 \f_2$ connected with $\Sq^2$.
\end{corollary}
\begin{proof}
Since \(\Sq^2(y)=y^2\) is the only nontrivial action on the generators, the action of \(\ca(1)\) decomposes \(H^{*}(BC_{2^n})\) into shifted copies of the submodule \(\ca(1) \ddash \ce(1)\). 
\end{proof}

Analogously, 
\begin{thm}
The $F$-motivic cohomology of $B\mu_{2^n}$ has module generators $x$ and $y$ in bidegrees $(1,1)$ and $(2,1)$ respectively and the action of $\ca(1)$ on generators is determined by
\[
\Sq^1(x) = 0 = \Sq^1(y), \qquad
\Sq^2(x) = 0, \qquad
\Sq^2(y) = y^2.
\]
Together with the Cartan formula and the Adem relations, this determines the  $H^{**}(B\mu_{2^n}; \f_2)$ as a module over the motivic Steenrod algebra $\ca(1)$.
\end{thm}

\begin{proof}
By \cite[Proposition~6.6]{Voe03}, the $F$-motivic cohomology of $B\mu_{2^n}$ with $\z/{2^n}$-coefficients admits an $\m_2^F$-basis
\[
\{\, v^i,\; u v^i \mid i \ge 0 \,\}.
\]
Here $v \in H^{2,1}(B\mu_{2^n};\z)$ is the Euler class of the line bundle over $B\mu_{2^n}$ associated to the tautological representation of $\mu_{2^n}$, and $u \in H^{1,1}(B\mu_{2^n};\z/{2^n})$ is the unique class characterized in \cite[Lemma~6.5]{Voe03} by the condition
\[
\delta_{2^n}(u) = v.
\]
The connecting homomorphism
\[
\delta_{2^n} \colon H^{1,1}(B\mu_{2^n};\z/{2^n})
\longrightarrow H^{2,1}(B\mu_{2^n};\z)
\]
is induced by the short exact sequence
\[
0 \longrightarrow \z \xrightarrow{2^n} \z \longrightarrow \z/{2^n} \longrightarrow 0.
\]

Let $x \in H^{1,1}(B\mu_{2^n};\f_2)$ and $y \in H^{2,1}(B\mu_{2^n};\f_2)$ denote the reductions of $u$ and $v$ mod $2$, respectively.  
The action of $\Sq^1$ on $x$ can be determined via the naturality of connecting homomorphisms with respect to coefficient reduction, which fits into the commutative diagram
\[
\begin{tikzcd}[column sep=5em, row sep=4em]
H^{1,1}(B\mu_{2^n};\z/{2^n})
  \arrow[r, "\delta_{2^n}\; (u \; \mapsto \; v)"]
  \arrow[d, "\text{mod }2"']
&
H^{2,1}(B\mu_{2^n};\z)
  \arrow[d, "2^{n-1}"]
\\
H^{1,1}(B\mu_{2^n};\f_2)
  \arrow[r, "\delta_2"']
&
H^{2,1}(B\mu_{2^n};\z).
\end{tikzcd}
\]

Commutativity implies that $\delta_2(x) = 2^{n-1}v$, and since $n \geq 2$, reducing modulo $2$ yields
\[
\Sq^1(x) = 0 \in H^{2,1}(B\mu_{2^n};\f_2),
\]
as claimed.

The remaining relations follow from comparison with the classical topological computation of $H^*(BC_{2^n};\f_2)$, together with Voevodsky’s identification of motivic reduced power operations with their classical counterparts in this setting. In particular,
\[
\Sq^2(y) = y^2 \quad \text{\cite[Lemma~9.8]{Voe03}},
\]
 $\Sq^2(x)=0$ \cite[Lemma~9.9]{Voe03} and $\Sq^1(y)=0$, since it lifts to the Euler class $v$. This further determines the full $\ca(1)$-action using the Cartan formula and the Adem relations.
\end{proof}

\hspace{-1.4em}
 \begin{minipage}{5.9in}
\begin{corollary}\label{cor: a0e1}
As a module over $\ca^F(0)$ and $\ce^F(1)$, we have a splitting
\[H^{**}(B\mu_{2^n}) \cong \bigoplus_{i \geq 1} \Sigma^{i,\lfloor \frac{i+1}{2}\rfloor}  \m_2^F \cong \bigoplus_{i \geq 1} \m_2^F{x_{i,\lfloor \frac{i+1}{2}\rfloor}} \]
\end{corollary}
\begin{corollary}\label{cor: a1}
As a module over $\ca^F(1)$, 
\[\begin{aligned}
	H^{**}(B\mu_{2^n}) & \cong \Sigma \m^F_2 \;\bigoplus_{i \ge 1} \Big( \Sigma^{4i+2}\ca(1)\ddash\ce(1) \;\oplus\; \Sigma^{4i+3}\ca(1)\ddash\ce(1)\Big)\\
& \cong \m^F_2{x_1} \;\bigoplus_{i \ge 1} \Big( \ca(1)\ddash\ce(1){x_{4i+2}} \;\oplus\; \ca(1)\ddash\ce(1){x_{4i+3}} \Big)
\end{aligned}
\]
where $\ca(1) \ddash \ce(1)$ is isomorphic to $\m^F_2 \oplus \Sigma^2 \m^F_2$ connected with $\Sq^2$.
\end{corollary}

\end{minipage} \hspace{-3.8em}
\begin{minipage}{0.1in}
   \includegraphics[scale=0.62]{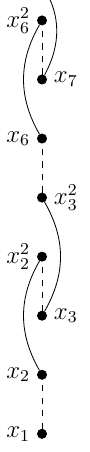}\label{fig:Coho}
\end{minipage}

\begin{remark} Again, $H^{**}(B\mu_{2^n})$ contains a single copy of $\m_2^F$ in bidegree $(0,0)$, corresponding to the 0-cell in the motivic cell decomposition. So, it is convenient to work with reduced mod 2 cohomology which we still denote by $H^{**}(B\mu_{2^n})$, as illustrated in the figure above.
\end{remark}

\subsection{Extension to other base fields}
\label{subsec:extention}

We conclude by explaining how the preceding calculations extend beyond
the base fields $\c$ and $\r$. Let $K$ be an algebraically closed field
of characteristic not equal to $2$. The motivic Lefschetz principle
of \cite[Proposition~5.2.1]{BCQ25} provides a zigzag of base-change
functors between the motivic stable homotopy categories over $K$ and
$\c$ which restricts to an equivalence on the full subcategories of cellular motivic spectra. This equivalence is compatible
with the corresponding motivic Adams spectral sequences and, in
particular, induces an identification
\[
\Ext_{\ca^K}(\m_2^K,\m_2^K)
\cong
\Ext_{\ca^\c}(\m_2^\c,\m_2^\c).
\]
Consequently, the calculations carried out over $\c$ extend to every
algebraically closed field of characteristic not equal to $2$.

Similarly, if $K$ is a real closed field, the real motivic Lefschetz
principle of \cite[Theorem~5.1]{DFNQJ25} identifies the corresponding
calculations with those over $\r$. Consequently, the calculations
carried out over $\r$ extend to every real closed field.


\section{General Idea}\label{sec:strategy}

The goal of the paper is to compute the \(2\)-completed effective and very effective algebraic \(K\)-theory of the classifying spaces \(B\mu_{2^n} , n>1\) over various fields. 
We first begin with the classical computations of the $2$-completed homology groups $H\z_* BC_{2^n}$, $ku_* BC_{2^n}$, and $ko_* BC_{2^n}$ and subsequently adapt these techniques to the $F$-motivic context.
For a detailed exposition of the classical tools, we refer the reader to \cite{BC18},\cite{BG03} and \cite{BG10}.

Our approach proceeds in two steps:  
\begin{enumerate}
	\item we determine some additive structure using the Atiyah--Hirzebruch spectral sequence;  
	\item we then verify the results from (1) and resolve potential extensions describing the multiplicative structure via the Adams spectral sequence.  
\end{enumerate}

Let $R = H\z, kgl$ or $ kq$ be the motivic spectra.

\subsection{First Step: The Atiyah--Hirzebruch Spectral Sequence}\label{subsec:AHSS}

	The Atiyah--Hirzebruch spectral sequence for computing
	\(R_{**}B\mu_{2^n}\)
	has the form
	\[
	E_1 = \bigoplus_{\substack {\text{i-cells } \\ \text{in } B\mu_{2^n}}} \Sigma^i R_{**} \;\;\Rightarrow\;\; gr^{*}_{AH}(R_{**}B\mu_{2^n}),
	\]
	where the differentials are determined by the attaching maps of the cellular decomposition of $B\mu_{2^n}$. Since each even-dimensional cell is attached to the preceding odd-dimensional cell via a map of degree $h^n$, the $d_1$-differentials take the form
	\[
	d_1(x[2i]) = h^n \cdot x[2i-1], \qquad i \geq 1,
	\]
	where $x[i]$ denotes an element in filtration $i$ with $x$ in the coefficient ring $R_{**}$. The non-zero $d_2$-differentials arise from attaching maps given by the first Hopf invariant one map $\eta$, so that an $i \equiv 0,1 \pmod 4$ cell is attached to the preceding $i \equiv 2,3 \pmod 4$ cell, yielding
	
	\[
	d_2(x[i]) = \eta \cdot x[i-2], \qquad i \equiv 0,1 \pmod 4,\quad i \geq 4.
	\]
	
	The \(d_3\)-differentials are
	determined, up to the usual Toda bracket indeterminacy, by the three-cell
	subquotients associated with two successive attaching maps; see
	\cite[Lemma~6.1]{WX17}. Thus,
	\[
	d_3(x[i])\in
	\begin{cases}
	\langle x,\eta,h^n\rangle[i-3],
	& i\equiv 0\pmod 4,\quad i>0,\\[2mm]
	\langle x,h^n,\eta\rangle[i-3],
	& i\equiv 2\pmod 4,\quad i>2.
	\end{cases}
	\]
	These Toda brackets are defined because the relevant consecutive
	composites vanish. When \(i\equiv 0\pmod 4\), the class \(x[i]\) can reach
	the \(E_3\)-page only if its \(d_2\)-differential vanishes, so that
	\(\eta\cdot x=0\). Similarly, when \(i\equiv 2\pmod 4\), existence of
	the \(d_1\)-differential implies \(h^n\cdot x=0\). The motivic null-Hopf relation \(h\eta=\eta h=0\), detected by the Adams relation \(h_0h_1=0\), is established in \cite[Theorem~1.4]{DI13}; see also \cite[\S~5.2, Table~2]{DI10}.

	In most cases, the spectral sequence collapses at the $E_4$-page, giving the ranks of the groups $R_{**}(B\mu_{2^n})$ as modules over $\m_2^F$. 
\begin{remark}
In all Atiyah--Hirzebruch spectral sequence charts, for example \cref{fig: HZ AHSS over R} appearing in this paper, the axes are indexed by (stem, topological degree of cell). The motivic weight $w$ is suppressed to simplify the presentation.
Differentials are shown in red (slope downward to the left).
\end{remark}

\subsection{Second Step: The Adams Spectral Sequence}\label{ASS}
The motivic Adams spectral sequence provides a primary computational tool for determining the completed motivic stable homotopy groups of a motivic spectrum.

For an $F$-motivic space $X$, the Adams spectral sequence with respect to the motivic Eilenberg--Mac Lane spectrum $H\f_2$ takes the form
\[
E_2^{s,(t,w)} = \Ext_{\ca^F}^{s,(t,w)}\!\Big(H^{**}(X), \m_2^F\Big)
\;\;\Rightarrow\;\;
\pi_{t-s,w}(X)^{\wedge}_{(2,\eta)}.
\] with differentials
\[
d_r \colon E_r^{s,(t,w)} \longmapsto E_r^{s+r,(t+r-1, w)}.
\]
 
For a motivic cell spectrum $X$ of finite type over a field of characteristic $0$, this spectral sequence converges to the homotopy groups of the $(2,\eta)$-completion of $X$ \cite[Thm.~1]{HKO11a}. Moreover, if
\(\operatorname{cd}_2(F(i))<\infty\), where \(i=\sqrt{-1}\), then the same
theorem identifies the \((2,\eta)\)-completion with the \(2\)-completion.
Consequently, over \(\c\) and \(\r\), the spectral sequence
computes the \(2\)-completed motivic stable homotopy groups.


Here, the groups $\Ext_{\ca^F}^{s,(t,w)}(X,\m_2^F)$ are trigraded, with grading conventions as follows:
\begin{itemize}
    \item $t-s$ is the \emph{stem};
    \item $s$ is the \emph{Adams filtration} (the filtration degree in the cobar complex);
    \item $w$ is the \emph{motivic weight};
    \item $(t,w)$ is the \emph{internal degree} (corresponding to the grading in $\ca^F$);
    \item $t-s-w$ is the \emph{Milnor--Witt degree} (also referred to as the coweight).
\end{itemize}

\begin{remark}
Throughout this paper, we write $\Ext_{\ca^F}^{s,(t,w)}(X)$ for $\Ext_{\ca^F}^{s,(t,w)}(X,\m_2^F)$ unless stated otherwise. In all Adams charts appearing in this paper, the axes are indexed by $(t-s,s)$, and the motivic weight $w$ is suppressed to simplify the presentation. 
Moreover, when writing $R_{**}(B\mu_{2^n})$ we implicitly mean the $2$-completed object, and we suppress the $2$-completion from the notation.
\end{remark}

Our goal is to compute the $2$-completed homology groups
\(
R_{**}(B\mu_{2^n})
\)
for the motivic spectra $R = H\z, kgl,$ and $kq$. So, we take
\(
X = B\mu_{2^n} \wedge R
\)
in the Adams spectral sequence. Using the standard identification
\[
\pi_{**}(B\mu_{2^n} \wedge R) \cong R_{**}(B\mu_{2^n}),
\]
the motivic Adams spectral sequence takes the form
\[
E_2^{**}
= \Ext_{\ca^F}^{***}\!\Big(H^{**}(B\mu_{2^n} \wedge R)\Big)
\;\overset{\textit{K{\"u}nneth}}{\cong}\;
\Ext_{\ca^F}^{***}\!\Big(H^{**}(B\mu_{2^n}) \otimes H^{**}(R)\Big)
\;\Rightarrow\;
R_{**}(B\mu_{2^n}).
\]
The isomorphism labeled ``K{\"u}nneth'' follows from the Künneth theorem \cite[Thm.~8.7]{DI05}. In our case, the higher Tor terms vanish since $H_{**}(B\mu_{2^n})$ is free as an $H_{**}$-module, yielding the stated tensor product decomposition.

When they occur, the differentials in this spectral sequence can often be understood by comparison with the Atiyah--Hirzebruch spectral sequence, which determines the ranks of the homology groups and detects certain necessary differentials, see for example in \cref{prop:diff in ko}. The module structure over $\Ext_{\ca^F}^{***}(\m_2^F)$ allows some differentials to extend naturally, while the remaining differentials are determined using more specific techniques, such as Betti realization, or known patterns of periodicity, see for example in \cref{thm:c-kgl}. 

\section{Classical Computations}\label{sec:classical}
The connective $K$-theory of the classifying spaces of cyclic groups of order $2^n$ was computed by Bruner and Greenlees in \cite{BG03,BG10}. In this section, we use their results to illustrate the strategy described in the previous section. Note that techniques used in \cite{BG03,BG10} to compute these homology groups differ from the ones used here.

\subsection{Integral Homology}

\begin{proposition}
	The $2$-completed integral homology of $BC_{2^n}$ is given by
	\[
	H\z_k BC_{2^n} \cong
	\begin{cases}
		\z/2^n, & k>0 \text{ odd},\\[1mm]
		0, & k>0 \text{ even}.
	\end{cases}
	\]
\end{proposition}
\begin{proof} Consider the Atiyah--Hirzebruch spectral sequence (\cref{subsec:AHSS}):
	\[
E_1 = \bigoplus_{\substack{\text{each cell } i \\ \text{in } BC_{2^n}}} \Sigma^i H\z_* = \bigoplus_{i \geq 1}  \z\{x[i]\} \;\Longrightarrow\; gr^*_{AH}H\z_* BC_{2^n},
\]
where \(H\z_* = \z\), and \(x[i]\) denotes the generator of \(x \in H\z_*\) in filtration \(i\).

 The only nontrivial \(d_1\)-differential,shown in \cref{fig:grAHHZ} are
\[
d_1(x[2i]) = 2^n x[2i-1], \quad i \ge 1.
\]

The spectral sequence collapses at \(E_2\) since $E_2$ is concentrated in odd degrees.
The associated graded of \(H\z_* BC_{2^n}\) with respect to the Atiyah--Hirzebruch filtration is
\[
	\mathrm{gr}^*_{AH} H\z_k BC_{2^n} \cong 
	\begin{cases}
		\mathbb{Z}/2^n, & k>0 \text{ odd},\\[1mm]
		0, & k>0 \text{ even}.
	\end{cases}
\]

There are no possible hidden extensions, so this completely determines the homology groups.
\end{proof}

\begin{figure}[H]
	\vspace{-0.5cm}
	\includegraphics[scale=0.4]{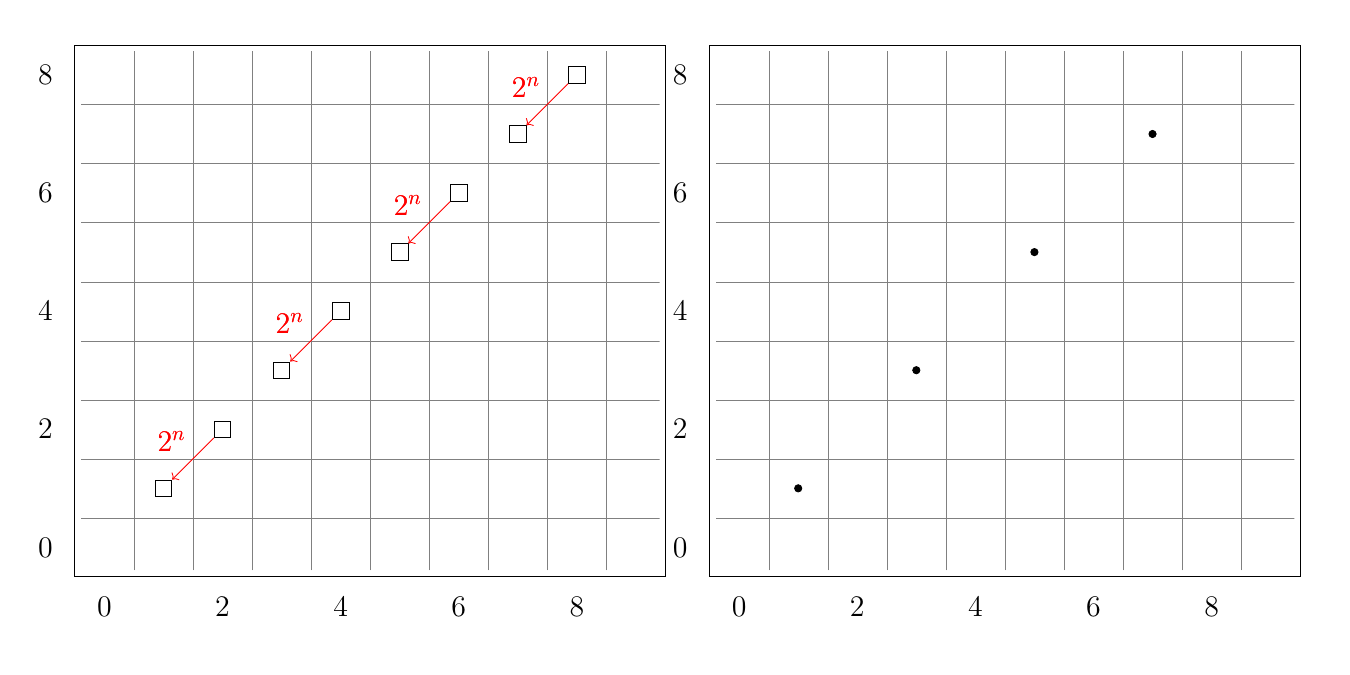}
	\vspace{-0.6cm}
	\caption{The $E_1$-page (left) and $E_{\infty}$-page (right) of the Atiyah--Hirzebruch spectral sequence converging to $gr^* H\z_{*}BC_{2^n}$. Each $\square$ represents a copy of $\z$,$\bullet$ represents a copy of $\z/2^n$, and the $d_1$-differentials are drawn in red (slanted down to the left).}
	\label{fig:grAHHZ}
\end{figure}

For reference, we compute the same result using the Adams spectral sequence (\cref{ASS}):
\[
E_2^{**} \cong \Ext_{\ca}^{**}\Big(H^*(BC_{2^n}) \otimes H^*(H\z)\Big) \Rightarrow H\z_*(BC_{2^n}).
\]

Using the isomorphism $H^* (H\z) \cong \ca \ddash \ca(0)$ from \cref{prop: classical description} and applying a change-of-rings isomorphism \cite{Rav86}, the $E_2$-term becomes
\[
E_2^{**} \cong \Ext_{\ca(0)}^{**}\Big(H^*(BC_{2^n})\Big).
\]
From the $\Sq^1$-action on $H^*(BC_{2^n})$ described in \cref{prop: classicala0e1a1}, we obtain an isomorphism of $\ca(0)$-modules:
\[
H^*(BC_{2^n}) \cong \bigoplus_{i \geq 1} \Sigma^i \f_2, 
\]
where $\f_2$ is the trivial $\ca(0)$-module.
Recall that
\[
\Ext_{\ca(0)}^{**}(\f_2,\f_2) \cong \f_2[h_0],
\]
where $|h_0| = (0,1)$ in Adams grading $(t-s,s)$ (see, for example, \cite[Remark~5.4.4]{BC18}).

The $E_2$-page may be rewritten as
\[
E_2^{**}
\cong \bigoplus_{k \geq 1} \Ext_{\ca(0)}^{**}(\Sigma^k \f_2)
\cong \bigoplus_{k \geq 1} \Sigma^k \Ext_{\ca(0)}^{**}(\f_2)
\cong \bigoplus_{k \geq 1} \f_2[h_0]\{x_k\}.
\]
Here, $x_k$ denotes a generator corresponding to the suspension $\Sigma^k \f_2$.

Comparing with the ranks obtained from the Atiyah--Hirzebruch spectral sequence forces
\[
d_n( x_{2k}) = h_0^{n} x_{2k-1}.
\]
All remaining differentials are then determined by $h_0$-linearity.		

Since the $E_{n+1}$-term is concentrated only in odd stems, no higher differentials are possible. The spectral sequence collapses and we conclude
\[
H\z_k BC_{2^n} \cong 
\begin{cases}
	\z/2^n, &  k>0 \text{ odd},\\[1mm]
	0, & \text{otherwise}.
\end{cases}
\]

\vspace{-0.5cm}
\begin{figure}[H]
	\vspace{-0.5cm}
	\includegraphics[scale=0.5]{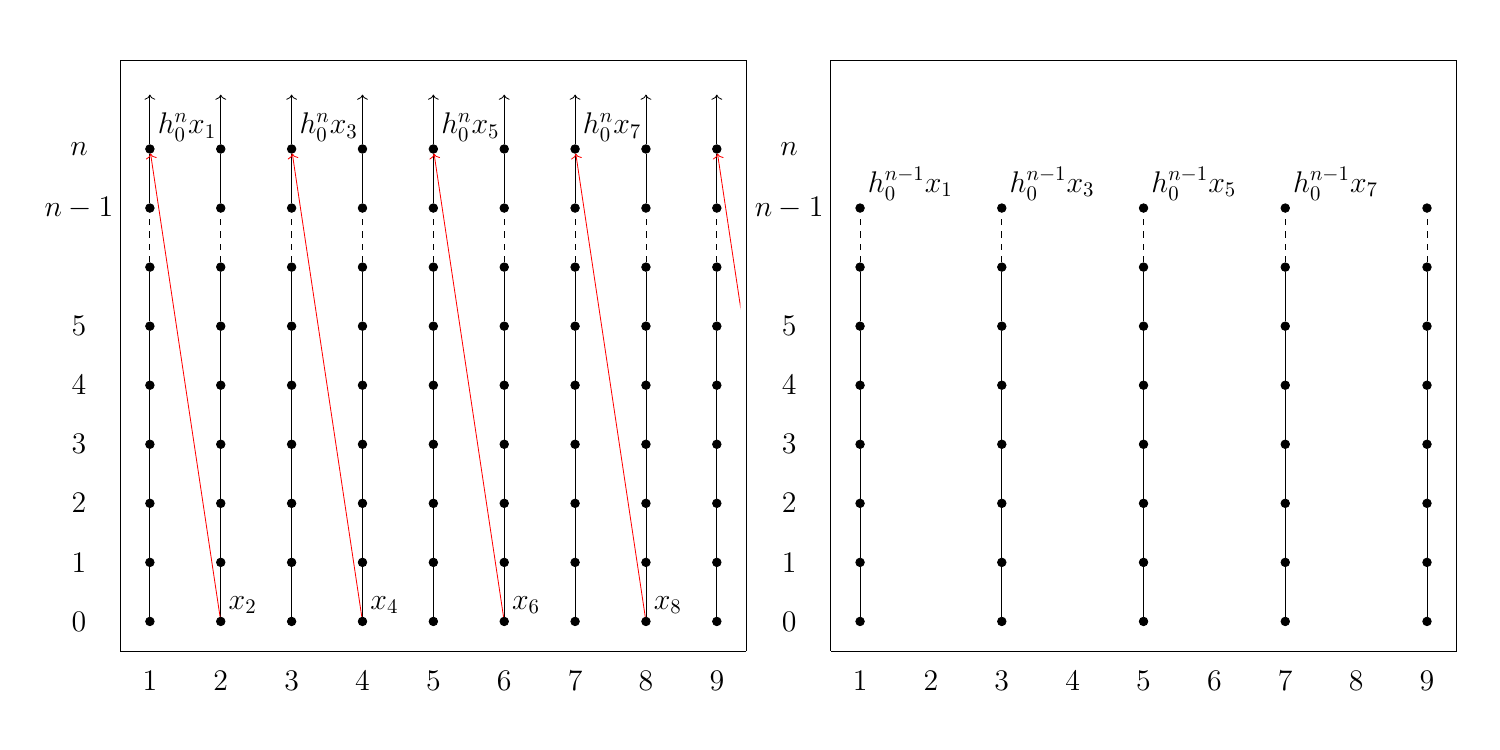}
	\vspace{-0.6cm}
	\caption{The $E_n$-page (left) and $E_{n+1} = E_{\infty}$-page (right) of the Adams spectral sequence converging to $H\z_{*}BC_{2^n}$. Each $\bullet$ represents a copy of $\f_2$, and the $d_n$-differentials are drawn in red (slanted upward to the left).}
	\label{fig:HZ}
\end{figure}

\subsection{Connective real K-theory} 
To compute the $ko$-homology of $BC_{2^n}$, the key ingredient for the Atiyah--Hirzebruch spectral sequence is the coefficient ring $ko_*$.

\begin{remark}
The groups $ko_*$ can be computed via the classical Adams spectral sequence, giving \[ ko_* \cong \frac{\z[\eta, \alpha, \beta]}{(\eta^3, 2\eta, \eta\alpha, \alpha^2-4\beta)}. \] 
\end{remark}

\begin{proposition}
The associated graded of \(ko_{*} BC_{2^n}\) with respect to the Atiyah--Hirzebruch filtration is
\[
gr^{*}_{AH} ko_k(BC_{2^n}) \cong
\begin{cases}
	\z/2^n \bigoplus \Big(\z/2^{n-1}\Big)^{\oplus 2m} &k = 8m+1> 0, \\
	\z/2 & k = 8m+2> 0, \\
	\Big(\z/2\Big)^{\oplus 2m+2} \bigoplus \Big(\z/2^{n}\Big)^{\oplus 2m+1} & k = 8m+3> 0, \\
	\Big(\z/2^{n-1}\Big)^{\oplus 2m+2} & k = 8m+5> 0, \\
	\Big(\z/2^{n} \oplus \f_2\Big)^{\oplus 2m+2} & k = 8m+7> 0, \\
	0 & otherwise\\
\end{cases}\]
\end{proposition}
\begin{proof} Consider the  Atiyah--Hirzebruch spectral sequence \cref{subsec:AHSS}
\[
E_1 = \bigoplus_{\substack{\text{each cell } i \\ \text{in } BC_{2^n}}} \Sigma^i ko_* \Rightarrow gr^{*}_{AH} ko_*(BC_{2^n}).
\]

We have \[
d_1(x[2i]) = 2^n \cdot x[2i-1],\qquad i > 0.
\]
Moreover, a cell in degree \(i\equiv 0,1\pmod 4\) is attached to the preceding cell in degree \(i\equiv 2,3\pmod 4\), respectively, via the Hopf map, which yields
\[
d_2(x[i]) = \eta \cdot x[i-2],\qquad i \equiv 0,1 \pmod 4,\quad i \ge 4.
\]
Finally, using the relation \(2^{n-1}\alpha\in\langle\eta^2,\eta,2^n\rangle\) in $ko_*$ we obtain
\[
d_3(\eta^2[4k]) = 2^{n-1}\alpha[4k-3],\qquad k \ge 1.
\]

\begin{figure}[H] \vspace{-0.5cm}
	\centering
	\includegraphics[scale=0.45]{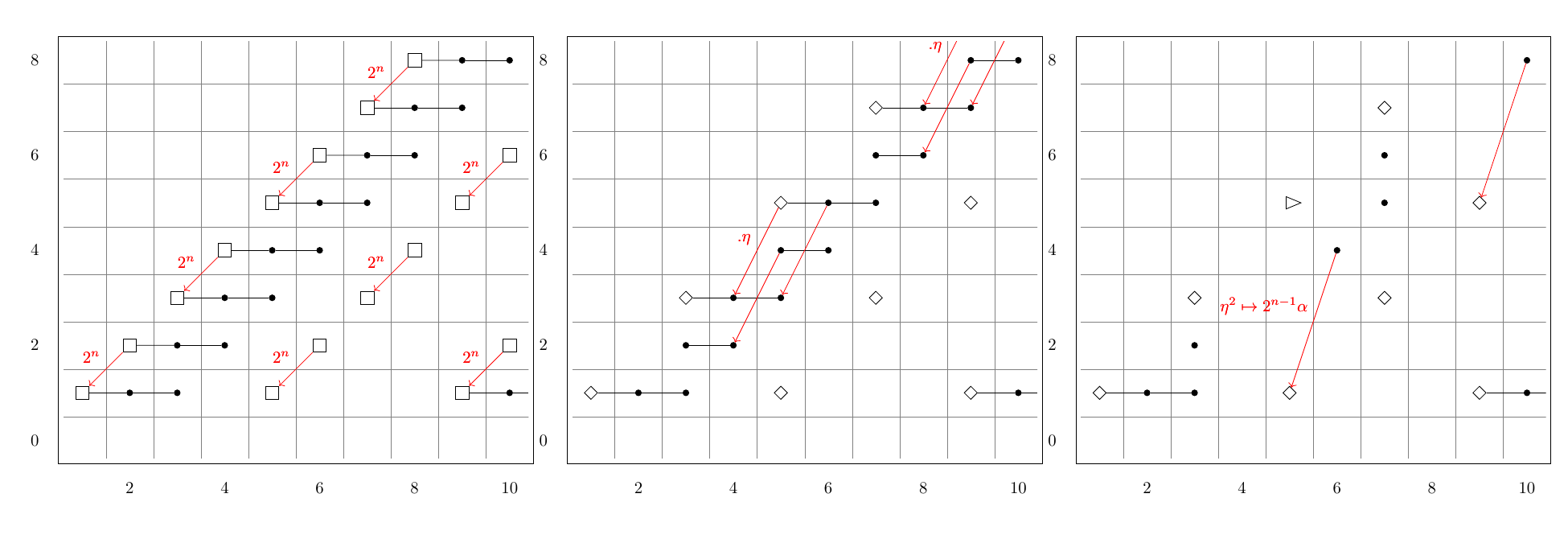}
	\vspace{-0.5cm} \caption{The $E_1 , E_2$ and $E_3$-pages (left to right) of the Atiyah--Hirzebruch spectral sequence converging to $gr^{*}_{AH} ko_*(BC_{2^n})$. Each $\square$ represents $\z$, $\diamond$ represents $\z/2^n$, $\triangleright $ represents $\z/2^{n-1}$ and each $\bullet$ represents $\f_2$. Horizontal lines represent multiplication by $\eta$. Differentials $d_1$, $d_2$, and $d_3$ are shown in red (slanted down).}
	\label{fig:E123koAH}
\end{figure}
Observe that all the target classes of higher differentials lie in even degrees and are of the form
\(\eta\cdot\beta^k[1]\).
\(\eta\)-linearity implies that the source of any higher differential must
also support nontrivial \(\eta\)-multiplication. However, every class in odd
degree is \(\eta\)-torsion, so no higher differentials can occur. The
associated graded \(ko\)-homology groups and their generators are depicted
in \cref{fig:koAH}.

\end{proof}
\vspace{-0.5cm}
\begin{figure}[H]
	\centering
	\includegraphics[scale=0.5]{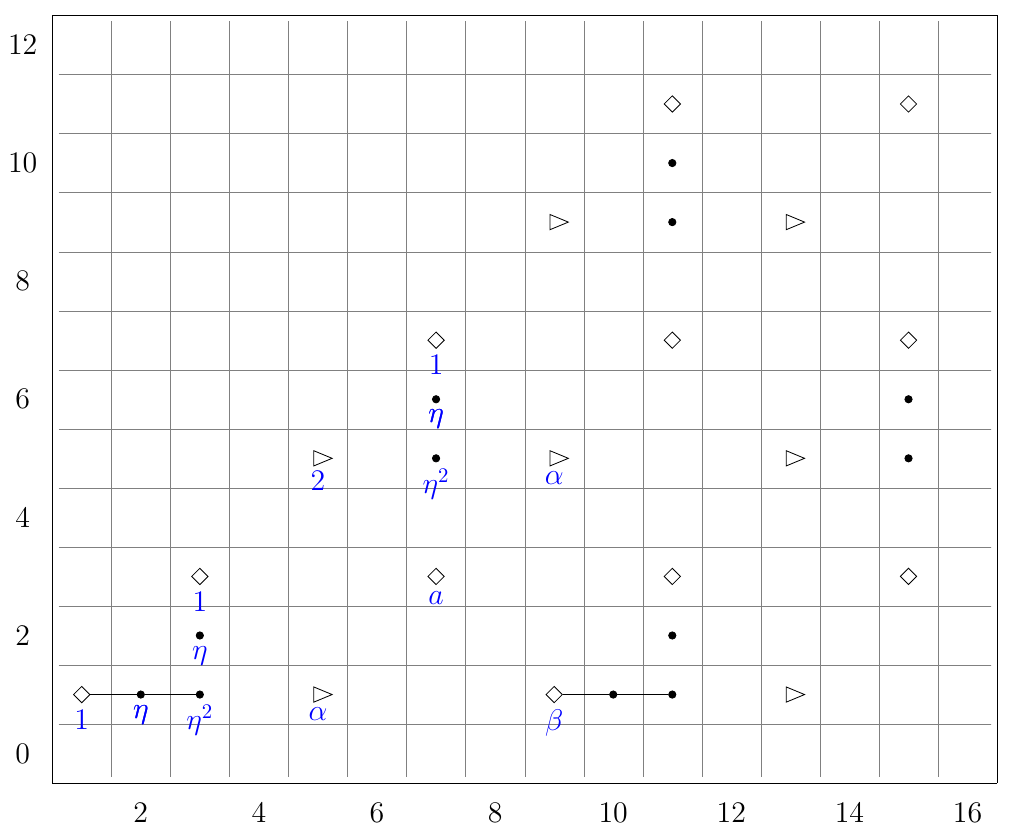}
	\vspace{-0.2cm} \caption{The chart is represents $gr^{*}_{AH} ko_*(BC_{2^n})$, where each $\square$ represents $\z$, $\diamond$  represents $\z/2^n$, $\triangleright $ represents $\z/2^{n-1}$ and each $\bullet$ represents $\f_2$. Horizontal lines represent multiplication by $\eta$.}
	\label{fig:koAH}
\end{figure}

To resolve the hidden extension problems, we now consider the Adams spectral sequence. By \cref{prop: classical description}, we have the isomorphism
\(
H^{**}(ko) \cong \ca \ddash \ca(1)
\) and further applying the change-of-rings isomorphism, the $E_2$-page simplifies to
\[
E_2^{**}
\cong
\Ext_{\ca(1)}^{**}(H^*(BC_{2^n}),\f_2).
\]

From \cref{cor: classicala1}, we have the decomposition
\[
H^*(BC_{2^n})
\cong
\Sigma \f_2
\;\bigoplus_{k \ge 0}
\left(
\Sigma^{4k+2}\ca(1)\ddash\ce(1)
\;\oplus\;
\Sigma^{4k+3}\ca(1)\ddash\ce(1)
\right).
\]
Applying $\Ext_{\ca(1)}^{**}(-,\f_2)$ and using the change-of-rings isomorphism yields
\begin{align*}
E_2^{**}
&\cong
\Sigma \Ext_{\ca(1)}^{**}(\f_2)
\;\oplus\;
\bigoplus_{k\ge0}
\left(
\Sigma^{4k+2}\Ext_{\ca(1)}^{**}(\ca(1)\ddash\ce(1))
\;\oplus\;
\Sigma^{4k+3}\Ext_{\ca(1)}^{**}(\ca(1)\ddash\ce(1))
\right)\\
&\cong
\Sigma \Ext_{\ca(1)}^{**}(\f_2)
\;\oplus\;
\bigoplus_{k\ge0}
\left(
\Sigma^{4k+2}\Ext_{\ce(1)}^{**}(\f_2)
\;\oplus\;
\Sigma^{4k+3}\Ext_{\ce(1)}^{**}(\f_2)
\right)\\
&\cong
\Ext_{\ca(1)}^{**}(\f_2)\{x_1\}
\;\oplus\;
\bigoplus_{k\ge0}
\left(
\Ext_{\ce(1)}^{**}(\f_2)\{x_{4k+2}\}
\;\oplus\;
\Ext_{\ce(1)}^{**}(\f_2)\{x_{4k+3}\}
\right).
\end{align*}

Recall from, for example, \cite[Section~4.5 and 4.6]{BC18}, \[ \Ext_{\ce(1)}^{**}(\f_2) \cong \f_2[h_0,v_1],\text{ and }
\Ext_{\ca(1)}^{**}(\f_2) \cong
\frac{\f_2[h_0,h_1,a,b]}
{(h_1^3,\;h_0h_1,\;h_1a,\;a^2-h_0^2b)}
\]
where \( |h_0|=(0,1)\), \( |h_1|=(1,1)\), \( |a|=(4,3)\), \( |b|=(8,4)\), and \( |v_1|=(2,1)\). The resulting $E_2$-page is shown in \cref{fig:diff ko}. 

Before calculating the differentials, we describe the multiplicative structure on the $E_2$-page of the Adams spectral sequence converging to $ko_*(BC_{2^n})$. A complete description of this multiplicative structure is given in \cite[Theorem~13.36]{BDDHJ25}. It is induced by the inclusion $\ca(0)\hookrightarrow\ca(1)$ together with the multiplicative structure on the $E_2$-page of the Adams spectral sequence converging to $H\z_*(BC_{2^n})$, which is determined by the Pontryagin product.

\begin{remark}\label{remark: v1 indecomposable}
Observe that $v_1 \notin \Ext_{\ca(1)}^{**}(\f_2)$ which implies that the classes $v_1x_{2^k-2}$ are indecomposable.  However, multiplication by $v_1^2$ is well-defined on the $E_2$-page. 
Since $a= h_0 v_1^2$ and all the classes, except $h_1$-multiple are non-$h_0$-torsion, multiplication by $a$ gives $v_1^2$- multiplication on all these classes on the $E_2$-page. In particular, we have
\(
a\cdot x= v_1^2 \cdot h_0 x.
\)
\end{remark}

\begin{remark}\label{remark: Multi on E_2page}
We record the following multiplicative relations on the $E_2$-page relevant for our computations. The $E_2$-page is generated as a module over $\Ext_{\ca(1)}^{**}(\f_2)$ by the classes $x_1$, $v_1^j x_{4k+2}$ and $v_1^j x_{4k+3}$ for $k\ge 0$ and $j\ge 0$. The multiplicative relations are given by
\[
ax_i=h_0 \cdot v_1^2x_i,\quad
x_i^2=0,\quad
h_1 \cdot x_i=0 \text{ for } i>1,\quad
v_1^jx_{4k+3}=x_1\cdot v_1^j x_{4k+2} \text{ for } j\geq 0,
\]
\[
x_{4k-2} \cdot v_1x_{4l-2}
=
\binom{2k+2l-1}{2k}h_0x_{4(k+l)-2} \pmod 2,
\] and 
\[
v_1x_{4k-2} \cdot v_1x_{4l-2}
=
\binom{k+l}{k}h_0 \cdot v_1x_{4(k+l)-2}\pmod 2 \quad \text{for } l \geq 0.
\]

All remaining products follow from $h_0,a$, and $b$-linearity. For example, we have \[ v_1^2 \cdot v_1^j x_{4k+2} = v_1^{j+2} x_{4k+2} \text{ and }
b \cdot v_1^j x_{4k+2} = v_1^{j+4} x_{4k+2}.\]
\end{remark}

\begin{proposition}\label{prop:diff in ko}

The nontrivial differentials in the Adams spectral sequence converging to
\(ko_*(BC_{2^n})\) are as follows: for arbitrary \(n\geq 2\), through stem \(14\),
we have
\[
\begin{aligned}
(1)\quad &d_n(x_2)=h_0^n x_1, &
(2)\quad &d_n(v_1x_2)=h_0^{n+1}x_3,\\
(3)\quad &d_n(x_6)=h_0^{n-1}v_1x_3, &
(4)\quad &d_n(v_1x_6)=h_0^{n+1}x_7+h_0^{n-1}v_1^2x_3,\\
(5)\quad &d_n(x_{10})=h_0^{n-1}v_1x_7, &
(6)\quad &d_n(v_1x_{10})=h_0^{n+1}x_{11}+h_0^{n-1}v_1^2x_7,\\
(7)\quad &d_n(x_{14})=h_0^{n-1}v_1x_{11}.
\end{aligned}
\]
In the case \(n=2\), the pattern continues beyond stem \(14\), and for \(k\geq 4\), we 
have
\[
\begin{aligned}
(8)\quad & d_2(x_{4k+2})=h_0v_1x_{4k-1},\qquad\qquad
(9)\quad d_2(v_1^2x_{4k+2})=h_0v_1^3x_{4k-1},\\ 
(10)\quad & d_2(v_1x_{4k+2})=h_0^3x_{4k+3}+h_0v_1^2x_{4k-1}.
\end{aligned}
\]
These differentials are $v_1^2$-linear.
\end{proposition}

\begin{remark}
We would like to point out that the differentials stated in \cite[Theorem~13.38]{BDDHJ25} require correction. In particular, the argument based on the May--Milgram theorem alone is not sufficient to determine the differentials in the higher stems. We provide here a more detailed analysis of these differentials through stem \(14\). We are grateful to the authors of \cite{BDDHJ25} for their openness in acknowledging this issue and for their helpful discussions in completing the corrected differentials.
\end{remark}

\begin{proof}
\text{(1) and (2).}
The first two differentials are forced by comparison with the Atiyah--Hirzebruch spectral sequence for stems \(1\)--\(4\).

\medskip

\text{(3).}
Using the multiplicative relations from \cref{remark: Multi on E_2page}
 and the Leibniz rule, we obtain
\[
\begin{aligned}
d_n(h_0x_6)
=d_n(x_2\cdot v_1x_2)
=h_0^nx_1\cdot v_1x_2+x_2\cdot h_0^{n+1}x_3
=h_0^n \cdot v_1x_3.
\end{aligned}
\]
Since the left-hand side is \(h_0d_n(x_6)\) and multiplication by
\(h_0\) is injective on the relevant target, it follows that
\[
d_n(x_6)=h_0^{n-1}\cdot v_1x_3.
\]

\medskip

\text{(4).}
Comparison with the Atiyah--Hirzebruch spectral sequence shows that the differential on \(v_1x_6\) must truncate the \(x_7\)-tower. Moreover, the computation of \(ko_7(BC_{2^n})\) in \cref{app:eta-invariant} shows that there is an element of order \(2^{n+3}\) giving a summand of order at least \(2^{n+3}\). If \(d_n(v_1x_6)\) did not contain the term $h_0^{n-1}v_1^2x_3$, then the order of the summand would be less than $2^{n+3}$, even after resolving hidden extensions, which is a contradiction.
Hence,
\[
d_n(v_1x_6)=h_0^{n+1}x_7+h_0^{n-1}v_1^2x_3.
\]

\medskip

\text{(5).}
Similarly, using the relation
\(
h_0x_{10}=x_2\cdot v_1x_6
\)
and the fact that \(x_2\cdot x_3=0\) and \(x_2\cdot x_7=0\), we obtain
\[
\begin{aligned}
d_n(h_0x_{10})
=d_n(x_2\cdot v_1x_6)
=h_0^nx_1\cdot v_1x_6
+x_2\cdot(h_0^{n+1}x_7+h_0^{n-1}v_1^2x_3)
=h_0^n\cdot v_1x_7.
\end{aligned}
\]
Therefore,
\[
d_n(x_{10})=h_0^{n-1}v_1x_7.
\]

\medskip

\text{(6).}
For \(v_1x_{10}\), using
\(
h_0v_1x_{10}=v_1x_2\cdot v_1x_6,
\)
and applying the Leibniz rule, we obtain
\[
\begin{aligned}
d_n(h_0v_1x_{10})
=d_n(v_1x_2\cdot v_1x_6)
=h_0^{n+2}x_{11}+h_0^nv_1^2x_7.
\end{aligned}
\]
Again, multiplication by \(h_0\) is injective on the relevant target, so it follows that
\[
d_n(v_1x_{10})
=h_0^{n+1}x_{11}+h_0^{n-1}v_1^2x_7.
\]

\medskip

\text{(7).}
Finally, using
\(
h_0x_{14}=x_2\cdot v_1x_{10},
\)
we obtain
\[
\begin{aligned}
d_n(h_0x_{14})
=d_n(x_2\cdot v_1x_{10})
=h_0^nx_1\cdot v_1x_{10}
+x_2\cdot(h_0^{n+1}x_{11}+h_0^{n-1}v_1^2x_7)
=h_0^nv_1x_{11},
\end{aligned}
\]
and hence
\[
d_n(x_{14})=h_0^{n-1}v_1x_{11}.
\]
(8), (9), and (10). For \(n=2\), the remaining differentials follow from the
computation of \(ko_{*}(BC_{4})\) by Bruner and Greenlees.

Bruner and Greenlees compute the groups \(ko_{*}(BC_{4})\) in
\cite[Example~7.3.3]{BG03}, obtaining
\[
ko_{k}(BC_{4}) \cong
\begin{cases}
	\mathbb{Z}/2^{2m+1}\oplus \mathbb{Z}/2, &
	k=8m+1,\quad m\geq 0, \\[2pt]
	\mathbb{Z}/2, &
	k=8m+2,\quad m\geq 0, \\[2pt]
	\mathbb{Z}/2^{4m+3}\oplus \mathbb{Z}/2^{2m+1}, &
	k=8m+3,\quad m\geq 0, \\[2pt]
	\mathbb{Z}/2^{2m+2}, &
	k=8m+5,\quad m\geq 0, \\[2pt]
	\mathbb{Z}/2^{4m+5}\oplus \mathbb{Z}/2^{2m+1}, &
	k=8m+7,\quad m\geq 0, \\[2pt]
	0, & \text{otherwise}.
\end{cases}
\]

Comparing these groups with the \(E_{2}\)-page shown in
\cref{fig:koAH}, the above \(d_{2}\)-differentials are the unique
differentials compatible with the module structure over the \(E_\infty\)-page of \(ko\). For example, the differential \(d_2(v_1x_{10})\) must contain the term \(h_0^3x_{11}\) to account for the \(\mathbb{Z}/2^7\) summand in \(ko_{11}(BC_{4})\). The remaining terms are determined by the multiplicative structure on the \(E_2\)-page and the requirement that the resulting \(E_\infty\)-page matches the known homology groups up to hidden extensions.

In particular, the differentials are linear over \(h_0\) and \(a\). Since
\(
ax=h_0v_1^2x,
\)
we have
\[
\begin{aligned}
h_0d_n(v_1^2x)
&=d_n(h_0v_1^2x)
=d_n(ax)
=a\,d_n(x)
=h_0v_1^2d_n(x).
\end{aligned}
\]
By \(h_0\)-injectivity in the relevant degrees, it follows that
\[
d_n(v_1^2x)=v_1^2d_n(x).
\]
Therefore, all the displayed differentials are \(v_1^2\)-linear.
\end{proof}

\begin{figure}[H] \vspace{-0.5cm}
	\includegraphics[scale=0.42]{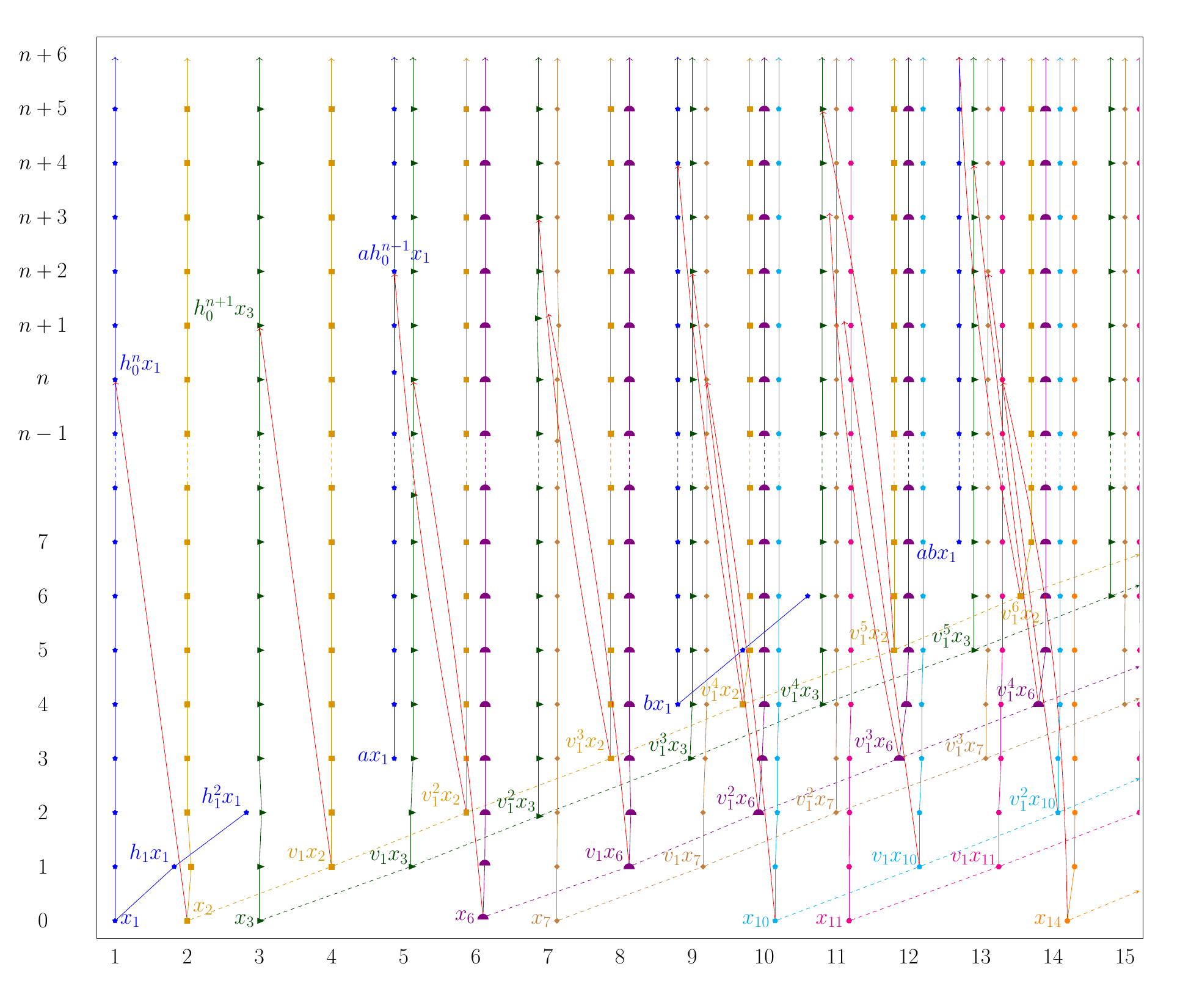}
	\caption{The \(E_n\)-page, with \(d_n\)-differentials (in red), of the Adams spectral sequence converging to 
	\(ko_k BC_{2^n}\). Each bullet represents \(\f_2\), and the vertical lines 
	represent multiplication by \(h_0\). Note that \(v_1\) is not an element of the \(E_n\)-page. However, \(H^*(BC_{2^n})\) has a submodule with a \(\ce(1)\)-module structure. The differentials shown between the classes indicate that the target is the sum of the two classes.}
	\label{fig:diff ko}
\end{figure}

There are no further nontrivial differentials through stem \(14\) for arbitrary
\(n\geq 2\), while for \(n=2\) the above list contains all nontrivial
differentials. Therefore, the spectral sequence collapses as shown in
\cref{fig:ko}.

\begin{remark}
Observe that it is enough to determine the differentials on the indecomposable classes \(v_1x_{2^k-2}\) for \(k > 1\). The multiplicative structure on the $E_2$-page then determines all remaining differentials. The preceding proposition determines these differentials for \(k < 4\), allowing us to resolve the Adams spectral sequence completely up to stem \(14\). Since the differentials on the classes \(v_1x_{2^k-2}\) for \(k > 3\) remain undetermined, we restrict our computations to this range.
\end{remark}

\begin{figure}[H] \vspace{-0.5cm}
   \includegraphics[scale=0.43]{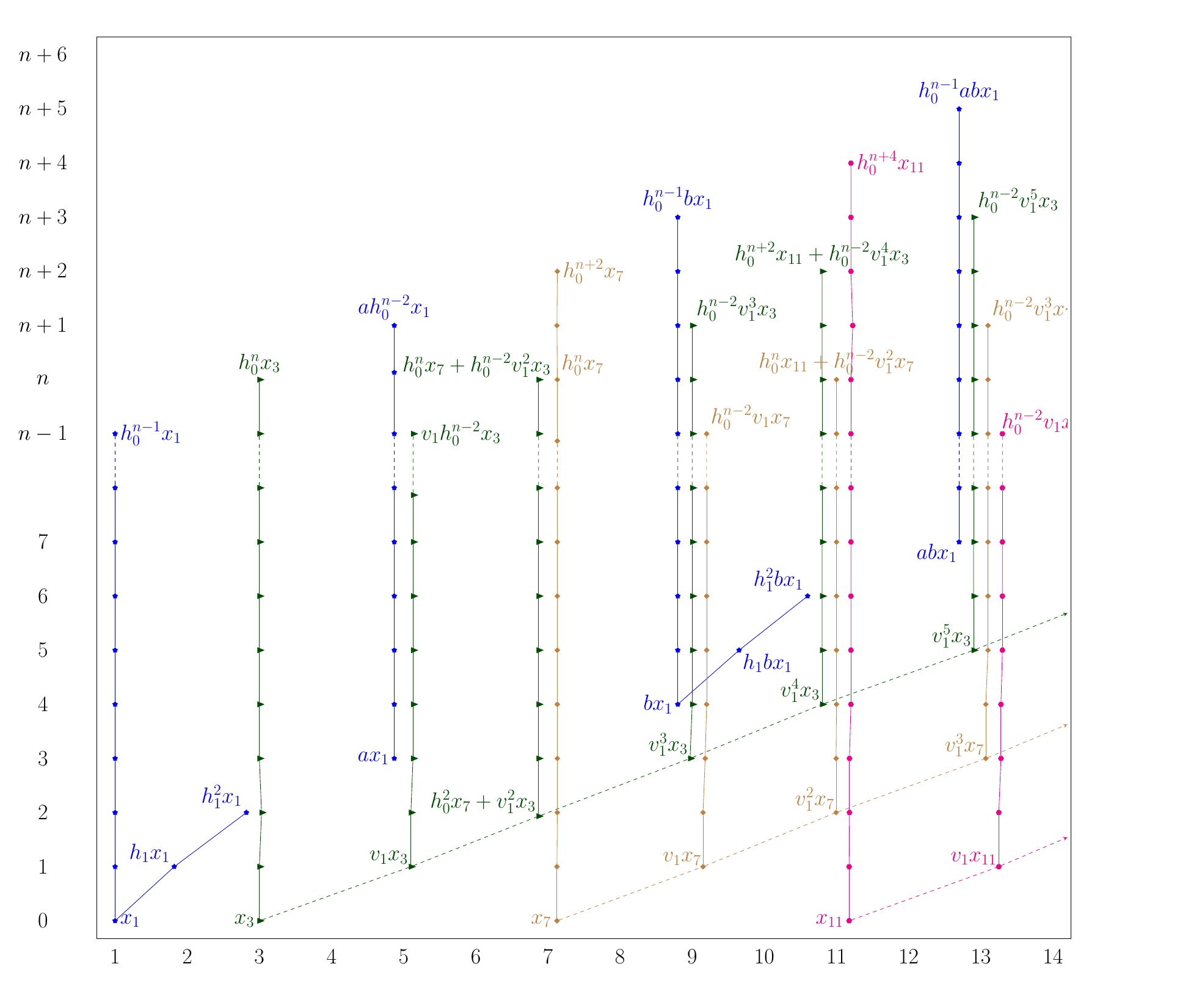}
	\vspace{-0.5cm} \caption{The $E_{\infty}$-page of the Adams spectral sequence converging to 
		$ko_k BC_{2^n}$. Each bullet represents $\f_2$ and vertical lines 
		represent $h_0$-multiplication.}
	\label{fig:ko}
\end{figure}

\subsection{Connective complex K-theory}\label{section: connective complex K-theory}
Finally, we compute the $ku$-homology of $BC_{2^n}$. 
We use the $ko$-homology computation from the preceding subsection, together with the natural map $ko\to ku$, to determine the differentials in the Adams spectral sequence converging to $ku_*(BC_{2^n})$.
We begin by computing the associated graded of \(ku_*(BC_{2^n})\) with respect to the Atiyah--Hirzebruch filtration.
\begin{proposition}
The associated graded of \(ku_*(BC_{2^n})\) is given by
\[
\operatorname{gr}^{*}_{AH}ku_k(BC_{2^n})
\cong
\begin{cases}
(\z/2^n)^{\oplus(j+1)}, & k=2j+1,\quad j\geq 0,\\
0, & \text{otherwise}.
\end{cases}
\]
\end{proposition}

\begin{proof}
Consider the Atiyah--Hirzebruch spectral sequence from \cref{subsec:AHSS}:
\[
E_1
=
\bigoplus_{\substack{\text{cells of dimension }i\\ \text{in }BC_{2^n}}}
\Sigma^i ku_*
\;\Longrightarrow\;
\operatorname{gr}^{*}_{AH}ku_*(BC_{2^n}),
\]
 where $ku_{2j} = \z$ for $j\geq 0$ and 0 otherwise. 
The only nontrivial \(d_1\)-differentials are
\[
d_1(x[2i]) = 2^n x[2i-1], \quad i \ge 1,
\]
where \(x[i]\) denotes an element \(x \in H\z_*\) in filtration \(i\).
\begin{figure}[H]
	\includegraphics[width=0.87\linewidth]{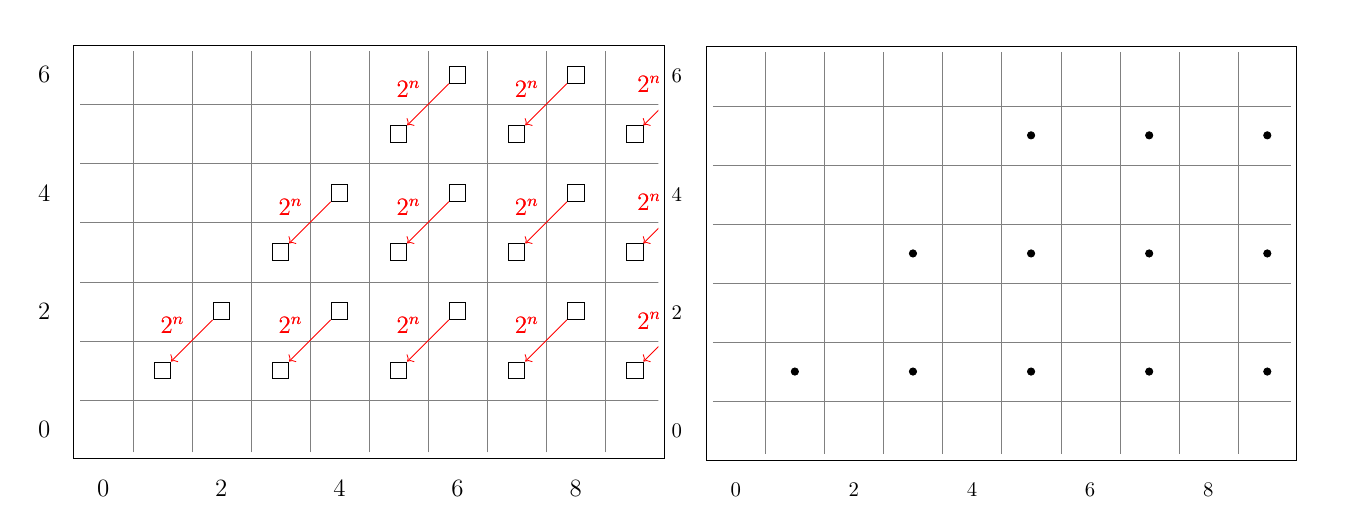}
	 \caption{The $E_1$-page (left) and the $E_\infty$-page (right) of the Atiyah--Hirzebruch spectral sequence converging to $\operatorname{gr}^{*}_{AH}ku_*(BC_{2^n})$. Each $\square$ represents a copy of $\z$ and each $\bullet$ represents a copy of $\z/2^n$.}\label{fig:kuAH}
\end{figure}
The spectral sequence collapses at \(E_2\) due to degree reasons.
\end{proof}

To determine the \(ku_*\)-module structure and resolve the extensions, we turn to the Adams spectral sequence
\[
E_2
=
\Ext_{\ca}^{**}\Big(H^*(BC_{2^n}\wedge ku)\Big)
\overset{\substack{\textit{K{\"u}nneth}}}{\cong}
\Ext_{\ca}^{**}\Big(H^*(BC_{2^n})\otimes H^*(ku)\Big)
\Rightarrow ku_*(BC_{2^n}).
\]
Recall from \cref{prop: classical description} that, as \(\ca\)-modules,
\[
H^*(ku)\cong \ca\ddash\ce(1).
\]
Applying the change-of-rings isomorphism, we may rewrite the \(E_2\)-term as
\[
E_2\cong \Ext_{\ce(1)}^{**}(H^*(BC_{2^n})).
\]
Since both \(\Sq^1\) and \(Q_1\) act trivially on \(H^*(BC_{2^n})\), we obtain
\[
E_2
\cong
\bigoplus_{k\geq 1}\Ext_{\ce(1)}^{**}(\f_2)\{\tilde{x}_k\}
\cong
\bigoplus_{k\geq 1}\f_2[h_0,v_1]\{\tilde{x}_k\}.
\]

The differentials in this spectral sequence are determined by those in the Adams spectral sequence for \(ko_*(BC_{2^n})\). We therefore begin by describing the map of Adams spectral sequences induced by the inclusion
\(
\ce(1)\xhookrightarrow{j}\ca(1).
\)

\begin{proposition}\label{prop: A1-to-E1-restriction}
On the \(E_2\)-pages, the induced map by the inclusion \(\ce(1)\xhookrightarrow{j}\ca(1)\),
\[
\Ext_{\ca(1)}^{**}(H^*(BC_{2^n}))
\xlongrightarrow{j^*}
\Ext_{\ce(1)}^{**}(H^*(BC_{2^n})),
\]
sends the indicated classes for \(k\geq 0\) and \(i=1,4k+2\), or \(4k+3\), as follows:

\[
\renewcommand{\arraystretch}{1.35}
\begin{array}{c|c}
y\in\Ext_{\ca(1)}^{**}(M)
&
j^*(y)\in\Ext_{\ce(1)}^{**}(M)
\\ \hline
x_i & \tilde{x}_i 
\\
h_1x_1 & 0 
\\
ax_1 & h_0v_1^2\tilde{x}_1
\\
bx_1 & v_1^4\tilde{x}_1
\\
v_1x_{4k+2}
&
h_0\tilde{x}_{4k+4}+v_1\tilde{x}_{4k+2}
\\
v_1x_{4k+3}
&
h_0\tilde{x}_{4k+5}+v_1\tilde{x}_{4k+3}
\end{array}
\]

More generally, for \(k,j\geq 0\) and
\(\varepsilon\in\{0,1\}\), the map sends

\[
\renewcommand{\arraystretch}{1.35}
\begin{array}{c|c}
\hspace{2cm}v_1^{2j+\varepsilon}x_{4k+2}
&
v_1^{2j+\varepsilon}\tilde{x}_{4k+2}
+\varepsilon h_0v_1^{2j}\tilde{x}_{4k+4}
\\
\hspace{2cm}v_1^{2j+\varepsilon}x_{4k+3}
&
v_1^{2j+\varepsilon}\tilde{x}_{4k+3}
+\varepsilon h_0v_1^{2j}\tilde{x}_{4k+5}.
\end{array}
\]

\end{proposition}

\begin{proof}
Consider the two chains of inclusions
\[
\ca(0) \xhookrightarrow{i} \ce(1)\xhookrightarrow{j}\ca(1)
\qquad\text{and}\qquad
\mathbb E(Q_1)\xhookrightarrow{k}\ce(1)\xhookrightarrow{j}\ca(1),
\]
where $\e(Q_1)$ is the exterior algebra generated by $Q_1$. The induced maps on Ext allow us to compare a class over \(\ce(1)\) with
its images over \(\ca(0)\) and \(\mathbb E(Q_1)\). To distinguish the different copies of the same class, we write \(x_i^{\ca(0)}\) , \(x_i^{\ce(1)}\) and \(x_i^{\mathbb E(Q_1)}\) for the corresponding classes over \(\ca(0)\), \(\ce(1)\) and \(\mathbb E(Q_1)\), respectively.

The images of the map induced by the inclusion
\(\ca(0)\xhookrightarrow{j\circ i}\ca(1)\) are given in
\cite[Theorem~13.30]{BDDHJ25}. The induced  
 map $i^*$ from \(\ce(1)\)-Ext to \(\ca(0)\)-Ext sends
\[
h_0^{\ce(1)}\longmapsto h_0^{\ca(0)}\qquad\text{and}\qquad v_1^{\ce(1)}\longmapsto 0.
\]
The map $k^*$ from \(\ce(1)\)-Ext to
\(\mathbb E(Q_1)\)-Ext sends
\[
h_0^{\ce(1)}\longmapsto 0\qquad\text{and}\qquad v_1^{\ce(1)}\longmapsto v_1^{E(Q_1)},
\]
while the map $(k^*\circ j^*)$ from \(\ca(1)\)-Ext to \(\mathbb E(Q_1)\)-Ext sends
\[
  \qquad \qquad \qquad h_0^{\ca(1)}\longmapsto 0\qquad\text{and}\qquad v_1x_i^{\ca(1)}\longmapsto v_1^{E(Q_1)}x_i^{E(Q_1)}, \quad \text{ for } i=1,4k+2,4k+3.
\]

These give the relevant images summarized in the following table, where $M$ denotes \(H^*(BC_{2^n})\), \(i=1,4k+2,4k+3\) and \(r=4k+2,4k+3\).
\[
\renewcommand{\arraystretch}{1.35}
\begin{array}{c|c|c|c}
y\in\Ext_{\ca(1)}^{**}(M)
&
j^*(y)\in\Ext_{\ce(1)}^{**}(M)
&
k^*\circ j^*(y)\in\Ext_{\mathbb E(Q_1)}^{**}(M)
&
i^*\circ j^*(y)\in\Ext_{\ca(0)}^{**}(M)
\\ \hline
x_i & x_i^{\ce(1)} & x_i^{\mathbb E(Q_1)} & x_i^{\ca(0)}
\\
h_1x_1 & 0 & 0 & 0
\\
ax_1 & h_0v_1^2x_1^{\ce(1)} & 0 & 0
\\
bx_1 & v_1^4x_1^{\ce(1)} & v_1^4x_1^{\mathbb E(Q_1)} & 0
\\
v_1x_r
&
p\,h_0x_{r+2}^{\ce(1)}+q\,v_1x_r^{\ce(1)}
&
v_1x_r^{\mathbb E(Q_1)}
&
h_0x_{r+2}^{\ca(0)}
\end{array}
\]

It remains to determine \(p,q\in\f_2\) in the last row of the table.
Mapping the proposed image of \(v_1x_r\) from \(\ce(1)\)-Ext to
\(\ca(0)\)-Ext gives \(p\,h_0x_{r+2}^{\ca(0)}\). This must agree with \(i^*\circ j^*(v_1x_r)\), which is \(h_0x_{r+2}^{\ca(0)}\).
Therefore, \(p=1\).

Similarly, mapping to \(\mathbb E(Q_1)\)-Ext gives
\(q\,v_1x_r^{\mathbb E(Q_1)}\). This agrees with \(k^*\circ j^*(v_1x_r)\), which is
\(v_1x_r^{\mathbb E(Q_1)}\), so \(q=1\). Hence
\[
v_1x_r\longmapsto
h_0x_{r+2}^{\ce(1)}+v_1x_r^{\ce(1)}.
\]

Finally, the map is \(a\)-linear, and \(a\) maps to \(h_0v_1^2\).
Combining \(a\)-linearity with \(h_0\)-linearity gives
\(
h_0d(v_1^2x)=h_0v_1^2d(x).
\)
Since multiplication by \(h_0\) is injective on the target, the map is
\(v_1^2\)-linear on these classes. It follows that
\[
v_1^{2j}x_r\longmapsto v_1^{2j}x_r^{\ce(1)},\qquad
v_1^{2j+1}x_r\longmapsto
v_1^{2j+1}x_r^{\ce(1)}+h_0v_1^{2j}x_{r+2}^{\ce(1)}.
\]
Taking \(r=4k+2\) and \(r=4k+3\), and identifying
\(x_i^{\ce(1)}\) with \(\tilde{x}_i\), gives the stated formulas.
\end{proof}

\begin{remark}
Another important point is the multiplicative structure on the
\(E_2\)-page of the Adams spectral sequence converging to
\(ku_*(BC_{2^n})\). This is induced by the multiplicative structure on
the \(E_2\)-page of the Adams spectral sequence converging to
\(H\z_*(BC_{2^n})\), which is determined in
\cite[Proposition~13.18]{BDDHJ25}. The relations are
\[
\tilde{x}_i^2=0,\qquad
\tilde{x}_i\cdot\tilde{x}_j
=
\binom{i+j}{i}\tilde{x}_{i+j}.
\]
\end{remark}

\begin{proposition}
The nontrivial differentials in the Adams spectral sequence converging
to \(ku_*(BC_{2^n})\) are given as follows.
\begin{enumerate}
	\item For \(n=2\),
	\[
	d_2(\tilde{x}_2)=h_0^2\tilde{x}_1 \qquad
	\text{ and } \qquad \text{ for } k>1,\;
	d_2(\tilde{x}_{2k})
	=
	h_0^2\tilde{x}_{2k-1}
	+h_0v_1\tilde{x}_{2k-3}.
	\]

	\item For \(n>2\), through stem \(14\),
	\[
	\hspace{0.6cm}d_n(\tilde{x}_2)=h_0^n\tilde{x}_1,
	\qquad
	\text{ and } \qquad \text{ for } 2\leq k\leq 7,\;
	d_n(\tilde{x}_{2k})
	=
	h_0^n\tilde{x}_{2k-1}
	+h_0^{n-1}v_1\tilde{x}_{2k-3}.
	\]
\end{enumerate}
These differentials are \(h_0\)-linear and \(v_1\)-linear.
\end{proposition}

\begin{proof}
First, for all \(n\), comparison with the Atiyah--Hirzebruch spectral sequence gives
\[
ku_1(BC_{2^n})\cong\z/2^n.
\]
To recover this group, the Adams spectral sequence must support the differential
\[
d_n(\tilde{x}_2)=h_0^n\tilde{x}_1.
\]
Since \(v_1\) is a permanent cycle, \(v_1\)-linearity gives
\[
d_n(v_1\tilde{x}_2)=h_0^nv_1\tilde{x}_1.
\]

Under the map of Adams spectral sequences induced by the inclusion \(\ce(1)\hookrightarrow\ca(1)\) described in \cref{prop: A1-to-E1-restriction}, we have
\[
v_1x_2\longmapsto h_0\tilde{x}_4+v_1\tilde{x}_2.
\]

By \cref{prop:diff in ko},
\(
d_n^{ko}(v_1x_2)=h_0^{n+1}x_3.
\)
Naturality therefore gives
\[
h_0d_n^{ku}(\tilde{x}_4)+v_1d_n^{ku}(\tilde{x}_2)
=h_0^{n+1}\tilde{x}_3.
\]
Since multiplication by \(h_0\) is injective on the relevant target group, it follows that
\[
d_n^{ku}(\tilde{x}_4)
=h_0^n\tilde{x}_3+h_0^{n-1}v_1\tilde{x}_1.
\]

Next, using \(\tilde{x}_6=\tilde{x}_2\tilde{x}_4\) and the Leibniz rule, we obtain
\[
d_n^{ku}(\tilde{x}_6)
=h_0^n\tilde{x}_1\tilde{x}_4
+\tilde{x}_2\Big(h_0^n\tilde{x}_3+h_0^{n-1}v_1\tilde{x}_1\Big)
=h_0^n\tilde{x}_5+h_0^{n-1}v_1\tilde{x}_3.
\]

Similarly,
\(
v_1x_6\longmapsto h_0\tilde{x}_8+v_1\tilde{x}_6,
\) 
and, by \cref{prop:diff in ko},
\[
d_n^{ko}(v_1x_6)
=h_0^{n+1}x_7+h_0^{n-1}v_1^2x_3.
\]
Naturality gives
\(
h_0d_n^{ku}(\tilde{x}_8)+v_1d_n^{ku}(\tilde{x}_6)
=h_0^{n+1}\tilde{x}_7+h_0^{n-1}v_1^2\tilde{x}_3.
\)
\vspace{0.1cm}

Substituting the differential on \(\tilde{x}_6\) and again canceling \(h_0\), we obtain
\[
d_n^{ku}(\tilde{x}_8)
=h_0^n\tilde{x}_7+h_0^{n-1}v_1\tilde{x}_5.
\]
Repeating these two arguments determines the remaining differentials in the stated range. 
\end{proof}
\begin{figure}[H] \vspace{-0.5cm}
	\hspace{-1.15cm}  \includegraphics[scale=0.45]{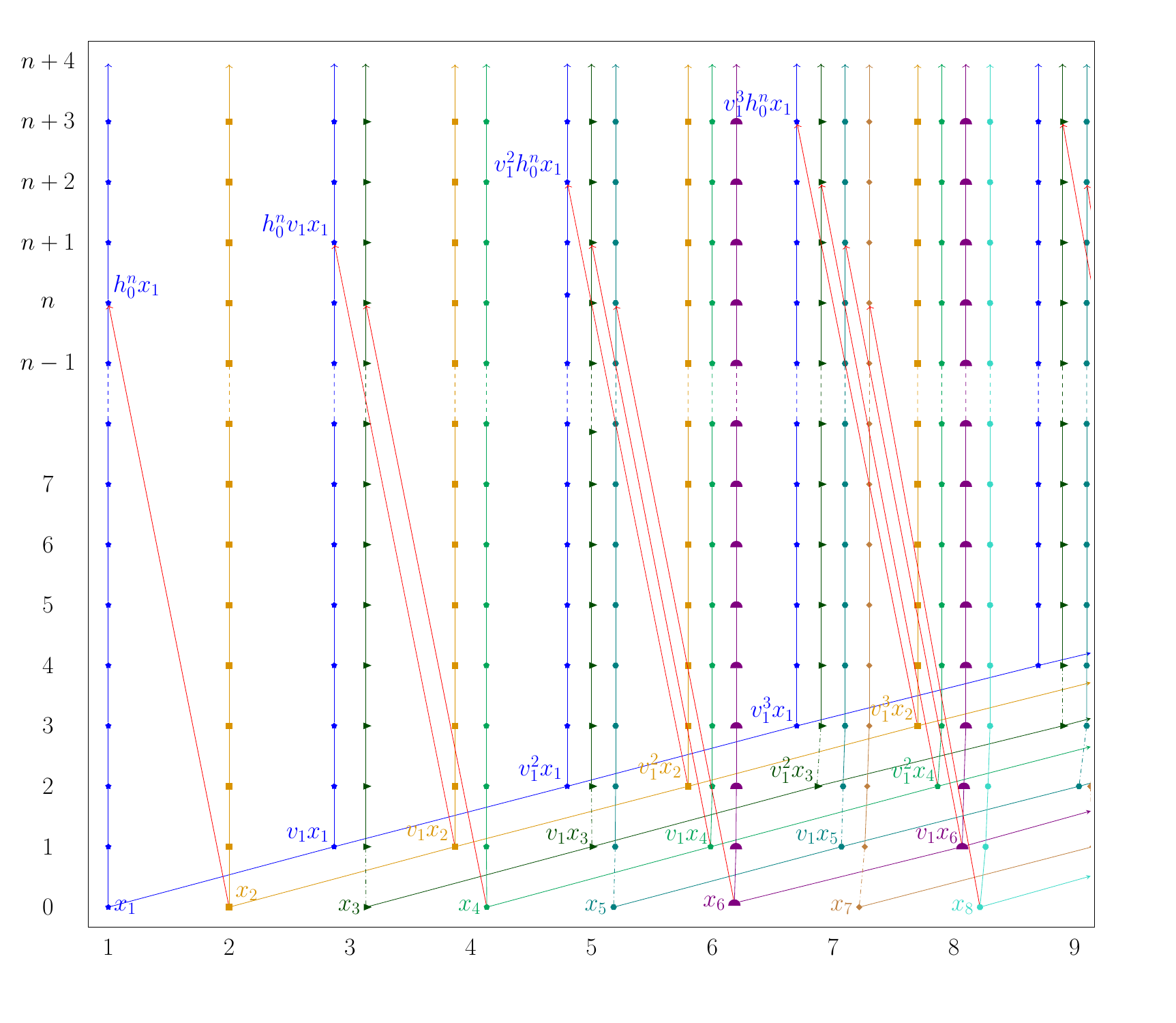}
	\vspace{-0.5cm} \caption{The $E_n$-page, with the $d_n$-differentials shown in red of the Adams spectral sequence converging to $ku_* BC_{2^n}$. Each bullet represents a copy of $\f_2$, and vertical solid lines denote multiplication by $h_0$. Dash--dot--dash lines indicate that a generator $\kappa_i$ maps to $h_0\kappa_i + v_1\kappa_{i-1}$; for example, $x_3 \mapsto v_1x_1 + h_0x_3$. Slanted lines represent multiplication by $v_1$.}
	\label{fig:E2 page of ku}
\end{figure}

\begin{remark}
Again, observe that it is enough to determine the differentials on the indecomposable classes \(\tilde{x}_{2^k}\). 
The multiplicative structure on the \(E_2\)-page, together with $ko$-differentials, then determines all remaining differentials. 
Therefore, the differentials remain undetermined for \(k>3\) and we restrict our computations to this range.
\end{remark}

These differentials propagate by \(h_0\)- and \(v_1\)-linearity, and then the spectral sequence collapses at \(E_{n+1}\) through stem \(14\), since the \(E_{n+1}\)-page is concentrated in odd stems.

There are higher differentials, so the module structure can be read directly from the Adams \(E_\infty\)-term shown in \cref{fig:ku}.

\begin{figure}[H] \vspace{-0.5cm}
	\hspace{-1.15cm}  \includegraphics[scale=0.5]{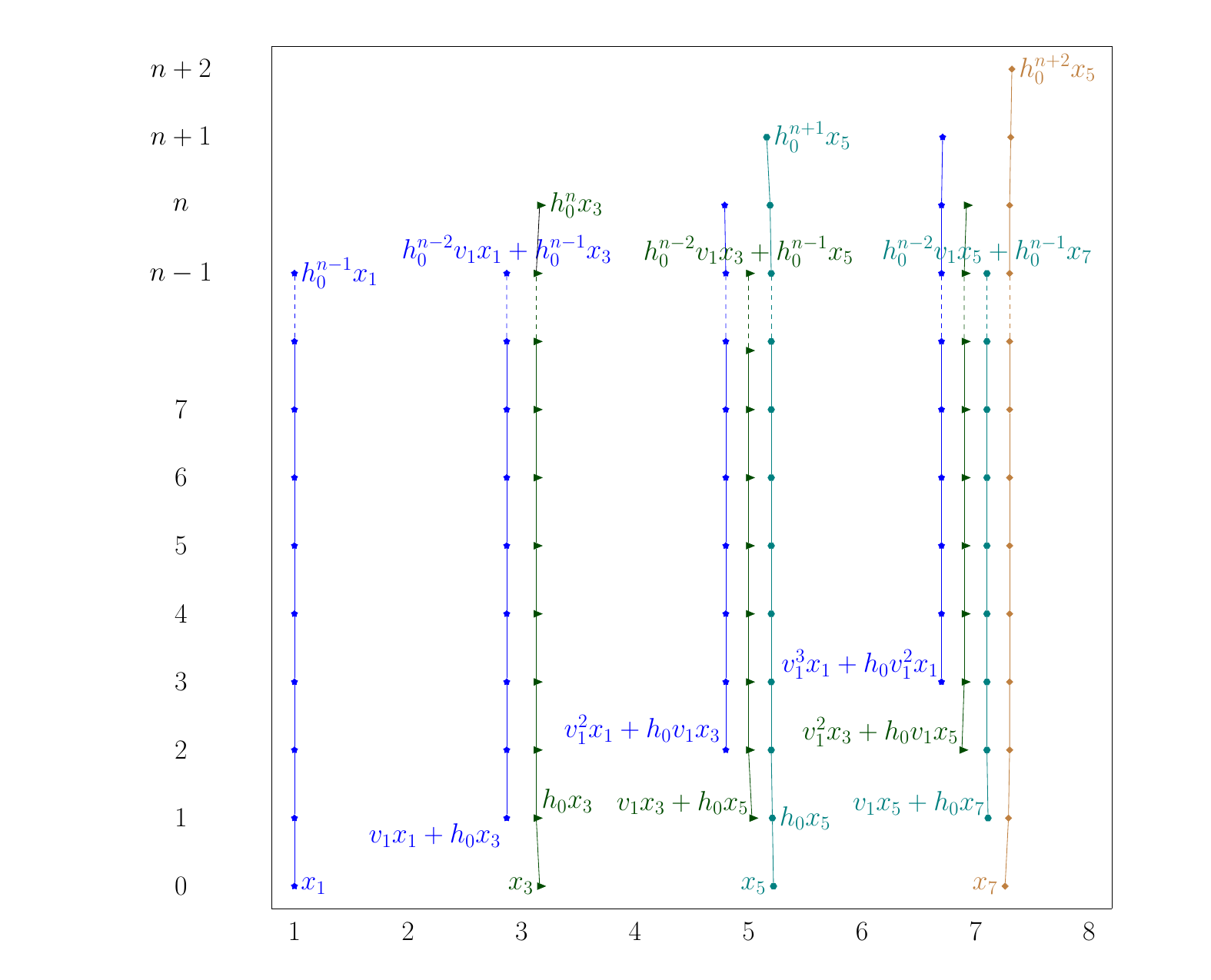}
	\vspace{-0.5cm} \caption{The  $E_{\infty}$-page of the Adams spectral sequence converging to $ku_* BC_{2^n}$. Each bullet represents a copy of $\f_2$, and vertical lines denote multiplication by $h_0$. $v_1$-multiplication is exists, but hidden in the chart.}
	\label{fig:ku}
\end{figure}


\section{Integral Motivic Homology}\label{Integral Homology}

We begin with computing the integral homology \(H\z_{**}(B\mu_{2^n})\) using the known structure of \(H^{**}(B\mu_{2^n})\) as $\ca^F(0)$-module. 

\subsection{Integral motivic homology over algebraically closed fields}
Note that it is enough to compute the integral motivic homology over $\c$ since the Lefschetz principle described in \cref{subsec:extention} allows us to extend the computation to any algebraically closed field of characteristic not equal to 2.
We use the same strategy as in the classical case.
Recall from \cref{prop: motivic description} that, analogously, we have
\[
H^{**}(H\z^{\c}) \cong \ca^{\c} \ddash \ca^{\c}(0),
\]
and from \cref{cor: a0e1} we have
\[
H^{**}B\mu_{2^n} \;\cong\; \bigoplus_{k \geq 1} \;\Sigma^{k,\lfloor \frac{k+1}{2}\rfloor} \m^{\c}_2.
\]

Using these to simplify the $E_2$- page of the motivic Adams spectral sequence \cref{ASS} converging to $H\z^{\c}_{**}B\mu_{2^n}$, we get
\[
E_2
 \cong  \Ext_{\ca^{\c}(0)}^{***}(H^{**}B\mu_{2^n}) \\
 \cong\bigoplus_{k \geq 1} \,\m^{\c}_2[h_0]\{x_k\}
\Rightarrow H\z^{\c}_{**}B\mu_{2^n}
\]
where $|x_k|=(k,\lfloor \tfrac{k+1}{2}\rfloor)$.

The Betti realization functor from motivic spectra over $\c$ to classical spectra discussed in \cref{sec:coho} gives that the differentials coincide with the classical ones:
\[
d_n(x_{2k}) = h_0^{n} x_{2k-1},
\]
and are $h_0$-linear.
This proves the following \cref{thm:c-HZ}.

\begin{thm}\label{thm:c-HZ}
The integral  $\c$-motivic homology groups of $B\mu_{2^n}$ are given by
\[
H\z_{**}^\c B\mu_{2^n} \;\cong\; \bigoplus_{k \geq 1} \;\Sigma^{2k+1,k+1}\, H\z_{**}^\c/2^n[\tau].
\]
\end{thm}

The theorem shows that the Adams chart for $H\z_{**}^\c(B\mu_{2^n})$ coincides with the Adams chart for $H\z_{*}(BC_{2^n})$, illustrated in \cref{fig:HZ}, where each $\bullet$ now represents $\f_2[\tau]$ instead of $\f_2$.

\subsection{Integral motivic homology over real closed fields}

Over $\r$, the presence of the class $\rho$ in $\r$-motivic cohomology introduces additional structure. But the overall strategy parallels the complex and classical calculations. Again, we only need to compute the integral motivic homology over $\r$ since the Lefschetz principle described in \cref{subsec:extention} allows us to extend the computation to any real closed field.  

\begin{remark}  
The $\r$-motivic integral homology of a point is  
\(
H\z^{\r}_{**} \cong \z[\rho,\tau^2]/ h\rho.
\)

For computing this, \cite{Hil11} uses the $\rho$-Bockstein spectral sequence: 
\[
E_1 = \Ext_{\ca(0)^\c}^{***}(\m_2^{\c})[\rho] \cong \f_2[\tau,\rho,h_0] \Rightarrow \Ext_{\ca^{\r}(0)}^{***}(\m_2^{\r}) ,\] 
where $ |\tau|=(0,0,1),\ |\rho|=(0,1,1),\ |h_0|=(0,1,0)$
and differentials $d_1(\tau) = \rho h_0$. The spectral sequence collapses immediately, giving  
\[
\Ext_{\ca^{\r}(0)}^{***}(\m_2^{\r}) \cong \f_2[\tau^2,\rho,h_0]/(\rho h_0).
\]  

Passing to the Adams spectral sequence, we find that it collapses at $E_2$ for degree reasons, yielding 
\(
H\z_{**}^{\r},
\)
as a ring.
\end{remark}  

We now use this as input for the Atiyah--Hirzebruch spectral sequence.
\begin{proposition}
	The associated graded of  $H\z_{**}^{\r}(B\mu_{2^n})$ with respect to the Atiyah--Hirzebruch filtration is 
	
\[
gr^*_{AH} H\z_{k,*}^{\r}B\mu_{2^n} \cong 
\begin{cases}
	\z/ 2^n[\tau^2] \bigoplus_{j \geq k } \f_2[\tau^2] & k >0  \text{ even }, \\
	\bigoplus_{ j \geq k } \f_2[\tau^2] & k > 1 \text{ odd },\\
	\bigoplus_{j \geq 0 } \f_2[\tau^2] & k \leq 1.
\end{cases}
 \] 
\end{proposition}
\begin{proof}
Consider the Atiyah--Hirzebruch spectral sequence (described in \cref{subsec:AHSS}):
\[
E_1 = \bigoplus_{\substack{\text{each cell } i \\ \text{in } B\mu_{2^n}}} \Sigma^{i,\lfloor \frac{i+1}{2}\rfloor} H\z^{\r}_{**} = \bigoplus_{i \geq 1} \Sigma^{i,\lfloor \frac{i+1}{2}\rfloor} H\z^{\r}_{**}  \Rightarrow H\z^{\r}_{**}B\mu_{2^n},
\] 
with nontrivial differentials
\(
d_1(x[2i]) = h^n x[2i-1], \; i\geq1.
\)

The spectral sequence collapses at the $E_2$-page. Since differentials are $\rho$-linear, there are no nontrivial higher differentials for degree reasons.
\end{proof}
\begin{figure}[H] \vspace{-0.5cm}
	\includegraphics[width=.95\linewidth]{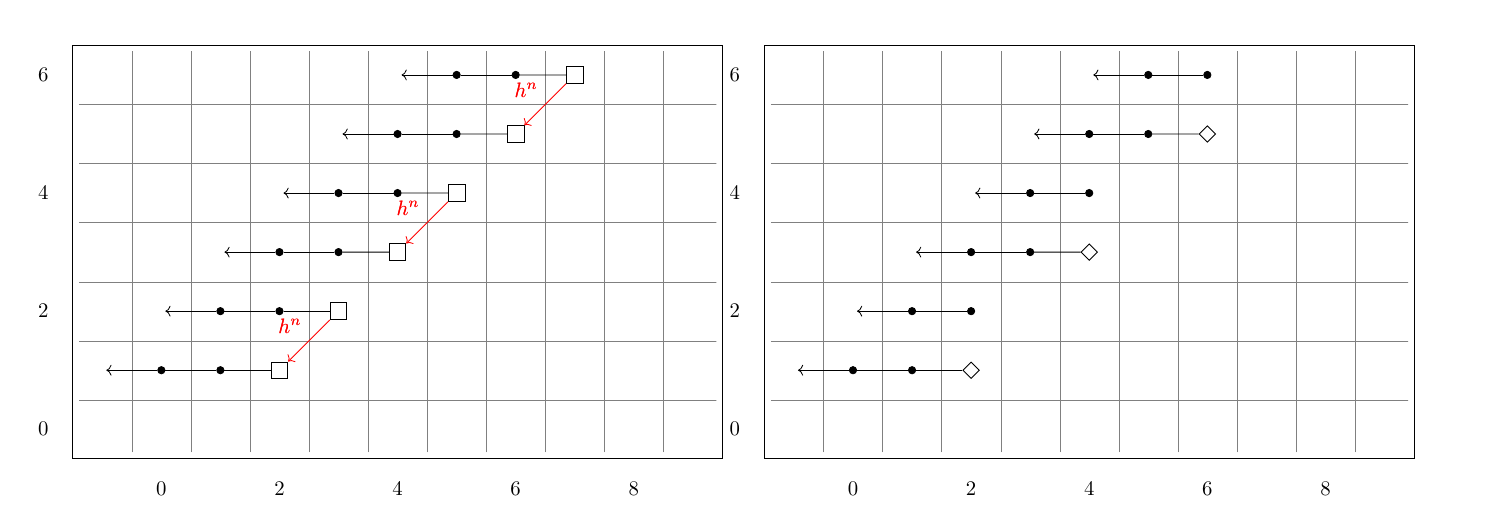}
	\vspace{-0.5cm} \caption{The \(E_1\)-page (left) and \(E_\infty\)-page (right) of the Atiyah--Hirzebruch spectral sequence converging to \(H\z_{**}^\r B\mu_{2^n}\). Each \(\square\) represents \(\z[\tau^2]\), each \(\bullet\) represents \(\f_2[\tau^2]\) , each  $\diamond$  represents $\z/2^n[\tau^2]$ and the red (slanted) arrows indicate multiplication by \(h^n\).}
	\label{fig: HZ AHSS over R}
\end{figure}

To resolve extension problems, we consider the motivic Adams spectral sequence (\cref{ASS}). Again,
\[
H^{**}(H\z^{\r}) \cong \ca^{\r} \ddash \ca^{\r}(0) \quad\text{ and }\quad H^{**}B\mu_{2^n} \;\cong\; \bigoplus_{k \geq 1} \;\Sigma^{k,\lfloor \frac{k+1}{2}\rfloor} \m^{\r}_2
\]
gives
\[
E_2
 \cong  \Ext_{\ca^{\r}(0)}^{***}(H^{**}B\mu_{2^n}) \cong \bigoplus_{k \geq 1} \, \Ext_{\ca^{\r}(0)}^{***}(\Sigma^{(k,\lfloor \tfrac{k+1}{2}\rfloor)}\m_2^\r)
 \cong\bigoplus_{k \geq 1} \, \Sigma^{(k,\lfloor \tfrac{k+1}{2}\rfloor)} \frac{\f_2[h_0,\tau^2,\rho]}{\rho h_0}\{x_k\},
\]
where $|x_k|=(k,\lfloor \tfrac{k+1}{2}\rfloor)$. Comparing to the Atiyah--Hirzebruch spectral sequence, we get

\[
d_n( x_{2k}) = h_0^{n} x_{2k-1}.
\]

Also, the differentials are $\tau^2$-linear as $\tau^2$ is an element of $H\z_{**}^\r$. The spectral sequence collapses for degree
reasons, depicted in \cref{fig:HZ over R}. This proves the following \cref{thm:r-HZ}.

\begin{figure}[H]\vspace{-0.5cm}
	\centering
	\includegraphics[scale=0.52]{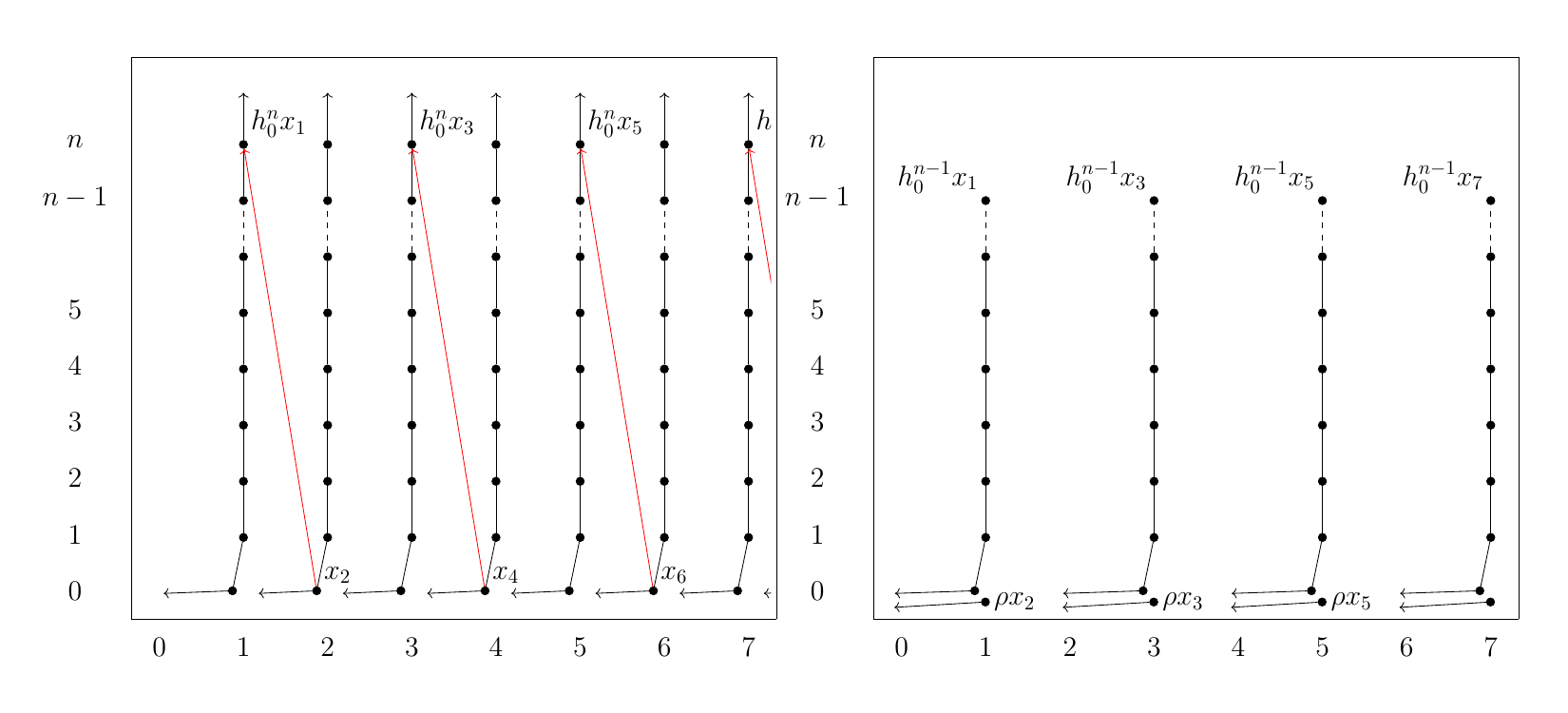}
	\vspace{-0.6cm}
	\caption{The $E_n$-page (left) and $E_{n+1} = E_{\infty}$-page (right) of the Adams spectral sequence converging to $H\z_{*}^\r BC_{2^n}$. Each $\bullet$ = $\f_2[\tau^2]$, and the $d_n$-differentials are drawn in red.}
	\label{fig:HZ over R}
\end{figure}

\begin{thm}\label{thm:r-HZ}
The \(2\)-completed integral homology of \(B\mu_{2^n}\) is given by
    \[
    H\z_{**}^\r(B\mu_{2^n});
    \cong
    \bigoplus_{k\geq 1}
    \Sigma^{2k-1,k} H\z_{**}^\r/h^n \oplus \bigoplus_{k \geq 1} \; \Sigma^{2k,k} \;\rho\; \z/2[\rho, \tau^2].
    \]
\end{thm}


\section{Effective algebraic K-theory}\label{kgl}

As discussed in \cref{theories}, in motivic homotopy theory, we have the notion of “connective” cover of the algebraic K-theory spectrum $KGL$, analogous to the classical complex K-theory spectrum $KU$.

We now analyze the effective algebraic K-theory of $B\mu_{2^n}$, proceeding via the Atiyah--Hirzebruch and Adams spectral sequences, and using the fact  from \cref{prop: motivic description} that
\[
H\f_2^{**}(kgl^{\textit{F}}) \;\cong\; \mathcal{A^{\textit{F}}} \ddash \mathcal{E^{\textit{F}}}(1).
\]

Again, calculations are done only over $\c$ and $\r$ as the Lefschetz principle extends to the computation over any algebraically closed field of characteristic not equal to 2 and any real closed field, respectively.

\subsection{Effective algebraic K-theory over algebraically closed fields} 

\begin{proposition}
The associated graded of \(kgl_{**}^{\c} (B\mu_{2^n})\) with respect to the Atiyah--Hirzebruch filtration is
\[gr^*_{AH} kgl_{k,w}^{\c} (B\mu_{2^n})  \cong 
\begin{cases}
	\Big(\z/2^n[\tau]\Big)^{\oplus(i+1)} & \text{if } k = 2i+1 > 0, \\
	0 & \text{otherwise.}
\end{cases}
\]
\end{proposition}
\begin{proof} Consider the Atiyah--Hirzebruch spectral sequence (\cref{subsec:AHSS}):
\[
E_1 = \bigoplus_{\substack{\text{each cell } i \\ \text{in } B\mu_{2^n}}} \Sigma^{i,\lfloor \frac{i+1}{2}\rfloor} kgl_{**}^{\c}  \Rightarrow gr^*_{AH} kgl_{**}^{\c} (B\mu_{2^n}) ,
\]
where $kgl_{n,*}^{\c} \cong \z_2[\tau] \text{ for } n=2k \geq 0$  and  0 otherwise, computed using the motivic Adams spectral sequence
with $\Ext$ generators $|h_0| = (0,1,0)$, $|v_1| = (2,1,1)$. The nontrivial $d_1$-differentials are
\[
d_1(x[2i]) = 2^n\cdot x[2i-1], \quad i \geq 1.
\]
For degree reasons, this spectral sequence collapses at $E_2$, similar to classical case.
\end{proof}

\begin{thm}\label{thm:c-kgl}
The \(E_{\infty}\)-term of the Adams spectral sequence converging to \(kgl_{**}^\c B\mu_{2^n}\) as a module over \(E_\infty\)-page of the Adams spectral sequence converging to \(kgl_{**}^\c\) is given by
\begin{enumerate}
    \item For arbitrary \(n\), through stem \(14\),
    \[
   E_\infty\Bigl(kgl_{k, *}^{\c} B\mu_{2^n}\Bigr)
    \cong
\begin{cases}
	\z/2^{n+i}[\tau]\bigoplus \Big(\z/2^{n-1}[\tau]\Big)^{\oplus i}  & \text{if } 0<k = 2i+1 \leq 14, \\
	0 & \text{otherwise.}
\end{cases}
    \]

    \item For \(n=2\),
    \[
   E_\infty\Bigl(kgl_{k, *}^{\c} B\mu_{4}\Bigr)
    \cong
\begin{cases}
	\z/2^{2+i}[\tau]\bigoplus \Big(\z/2[\tau]\Big)^{\oplus i}  & \text{if } 0<k = 2i+1 , \\
	0 & \text{otherwise.}
\end{cases}
    \]
\end{enumerate}
\end{thm}

\begin{proof}
Recall that, as in the classical case, \cref{prop: motivic description} gives an isomorphism
 \[
H^{**}(kgl^{\c}) \cong \ca^{\c}\ddash\ce^{\c}(1),
\] 
and from \cref{cor: a0e1} we have
\[
H^{**}B\mu_{2^n} \;\cong\; \bigoplus_{k \geq 0} \;\Sigma^{k,\lfloor \frac{k+1}{2}\rfloor} \m^{\c}_2.
\]
Using these together with the change-of-rings isomorphism to simplify the \(E_2\)-page of the motivic Adams spectral sequence (\cref{ASS}), we obtain
\[
E_2^{**} \cong \bigoplus_{k \geq 1} \Sigma^k \Ext_{\ce^{\c}(1)}^{***}(\m_2)  \cong \bigoplus_{k \geq 1} \m_2^{\c}[h_0,v_1]\{\tilde{x}_k\} \Rightarrow kgl_{**}^{\c}(B\mu_{2^n}),
\]
 with $|\tilde{x}_k|=(k,\lfloor \tfrac{k+1}{2}\rfloor), \ |h_0|=(1,1,1), \ |v_1|=(3,1,1)$. Using the Betti realization functor and the fact that the differentials are $\tau$-linear, we get differentials analogous to those in the classical case.

 For arbitrary \(n\), through stem \(14\),
	\[
	\hspace{0.6cm}d_n(\tilde{x}_2)=h_0^n\tilde{x}_1,
	\qquad
	\text{ and } \qquad \text{ for } 2\leq k\leq 7,\;
	d_n(\tilde{x}_{2k})
	=
	h_0^n\tilde{x}_{2k-1}
	+h_0^{n-1}v_1\tilde{x}_{2k-3}.
	\]

Propagation by \(h_0\)- and \(v_1\)-linearity determines all remaining differentials through stem \(14\), so the spectral sequence collapses at \(E_{n+1}\).  

These differentials continue beyond stem \(14\) for \(n=2\), giving the second part of the theorem,
	\[
	d_2(\tilde{x}_2)=h_0^2\tilde{x}_1 \qquad
	\text{ and } \qquad \text{ for } k>1,\;
	d_2(\tilde{x}_{2k})
	=
	h_0^2\tilde{x}_{2k-1}
	+h_0v_1\tilde{x}_{2k-3}.
	\]

The Adams chart for \(\c\)-motivic \(kgl^{\c}_{**}(B\mu_{2^n})\) coincides with the classical Adams chart for \(ku_{*}(BC_{2^n})\), shown in \cref{fig:ku}, where each \(\bullet\) now represents \(\f_2[\tau]\) instead of \(\f_2\). 
\end{proof}

\subsection{Effective algebraic K-theory over real closed fields}

Next, we compute $kgl_{**}^{\r}B\mu_{2^n}$. 

As shown in \cite{Hil11}, we have
\begin{equation} \label{eqn: result for Ext_E1(M2) in R}
	\Ext_{\ce^{\r}(1)}^{***}(\m_2^{\r}) \cong \dfrac{\f_2[\tau^4,\rho,h_0,v_1,\tau^2 h_0]}{(\rho h_0, \rho^3 v_1, \rho \cdot [\tau^2 h_0])}, \text{\quad as a ring}.
\end{equation}

The $\r$-motivic Adams spectral sequence collapses, as there is no room for nontrivial differentials from the multiplicative generators (see \cite[Fig.~1]{Hil11}). 

\begin{figure}[H] \vspace{-1cm}
	\includegraphics[scale=0.65]{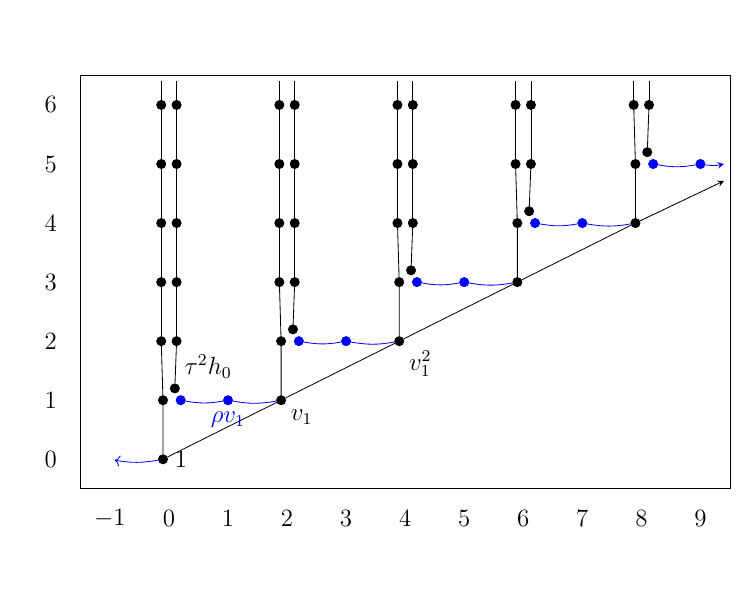}
	\vspace{-1cm} \caption{The $E_2 = E_\infty$-page of the motivic Adams spectral sequence converging to $kgl_{**}^{\r}$ ; each $\bullet$ represents $\f_2[\tau^4]$. Note that every class shown which is not $\rho^3$-divisible supports $v_1$-multiplication, although this is hidden in the chart. } 
	\label{Fig. 1} 
\end{figure}
Consider the Atiyah--Hirzebruch spectral sequence
\[
E_1 = \bigoplus_{\substack{\text{each cell } i \\ \text{in } B\mu_{2^n}}}  \Sigma^{i,\lfloor \frac{i+1}{2}\rfloor} kgl_{**}^{\r} \;\;\Rightarrow\;\; kgl_{**}^{\r}B\mu_{2^n}, 
\]
where \(x[i]\) denotes an element \(x \in kgl_{**}^{\r}\).
\begin{proposition}
The differentials are given by:
\begin{enumerate}
    \item \(d_1(x[2i]) = h^n \cdot x[2i-1], \quad i \geq 1,\)
    \item no non-trivial \(d_2\) or \(d_3\)-differentials exist.
\end{enumerate}
\end{proposition}

\begin{proof}
\begin{enumerate}
\item Each even-dimensional cell of \(B\mu_{2^n}\) is attached to the preceding odd dimensional cell by degree \(h^n\) map. 
This gives the stated \(d_1\)-differentials.  

\item The \(d_2\)-differentials corresponds to  the Hopf map  which is zero in \(kgl\), so the \(d_2\)-differentials vanishes.  

For the \(d_3\)-differentials, recall they are determined by Toda brackets of the form
\[
\langle x[i], \eta, h^n \rangle \quad (\text{for }i \equiv 0 \pmod 4, \; i>0)
\text{ and }
\langle x[i], h^n, \eta \rangle \quad (\text{for }i \equiv 2 \pmod 4, \; i>2).
\]
Using the relation \(h^{n-1} v_1 z \in \langle z, \eta, h^n \rangle\) and  \(h^{n-1} v_1 z \in \langle z, h^n, \eta \rangle\), we would obtain
\[
d_3(z[4k]) \;=\; h^{n-1} v_1 \cdot z[4k-3] \quad \text{ and } \]
\[
d_3(z[4k-2]) \;=\; h^{n-1} v_1 \cdot z[4k-1] \quad  \forall \; k \ge 1.
\]
However, there are no $d_3$ differentials because all the target classes are $h^{n-1}$ torsion.
\end{enumerate}
\end{proof}

\begin{remark}
In this paper, our primary technique is to compute the Atiyah-–Hirzebruch differentials using the cellular filtration. However, beyond the $E_3$-page, the higher differentials are not determined directly by this approach. In particular, there may exist potential differentials, as shown in \cref{fig:E2page}.

\end{remark}

\begin{figure}[H]\vspace{-0.3cm}
	\includegraphics[scale=0.4]{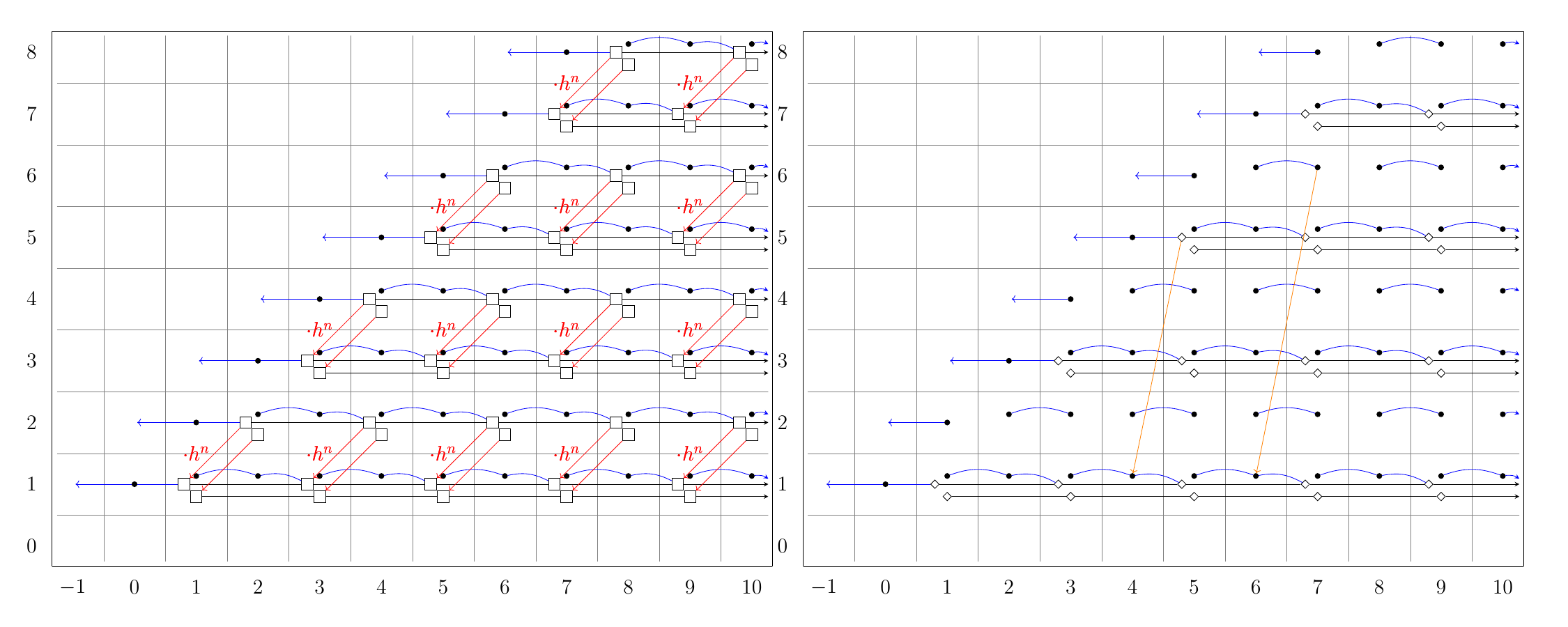}
	\vspace{-0.9cm} 
 \caption{The chart represents the \(E_1\)-page (left), and the \(E_4\)-page with some potential higher differentials (right) for the Atiyah--Hirzebruch spectral sequence converging to \(kgl_{**}^{\r}B\mu_{2^n}\). Here \(\bullet = \z_2[\tau^4]\), $\diamond = \z/2^n[\tau^4]$ and horizontal blue and black lines represent $\rho$- and $v_1$-multiplication respectively.}
 \label{fig:E2page}
\end{figure}

To determine whether these differentials actually occur (see \cref{fig:kgl over R}) and to resolve the hidden extension problems, we appeal to the Adams spectral sequence (\cref{ASS}).
\begin{thm}\label{thm:r-kgl}
As a module over the \(E_\infty\)-page of the Adams spectral
sequence converging to \(kgl_{**}^{\r}\), the \(E_\infty\)-term
of the Adams spectral sequence converging to
\(kgl_{**}^{\r} B\mu_{2^n}\) is given as follows.
\begin{enumerate}
    \item For arbitrary \(n\), through stem \(14\),
    \[
    E_\infty\Bigl(kgl_{k,*}^{\r} B\mu_{2^n}\Bigr)
    \cong
    \begin{cases}
        \begin{aligned}
            &\Bigl(\z/2^{n+i}[\tau^2]\Bigr)^{\oplus 2}
            \oplus
            \Bigl(\z/2^{n-1}[\tau^2]\Bigr)^{\oplus 2i}
            \\
            &\qquad {}\oplus
            \f_2[\tau^2]^{\oplus (i+1)}
            \oplus \f_2[\tau^2,\rho],
        \end{aligned}
        & 0 < k = 2i+1 \leq 14,
        \\[1.5ex]
        \f_2[\tau^2]^{\oplus i}
        \oplus \f_2[\tau^2,\rho],
        & 0 < k = 2i \leq 14,
        \\[1ex]
        \f_2[\tau^2,\rho],
        & \text{otherwise}.
    \end{cases}
    \]

    \item For \(n=2\),
    \[
    E_\infty\Bigl(kgl_{k,*}^{\r} B\mu_4\Bigr)
    \cong
    \begin{cases}
        \begin{aligned}
            &\Bigl(\z/2^{i+2}[\tau^2]\Bigr)^{\oplus 2}
            \oplus
            \Bigl(\z/2[\tau^2]\Bigr)^{\oplus 2i}
            \\
            &\qquad {}\oplus
            \f_2[\tau^2]^{\oplus (i+1)}
            \oplus \f_2[\tau^2,\rho],
        \end{aligned}
        & 0 < k = 2i+1,
        \\[1.5ex]
        \f_2[\tau^2]^{\oplus i}
        \oplus \f_2[\tau^2,\rho],
        & 0 < k = 2i,
        \\[1ex]
        \f_2[\tau^2,\rho],
        & \text{otherwise}.
    \end{cases}
    \]
\end{enumerate}
\end{thm}

\begin{proof}
	Recall from \cref{prop: motivic description},
\[
H^{**}(kgl^{\r}) \cong \ca^{\r} \ddash \ce^{\r}(1).
\]
and from \cref{cor: a0e1} we have
\[
H^{**}B\mu_{2^n} \;\cong\; \bigoplus_{k \geq 1} \;\Sigma^{k,\lfloor \frac{k+1}{2}\rfloor} \m^{\r}_2.
\]
Using these and change of rings isomorphism to simplify the $E_2$-page of the motivic Adams spectral sequence we get, 
\[
E_2 \cong \bigoplus_{k\geq 1} \Ext^{***}_{\ce(1)^\r}(\m_2^\r)\{\tilde{x}_k\}.
\]
The extension of scalars from the $\r$-motivic stable homotopy category to the $\c$-motivic stable homotopy category induces a morphism of Adams spectral sequences which is compatible with all differentials. In particular, any Adams differential detected in the $\c$-motivic Adams spectral sequence must also occur in the $\r$-motivic Adams spectral sequence up to $\rho$-multiples, provided the relevant classes are present; see \cite{DI16a} for details.

In the $\c$-motivic Adams spectral sequence, the differentials are known to occur; see \cref{thm:c-kgl}. Since the extension-of-scalars map preserves the generators $h_0$, $v_1$, and the classes $\tilde{x}_j$, and because the target degrees contain no $\rho$-multiples of the appropriate weight, these differentials carry over directly to the $\r$-motivic setting.

Therefore, for arbitrary \(n\), the Adams differentials through stem \(14\) are given by
\[
\hspace{0.6cm}d_n(\tilde{x}_2)=h_0^n\tilde{x}_1,
\qquad
\text{and} \qquad
\text{for } 2\leq k\leq 7,\;
d_n(\tilde{x}_{2k})
=
h_0^n\tilde{x}_{2k-1}
+h_0^{n-1}v_1\tilde{x}_{2k-3}.
\]
These differentials continue beyond stem 14 for $n=2$,
\[
d_2(\tilde{x}_2)=h_0^2\tilde{x}_1
\qquad
\text{and} \qquad
\text{for } k>1,\;
d_2(\tilde{x}_{2k})
=
h_0^2\tilde{x}_{2k-1}
+h_0v_1\tilde{x}_{2k-3}.
\]

\begin{figure}[H] \vspace{-0.5cm}
	\includegraphics[scale=0.44]{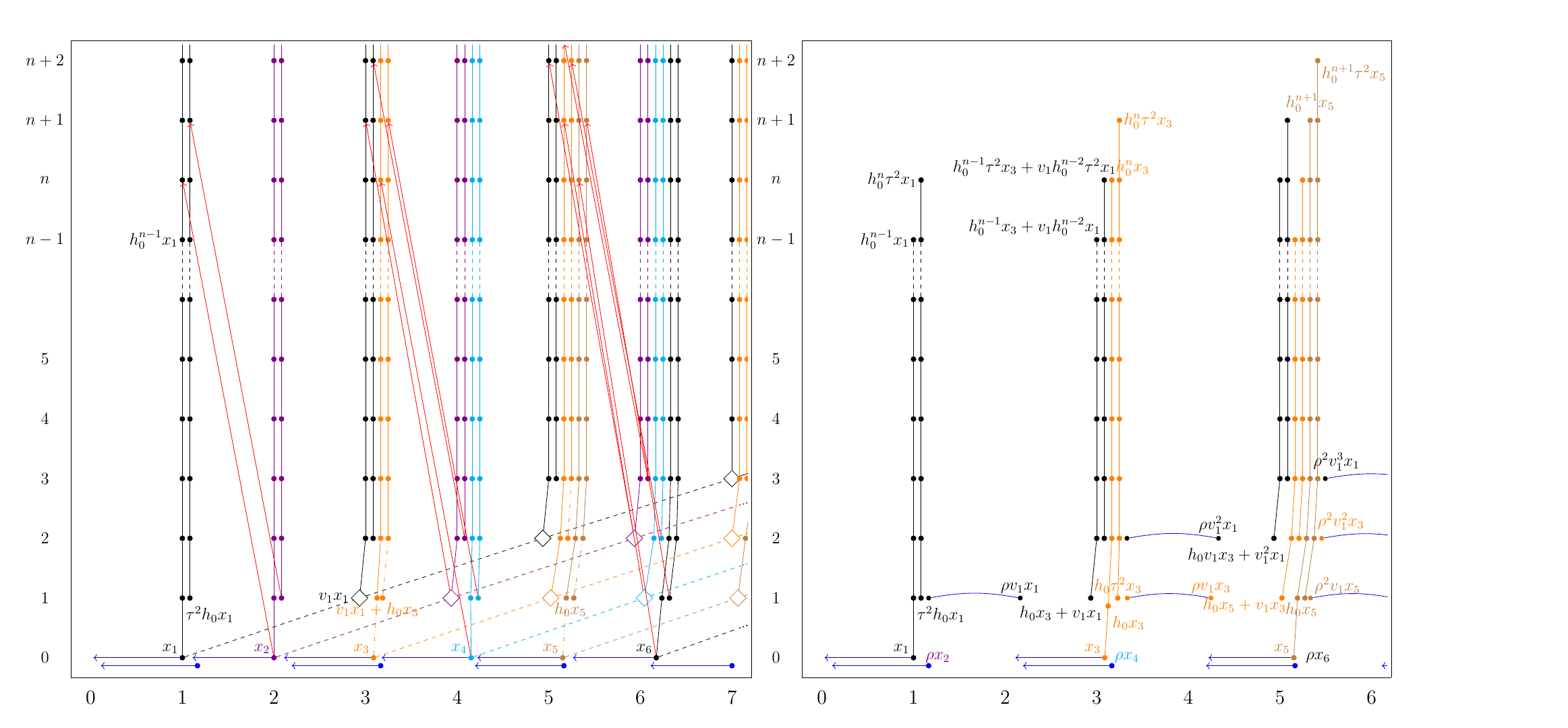}
	\vspace{-0.8cm} \caption{The $E_n$-page (left), where the $d_n$-differentials are drawn in red, and the $E_{\infty}$-page (right) of the Adams spectral sequence converging to $kgl_{k,*}^{\r} B\mu_{2^n}$. Each $\bullet$ represents $\f_2[\tau^2]$, \textbf{ $\diamond$ represents $\f_2[\tau^2,\rho]/\rho^3$ }, vertical solid lines represent multiplication by $h_0$; dash--dot--dash lines indicate that a generator $\kappa_i$ maps to $h_0 \kappa_i + v_1 \kappa_{i-2}$, dashed slanted lines 
	represent multiplication by $v_1$ and left arrowed horizontal (blue) lines represents $\rho$-tower.}
	\label{fig:kgl over R}
\end{figure}
Moreover, these propagate by \(h_0\), \(v_1\), and \(\tau^2 h_0\)-linearity. Therefore, once the differential on \(\tilde{x}_{2i}\) is understood, all corresponding differentials on its multiples follow automatically. Finally, the Adams spectral sequence collapses at the \(E_{n+1}\)-page, as illustrated in \cref{fig:kgl over R}, which rules out the possibility of any further nontrivial differentials through stem \(14\) for arbitrary \(n\) and in all stems for \(n=2\).

This completes the proof.
\end{proof}

\begin{remark}
After this calculation, we can definitely say that the higher differentials in the Atiyah--Hirzebruch spectral sequence do not occur, and the spectral sequence therefore yields the complete associated graded object. 

We emphasize that this reverse approach of deducing Atiyah--Hirzebruch differentials via the Adams spectral sequence is efficient and conceptually powerful, allowing spectral sequences to be used in-combination to ultimately leading to the desired result.
\end{remark}

\section{Very Effective Hermitian K-Theory}\label{kq}

As discussed in \cref{theories}, in motivic homotopy theory, we have the notion of “connective” cover of the Hermitian $K$-theory spectrum $KQ$, analogous to the classical real K-theory spectrum $KO$.

In this section, our goal is to compute the very effective Hermitian $K$-theory of $B\mu_{2^n}$. The computation also proceeds via the Atiyah--Hirzebruch and Adams spectral sequences, in close analogy with the $kgl$-case. 

Again, using the Lefschetz principle described in \cref{subsec:extention}, we can extend the computation over $\c$ and $\r$ to any algebraically closed field of characteristic not equal to 2 and any real closed field, respectively.

\subsection{Very Effective Hermitian K-theory over algebraically closed fields}\label{kqC}
In this section we compute $kq_{**}^{\c}(B\mu_{2^n})$, which uses $kq_{**}^{\c}$ as a main ingredient in the spectral sequences.  

\begin{remark}\label{kq_*}
As shown in \cite[Theorem~4.10]{IS11} with chart \cite[Figure~ A.1]{IS11}, the bigraded homotopy ring of very effective Hermitian $K$-theory is
\[
kq_{**}^{\c}:= \pi_{**}(kq) \;\cong\; 
\dfrac{\z[\tau, \eta, \alpha, \beta]}{(2\eta, \tau \eta^3, \eta \alpha, \alpha^2 - 4\beta)}.
\]
\end{remark}

So, let's start with the Atiyah--Hirzebruch spectral sequence, which gives

\begin{proposition}
	The associated graded groups are given by
		\[
gr^{*}_{AH} kq_{k,*}^{\c}(B\mu_{2^n}) \cong
\begin{cases}
	\z/2^n[\tau] \bigoplus \Big(\z/2^{n-1}\Big)^{\oplus 2m}[\tau], & k = 8m+1> 0, \\
	\f_2[\tau], & k = 8m+2> 0, \\
	\Big(\f_2[\tau]\Big)^{\oplus 2m+2} \bigoplus \Big(\z/2^{n}[\tau]\Big)^{\oplus 2m+1}, & k = 8m+3> 0, \\
	(\f_2)^{\oplus m+1}, & k = 8m+4, 8m+6, 8m> 0, \\
	\Big(\z/2^{n-1}\Big)^{\oplus 2m+2}[\tau] \oplus (\f_2)^{\oplus m+1}, & k = 8m+5> 0, \\
	\Big(\z/2^{n}[\tau] \oplus \f_2[\tau]\Big)^{\oplus 2m+2} \oplus (\f_2)^{\oplus m+1}, & k = 8m+7> 0, \\
\end{cases}\]
\end{proposition}
\begin{proof} By \cref{subsec:AHSS}, the $E_1$-page of the Atiyah-–Hirzebruch spectral sequence is 
\[
E_1 = \bigoplus_{\substack{\text{each cell } i \\ \text{in } B\mu_{2^n}}}  \Sigma^{i,\lfloor \frac{i+1}{2}\rfloor} kq_{**}^{\c}\Rightarrow gr^{**}_{AH}kq_{**}^{\c}(B\mu_{2^n}),
\]
with nontrivial differentials computed similar to classical ones
\[
\begin{aligned}
	d_1(x[2i]) &= 2^n x[2i-1], \quad i\geq 1, \\
	d_2(x[i]) &= \eta x[i-2], \quad i = 4k, 4k+1 \geq 4, \\
	d_3(\tau \eta^2[4i]) &= 2^{n-1} \alpha[4i-3] \quad \textit{ as } 2^{n-1}\alpha \in <\tau \eta^2, \eta, 2^n>  \text{ shown in \cref{fig:kqAHSS} }.
\end{aligned} \]
The spectral sequence converges at $E_4$-page because of degree reasons and $\eta$, $\alpha$, $\beta$-linearity.
\end{proof}
\begin{figure}[H] \vspace{-0.5cm}
	\includegraphics[width=0.9\linewidth]{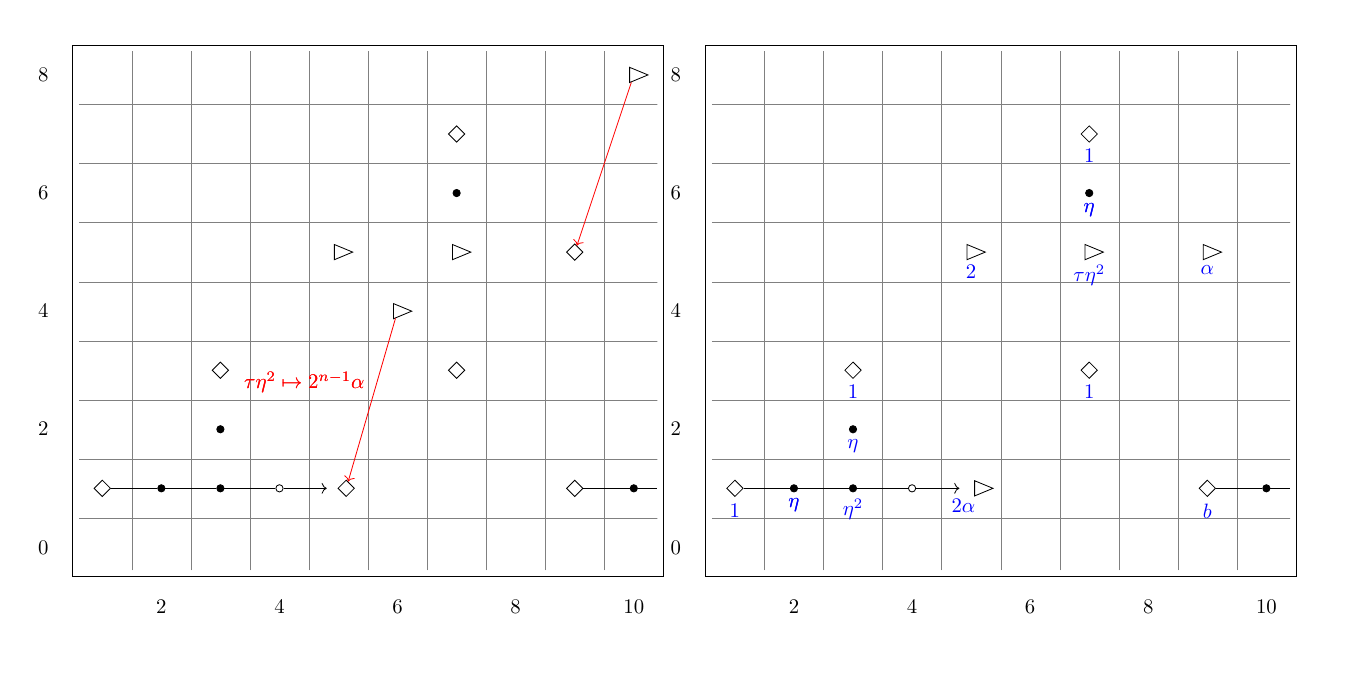}
	\vspace{-0.5cm} 
	\caption{The $ E_3$(left) and $E_\infty$(right) page of Atiyah--Hirzebruch spectral sequence converging to $gr^*_{AH} kq_{**} B\mu_{2^n}^\c$. Each $\diamond$  represents $\z/2^n[\tau]$, $\triangleright $ represents $\z/2^{n-1}[\tau]$, $\bullet$ represents $\f_2[\tau]$ and each $\circ$ represents $\f_2$. Horizontal lines represent multiplication by $\eta$. Differentials $d_3$ are shown in red.}
	\label{fig:kqAHSS}
\end{figure}

Again extension problems are resolved using the motivic Adams spectral sequence in \cref{ASS}. So, recall from \cref{prop: motivic description} that, analogously, we have
\[
H^{**}(kq^{\c}) \cong \ca^{\c} \ddash \ca^{\c}(1).
\]
Using this and a change of rings isomorphism to simplify the $E_2$-page of the motivic Adams spectral sequence we get

\[
E_2 = \Ext_{\ca^{\c}(1)}^{***}(H^{**}(B\mu_{2^n})) \Rightarrow kq_{**}^{\c}( B\mu_{2^n}).
\]
From \cref{cor: a1} we have
\[
H^{**}(B\mu_{2^n}) = \Sigma^{1,1}\m_2^{\c} \bigoplus_{k\geq 0} \Big( \Sigma^{4k+2,2k+1}\ca^{\c}(1)\ddash \ce^{\c}(1) \oplus \Sigma^{4k+3,2k+2}\ca^{\c}(1)\ddash \ce^{\c}(1)\Big).
\]
Consequently, we get
\[
E_2^{***} = \Ext_{\ca^{\c}(1)}^{***}(\Sigma \m_2^{\c}) \oplus \bigoplus_{k\geq 0} \Big( 
\Ext_{\ce^{\c}(1)}^{***}(\Sigma^{4k+2,2k+1}\m_2^{\c}) \oplus \Ext_{\ce^{\c}(1)}^{***}(\Sigma^{4k+3,2k+2}\m_2^{\c})\Big), 
\]
with
\[\Ext_{\ce^{\c}(1)}^{***}(\m_2^{\c}) \cong\m_2^{\c}[h_0, v_1]\] and
\[\Ext_{\ca^{\c}(1)}^{***}(\m_2^{\c}) \cong \dfrac{\m_2^{\c}[h_0,h_1,a,b]}{(h_0 h_1,\tau h_1^3, h_1 a, a^2 = 4 b)} \text{(by \cref{kq_*})}.\]
Now, we can proceed to analyze the differentials giving us the following theorem.

\begin{thm}\label{thm:c-kq}
The \(E_{\infty}\)-term of the Adams spectral sequence converging to \(kq_{**}^\c B\mu_{2^n}\) as a module over \(E_\infty\)-page of the Adams spectral sequence converging to \(kq_{**}^\c\) is given by

\begin{enumerate}
    \item For arbitrary \(n>1\), through stem \(14\),
    \[
	E_\infty\Bigl(kq_{**}^{\c} B\mu_{2^n}\Bigr)
	\cong
	\frac{E_\infty(kq_{**}^{\c})\{x_1\}}
		{(h_0^n,\, ah_0^{n-1})}
	\;\oplus \;
	\bigoplus_{i=1}^3
	\frac{\Biggl(
	\frac{\m_2^\c[h_0,v_1^2]}
		{h_0^{n+2i-1}}\{x_{4i-1}\}
	\oplus
	\frac{\m_2^\c[h_0,v_1^2]}{h_0^{n-1}}\{v_1x_{4i-1}\}
	\Biggr)}{
	\Bigl( a+h_0v_1^2,
	v_1^2h_0^{n-1}x_3+h_0^{n+1}x_7,\,
	v_1^2h_0^{n-1}x_7+h_0^{n+1}x_{11}
	\Bigr)}.
\]
    \item For \(n=2\),
    \[
    kq_{k,*}^{\c} B\mu_{4}\cong
\begin{cases}
	\z/2^{2m+1}[\tau]\oplus \z/2[\tau], &
	k=8m+1,\quad m\geq 0, \\[2pt]
	\z/2[\tau], &
	k=8m+2,\quad m\geq 0, \\[2pt]
	\z/2^{4m+3}[\tau]\oplus \z/2^{2m+1}[\tau], &
	k=8m+3,\quad m\geq 0, \\[2pt]
	\z/2^{2m+2}[\tau], &
	k=8m+5,\quad m\geq 0, \\[2pt]
	\z/2^{4m+5}[\tau]\oplus \z/2^{2m+1}[\tau], &
	k=8m+7,\quad m\geq 0, \\[2pt]
	0, & \text{otherwise}.
\end{cases}
\]
\end{enumerate}
\end{thm}

\begin{proof}
Continuing from the previous discussion, the \(E_2\)-page of the Adams spectral sequence converging to \(kq_{**}^{\c}(B\mu_{2^n})\) is shown in \cref{fig:kq ASS over C}.
The nontrivial differentials are as follows. For arbitrary \(n\), through stem \(14\), they are
\[
\begin{aligned}
d_n(x_2)&=h_0^n x_1, &
d_n(v_1x_2)&=h_0^{n+1}x_3,\\
d_n(v_1^2x_2)&=a h_0^{n-1}x_1, &
d_n(x_6)&=h_0^{n-1}v_1x_3,\\
d_n(v_1x_6)&=h_0^{n+1}x_7+h_0^{n-1}v_1^2x_3, &
d_n(v_1^2x_6)&=h_0^{n-1}v_1^3x_3,\\
d_n(x_{10})&=h_0^{n-1}v_1x_7, &
d_n(v_1x_{10})&=h_0^{n+1}x_{11}
+h_0^{n-1}v_1^2x_7,\\
d_n(v_1^2x_{10})&=h_0^{n-1}v_1^3x_7, &
d_n(x_{14})&=h_0^{n-1}v_1x_{11}.
\end{aligned}
\]

In the case \(n=2\), this pattern continues beyond stem \(14\), and for \(k>3\) we have
\[
\begin{aligned}
d_2(x_{4k+2})
&=h_0v_1x_{4k-1},\\
d_2(v_1^2x_{4k+2})
&=h_0v_1^3x_{4k-1},\\
d_2(v_1x_{4k+2})
&=h_0^3x_{4k+3}+h_0v_1^2x_{4k-1}.
\end{aligned}
\]
These differentials are \(\tau\) and \(v_1^2\)-linear, and the spectral sequence collapses at \(E_{n+1}\).
The differentials follow from \cref{prop:diff in ko} via Betti realization and propagate by \(h_0\)- and \(v_1^2\)-linearity. Degree considerations rule out any further nontrivial differentials.
This completes the computation of the \(E_{n+1}\)-term of the Adams spectral sequence. The spectral sequence collapses at \(E_{n+1}\) for degree reasons. The resulting \(E_{\infty}\)-term is illustrated in \cref{fig:kq ASS over C}.
\end{proof}

\begin{figure} [H]\vspace{-0.5cm}
	\centering    \includegraphics[width=1.0\linewidth, height=0.32\textheight]{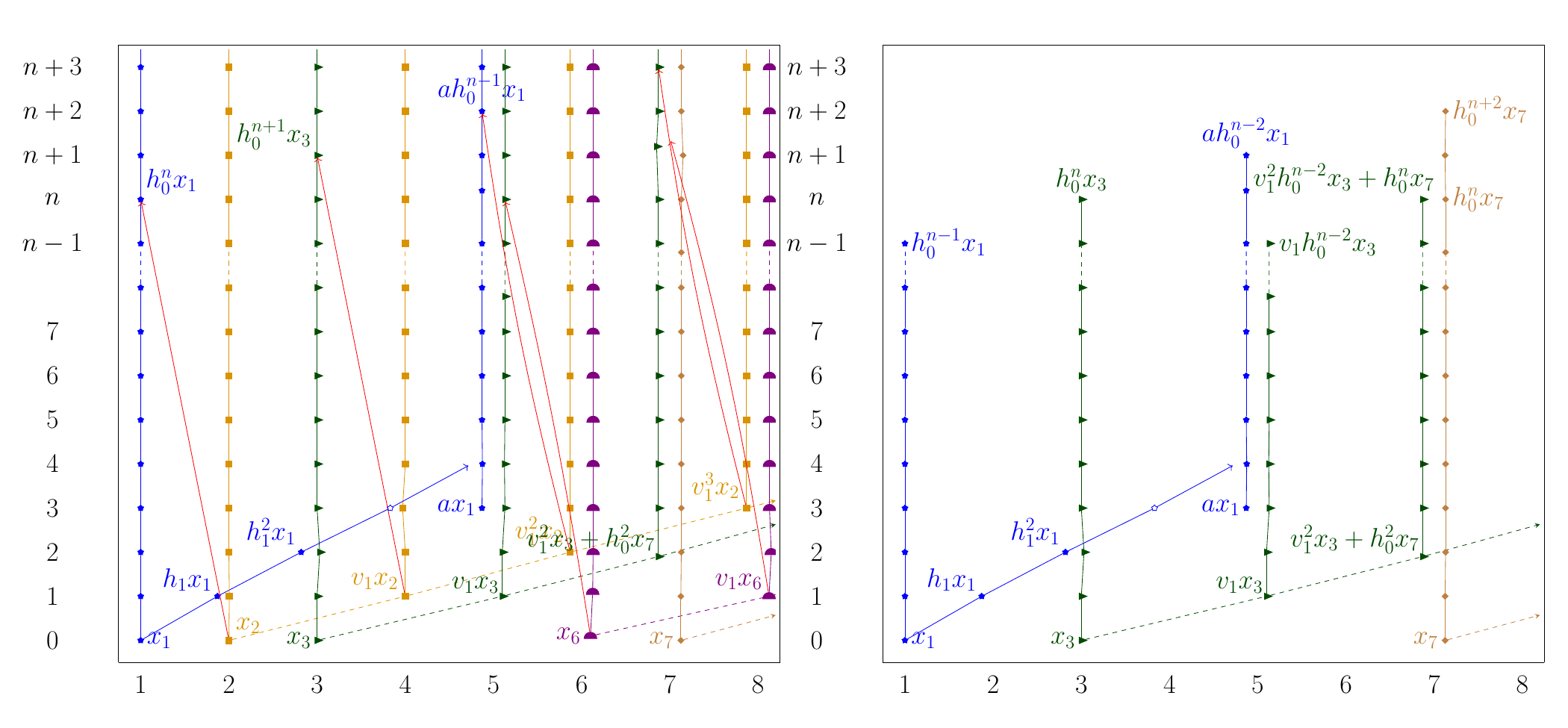}
	\vspace{-1cm} \caption{The $E_n$-page (left), where the $d_n$-differentials are in red, 
		and the $E_{n+1} = E_{\infty}$-page (right) of the Adams spectral sequence converging to 
		$kq_{**} B\mu_{2^n}$. Each bullet represents $\f_2[\tau]$; $\circ$ represents $\f_2$, vertical lines 
		represent multiplication by $h_0$ and slanted lines 
		represent multiplication by $v_1$.}
	\label{fig:kq ASS over C}
\end{figure}

\subsection{Very Effective Hermitian K-theory over real closed fields}\label{kqR}
Finally, we compute $kq_{**}^\r B\mu_{2^n}$. We proceed as before, but the presence of $\rho$ makes the computation more involved. Instead of presenting the full calculation at once, we use the periodicity in the cell structure of $B\mu_{2^n}$ and organize the analysis by a filtration through its cells and their $stem - weight$ pairs.

To begin the Atiyah--Hirzebruch spectral sequence, we first need to understand its input, $kq_{**}^\r$.
\cite{Hil11} uses the $\rho$-Bockstein spectral sequence which has the form
\[
E_1 = \Ext_{\ca(1)^\c}^{***}(\m_2^\c)[\rho] \Rightarrow \Ext_{\ca(1)^\r}^{***}(\m_2^\r).
\]
to compute the $E_2$-term of the motivic Adams spectral sequence converging to $kq_{**}^\r$, \[
E_2 = \Ext_{\ca(1)^\r}^{***}(\m_2^\r) \Rightarrow kq_{**}^\r.
\]
The Adams spectral sequence collapses, giving a complete computation of $kq_{**}^\r$, organized by $coweight = stem - weight$, as shown in \cref{fig:kqR}. See \cite{GHIR19} for detailed calculations.

\begin{figure}[h] \vspace{-0.5cm}
	\includegraphics[scale=0.43]{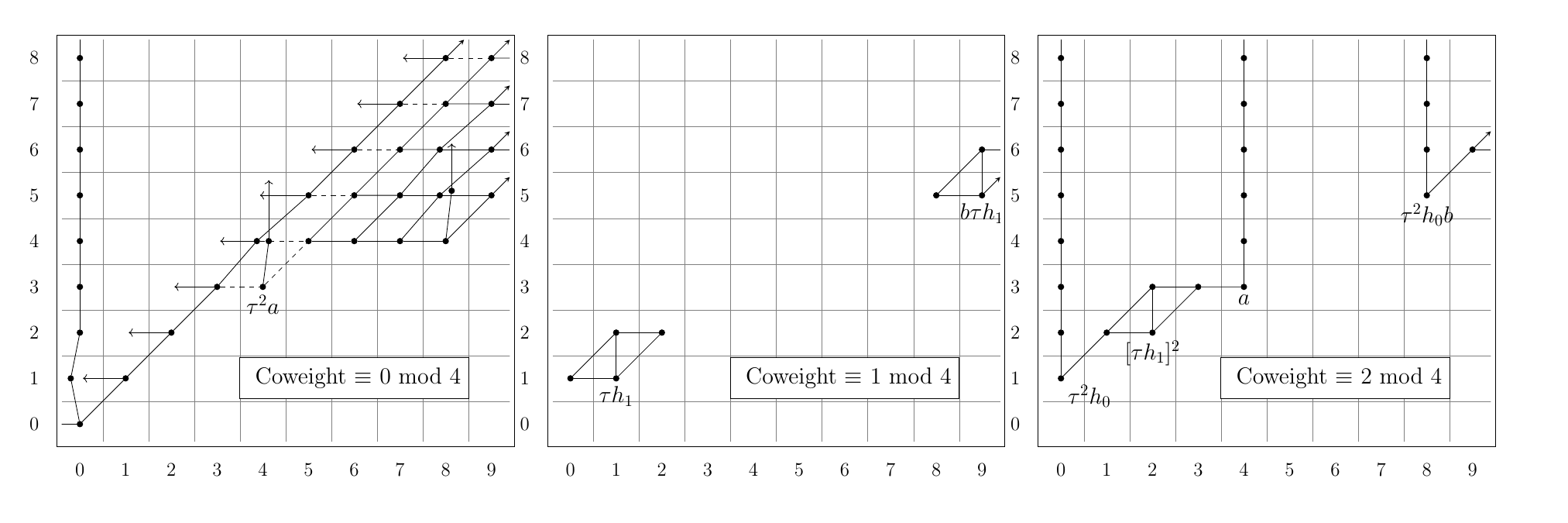}
	\vspace{-0.7cm} \caption{$\Ext_{\ca^\r(1)}^{***}(\m_2^\r)$ organized by coweights modulo $4$ converging to $kq_{**}^\r$. Each $\bullet$ represents $\f_2[\tau^4]$. We have redrawn the charts from \cite[Section 12.1]{GHIR19}, which can be extracted from \cite{Hil11}.}\label{fig:kqR}
\end{figure}

We now analyze the Atiyah--Hirzebruch spectral sequence
\[E_1 = \bigoplus_{\substack{\text{each cell } i \\ \text{in } B\mu_{2^n}}} \Sigma^{i,\lfloor \frac{i+1}{2}\rfloor} kq_{**}^\r \Rightarrow kq_{**}^\r B\mu_{2^n}.\]

Since $B\mu_{2^n}$ has one cell in each positive degree, the $E_1$-page contains a suspended copy of $kq_{**}^\r$ in every positive filtration. Note that we use the $\Ext$-charts for simplicity of the calculation. The generators, relations, hidden extensions and massey products in the $\Ext$-charts are all described in \cite[Section~6]{GHIR19}.
The $E_1$-page is organized by coweight modulo 4.
Also, each even cell attached to the preceding odd cell by $h^n$ map, so the resulting $d_1$-differentials are induced by $h^n$ attaching map.

Thus, the $E_2$-page of the Atiyah–-Hirzebruch spectral sequence is given by
\[
E_2^{i,**} \;\cong\;
\begin{cases}
	kq_{**}/h^n & \text{if } i = 2k+1 > 0, \\[6pt]
	\ker\!\Big(h^n : kq_{**} \to kq_{**}\Big) & \text{if } i = 2k > 0.
\end{cases}
\]

From the chart perspective, for odd filtrations ($i = 2k+1$) the $E_2$-page is obtained by
removing classes divisible by $h^n$(target of $h^n$-multiplication), while for even filtrations($i = 2k+2$) it is obtained by removing classes that support nontrivial multiplication by $h^n$(source of $h^n$-multiplication). This produces the charts shown in \cref{fig:kqAH_E2_f0} for even filtrations and \cref{fig:kqAH_E2_f1} for odd filtrations.
\begin{figure}[H] \vspace{-0.5cm}
\includegraphics[scale=0.42]{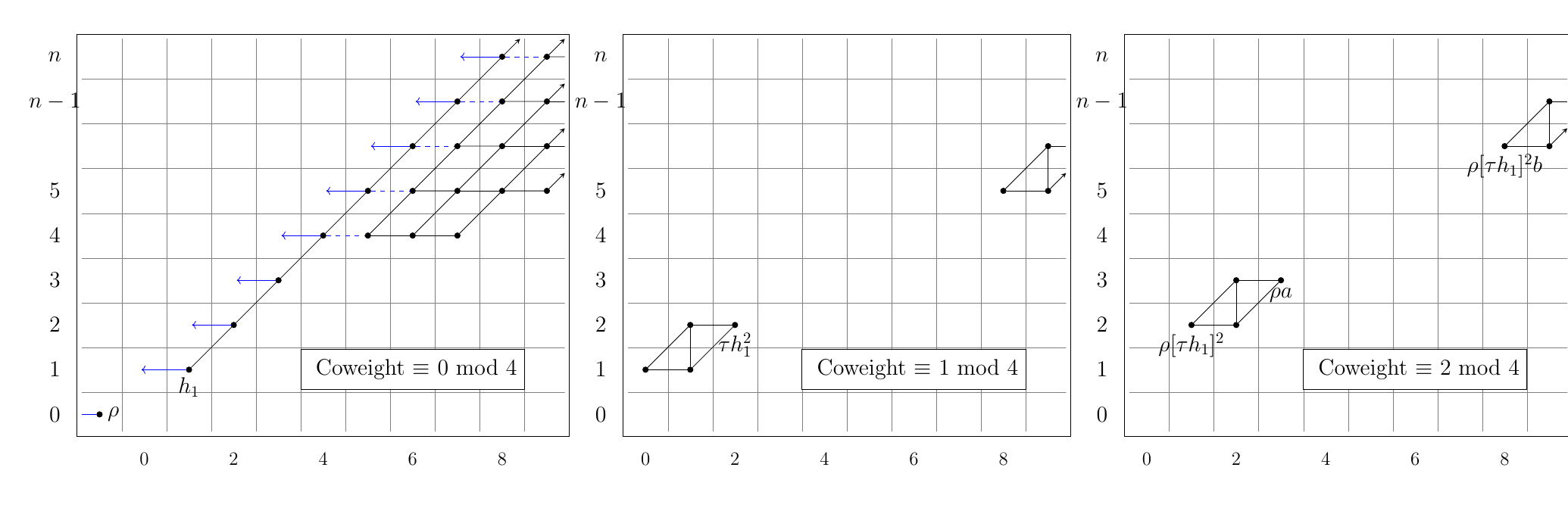}
	\vspace{-0.7cm} \caption{Charts representing the even filtrations of the $E_2$-term of the Atiyah--Hirzebruch spectral sequence converging to $kq_{**}^\r B\mu_{2^n}$ organized by coweights. Each $\bullet$ represents $\f_2[\tau^4]$.}\label{fig:kqAH_E2_f0}
\end{figure}

\begin{figure}[H] \vspace{-1cm}
	\includegraphics[scale=0.43]{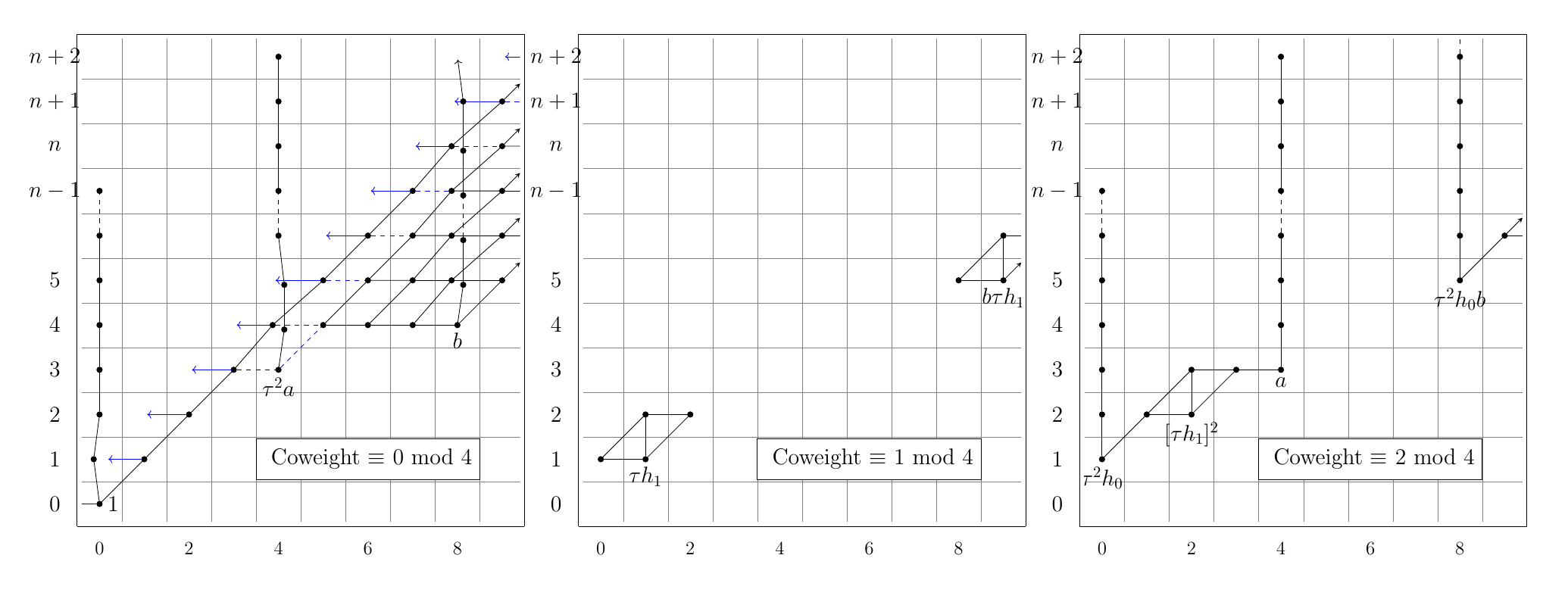}
	\vspace{-0.7cm} \caption{Charts representing the odd filtrations of the $E_2$-term of the Atiyah--Hirzebruch spectral sequence converging to $kq_{**}^\r B\mu_{2^n}$ organized by coweights. Each $\bullet$ represents $\f_2[\tau^4]$.}\label{fig:kqAH_E2_f1}
\end{figure}

Next, the $d_2$-differentials are induced by the Hopf map, attaching the $4i$ and $4i+1$ cells to the $4i-2$ and $4i-1$ cells, respectively, yielding
\[
d_2(x[4i]) = \eta \cdot x[4i-2], \quad i>0, \qquad
d_2(x[4i+1]) = \eta \cdot x[4i-1], \quad i>0.
\]
Recall, $d_2$ increases the stem $t-s$ by $1$, and increases the coweight by $1$  giving that it preserve the weight and pairity. The $E_3$-page is obtained by taking the kernel and cokernel of multiplication by $\eta$ on each row of $E_2$. Explicitly,
\[
E_3^{i,**} \;\cong\;
\begin{cases}
	\ker(\eta : kq_{**}/h^n \to kq_{**}/h^n) & \text{if } i = 4k+1 > 4, \\[2mm]
	\ker(h^n) / \eta & \text{if } i = 4k+2 > 0, \\[1mm]
	(kq_{**}/h^n)/\eta & \text{if } i = 4k+3, \\[1mm]
	\ker(\eta : \ker(h^n) \to \ker(h^n)) & \text{if } i = 4k+4 > 0.
\end{cases}
\]

From the chart perspective, each row of the $E_3$-page can be visualized as the corresponding row of $E_2$, with classes that are $\eta$-divisible or support nontrivial multiplication by $\eta$ removed, depending on the filtration. This produces the charts shown in \cref{fig:kqAH_E3_f0}, \cref{fig:kqAH_E3_f1}, \cref{fig:kqAH_E3_f2}, and \cref{fig:kqAH_E3_f3} for each congruence class of the filtration modulo $4$.

\begin{figure}[H] \vspace{-0.2cm}
	\includegraphics[scale=0.39]{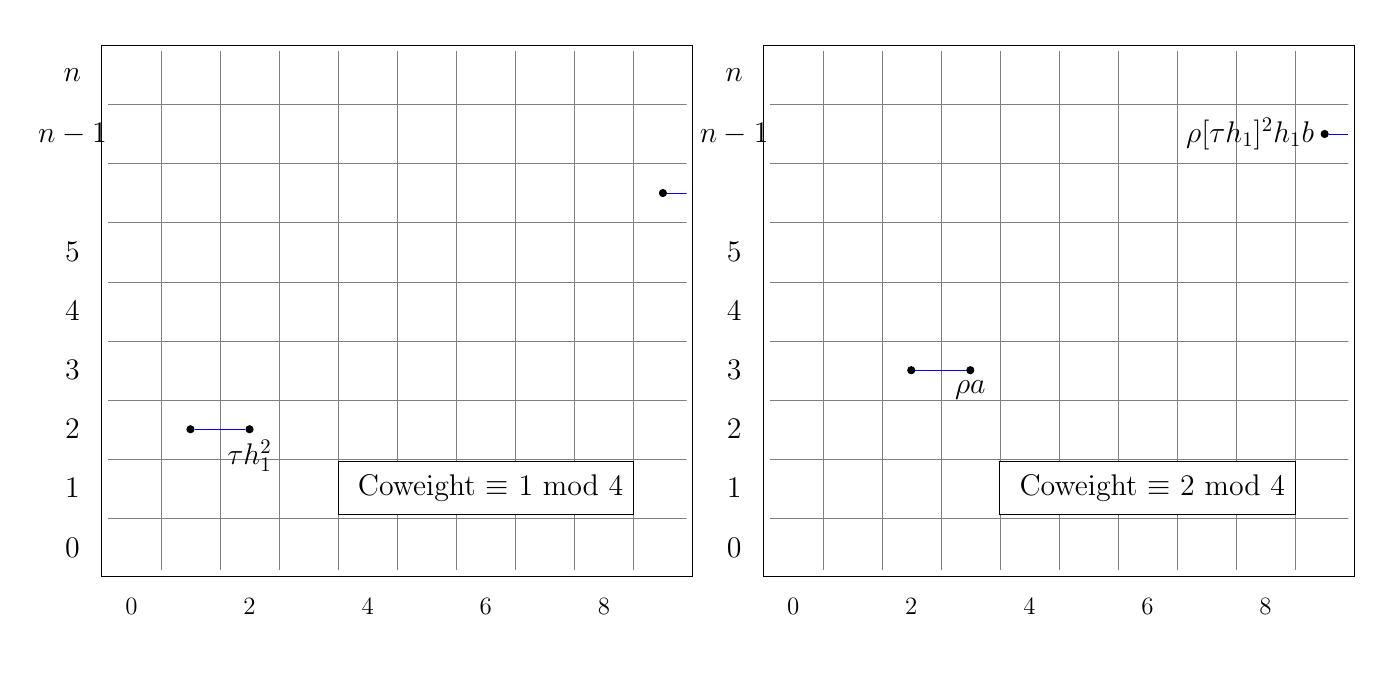}
	\vspace{-0.7cm} \caption{Charts represents filtrations congruent to $0$ mod $4$ of the $E_3$-page of the Atiyah--Hirzebruch spectral sequence organized by coweights since the other coweights vanish. Each $\bullet$ represents $\f_2[\tau^4]$.}\label{fig:kqAH_E3_f0}
\end{figure}
\begin{figure}[H] \vspace{-0.4cm}
	\includegraphics[scale=0.38]{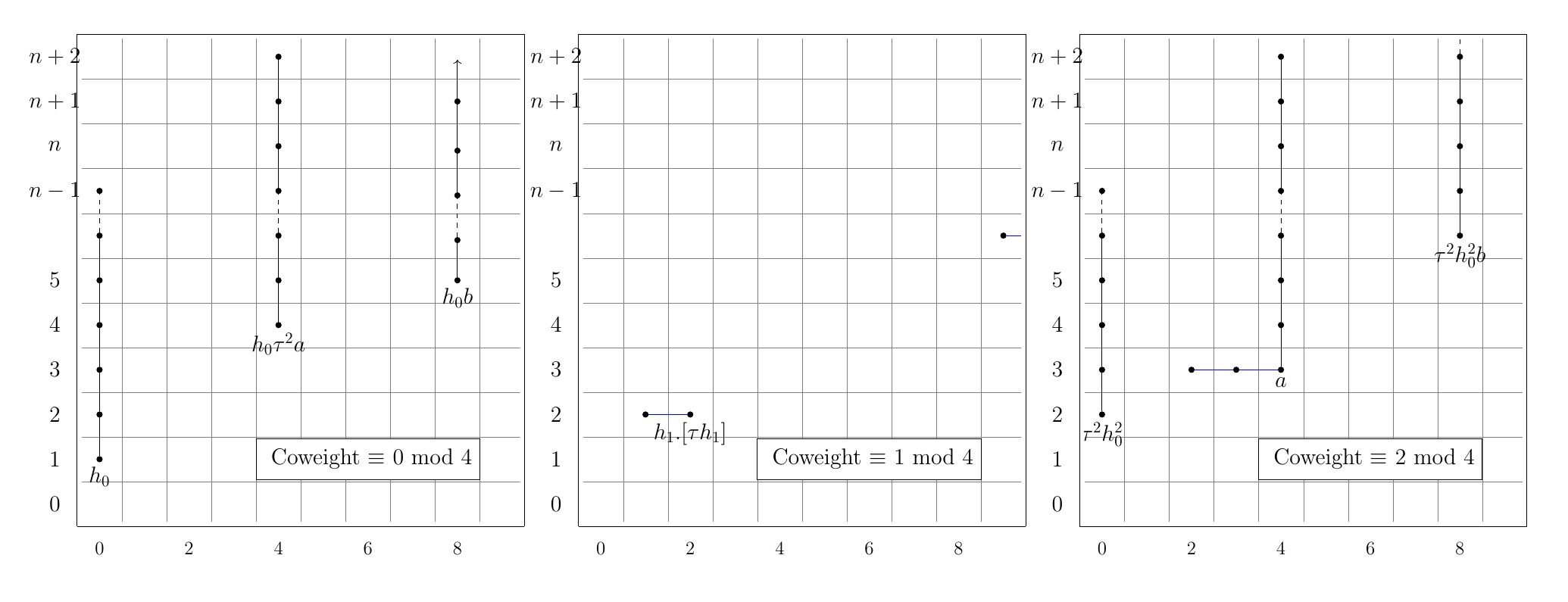}
	\vspace{-0.7cm} \caption{Charts represents filtrations congruent to $1$ mod $4$ and greater than 1 organized by coweights. Each $\bullet$ represents $\f_2[\tau^4]$.}\label{fig:kqAH_E3_f1}
	\vspace{-0.1cm}
\end{figure}
\begin{figure}[H] \vspace{-0.2cm}
	\includegraphics[scale=0.38]{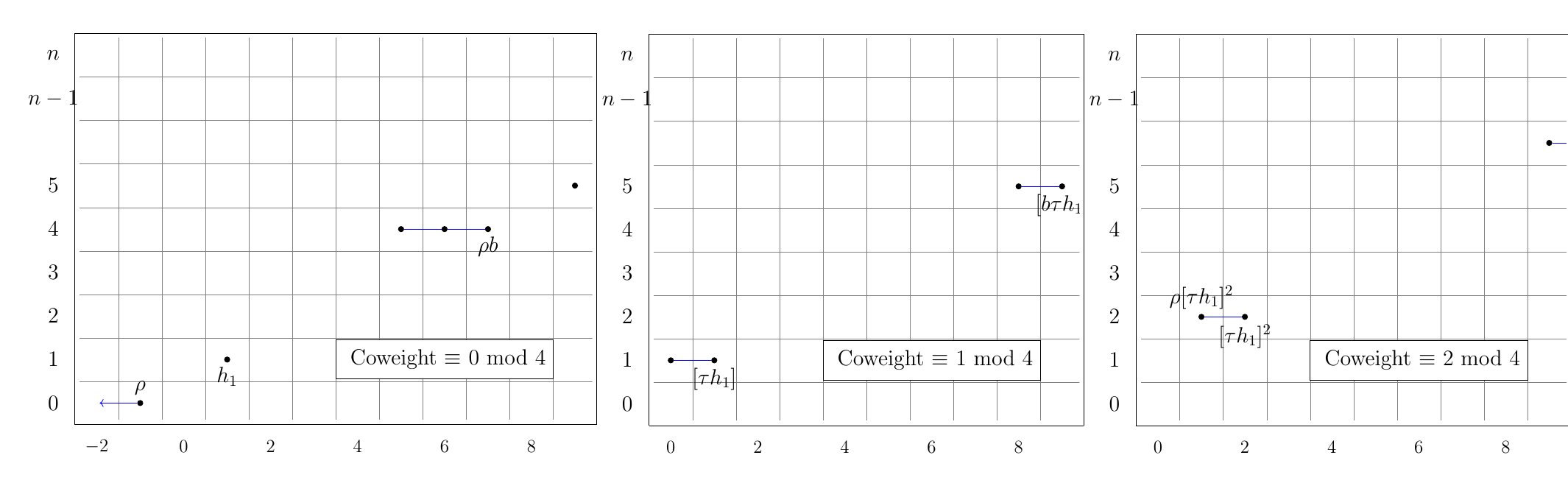}
	\vspace{-0.7cm} \caption{Charts represents filtrations congruent to $2$ mod $4$ of the $E_3$-page organized by coweights. Each $\bullet$ represents $\f_2[\tau^4]$.}\label{fig:kqAH_E3_f2}
	\vspace{-0.1cm}
\end{figure}
\begin{figure}[H] \vspace{-0.2cm}
	\includegraphics[scale=0.38]{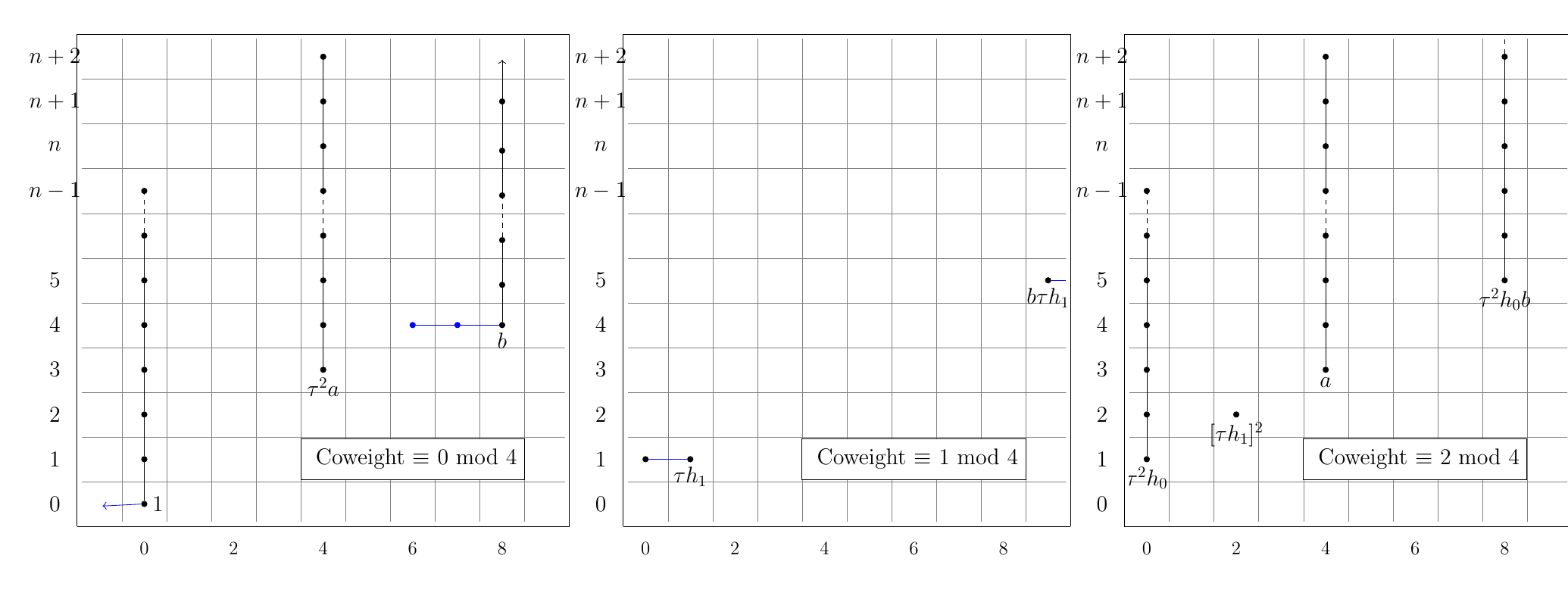}
	\vspace{-0.7cm} \caption{Charts represents filtrations congruent to $3$ mod $4$ of the $E_3$-page organized by coweights. Each $\bullet$ represents $\f_2[\tau^4]$.}\label{fig:kqAH_E3_f3}
\end{figure}

\begin{proposition}
Up to Bott periodicity, the only nontrivial $d_3$-differential is
\[
d_3(h_1 \cdot \tau h_1) = h_0^{n-1} a.
\]
\end{proposition}

\begin{proof}
Recall that the $d_3$-differentials decrease the filtration by 3, increase the stem $t-s$ by $2$, and increase the coweight by $1$. 
They arise from Toda brackets of the form:
\[
d_3(x[i]) \;\in\;
\begin{cases}
\langle x[i], \eta, h\rangle, & i \equiv 0 \pmod 4,\;\; i>0, \\[1mm]
\langle x[i], h, \eta\rangle, & i \equiv 2 \pmod 4,\;\; i>2.
\end{cases}
\]

In particular,
\[
h_0^{n-1} a \in \langle \tau h_1 \cdot h_1, h_1, h_0^n \rangle,
\]
which gives a $d_3$-differential connecting \cref{fig:kqAH_E3_f0} to \cref{fig:kqAH_E3_f1}, from the class $h_1 \cdot \tau h_1$ (the dot $(2,2)$ in the left chart of \cref{fig:kqAH_E3_f0}, corresponding to coweight $1 \bmod 4$) to $h_0^{n-1} a$ (the dot $(4,n+2)$ in the right chart of \cref{fig:kqAH_E3_f1}, corresponding to coweight $2 \bmod 4$).

On the other hand,
\[
d_3(\rho \tau h_1 \cdot h_1) = \rho h_0^{n-1} a = 0,
\]
so the class $\rho \tau h_1 \cdot h_1$, along with the remaining classes in \cref{fig:kqAH_E3_f0} and \cref{fig:kqAH_E3_f1}, survives to the next page, since there are no nontrivial targets for the differential.

There are no possible $d_3$-differentials from \cref{fig:kqAH_E3_f2} to \cref{fig:kqAH_E3_f3}, because $\langle \rho, h_0^n, h_1 \rangle = 0$ and degree considerations. Similarly, there are no $d_3$-differentials originating in filtrations $1$ or $3 \bmod 4$, as there are no primary attaching maps from filtration $1 \bmod 4$ to $2 \bmod 4$, or from $3 \bmod 4$ to $0 \bmod 4$.
\end{proof}

\begin{remark}
The row in filtration $i=1$ remains unchanged from the $E_2$-page, since no $d_2$-differentials enter or leave that row.
\end{remark}

By assembling these charts after applying the $d_3$-differentials, we obtain the complete $E_4$-page shown in \cref{fig:kq_AHSS_BC_{2^n}_E4all}.  

Finally, note that the $d_4$-differentials decrease the vertical degree by four, the horizontal degree by one and increases the coweight by two, which rules out possibility of higher differentials for degree reasons. Hence, the Atiyah--Hirzebruch spectral sequence for $kq_{**}^\r B\mu_{2^n}$ collapses at the $E_4$-page. This determines the associated graded object $\mathrm{gr}^*_{AH} kq_{**}^\r B\mu_{2^n}$.

\vspace{-0.5cm}
\begin{figure}[H]
	\includegraphics[width=15.25cm]{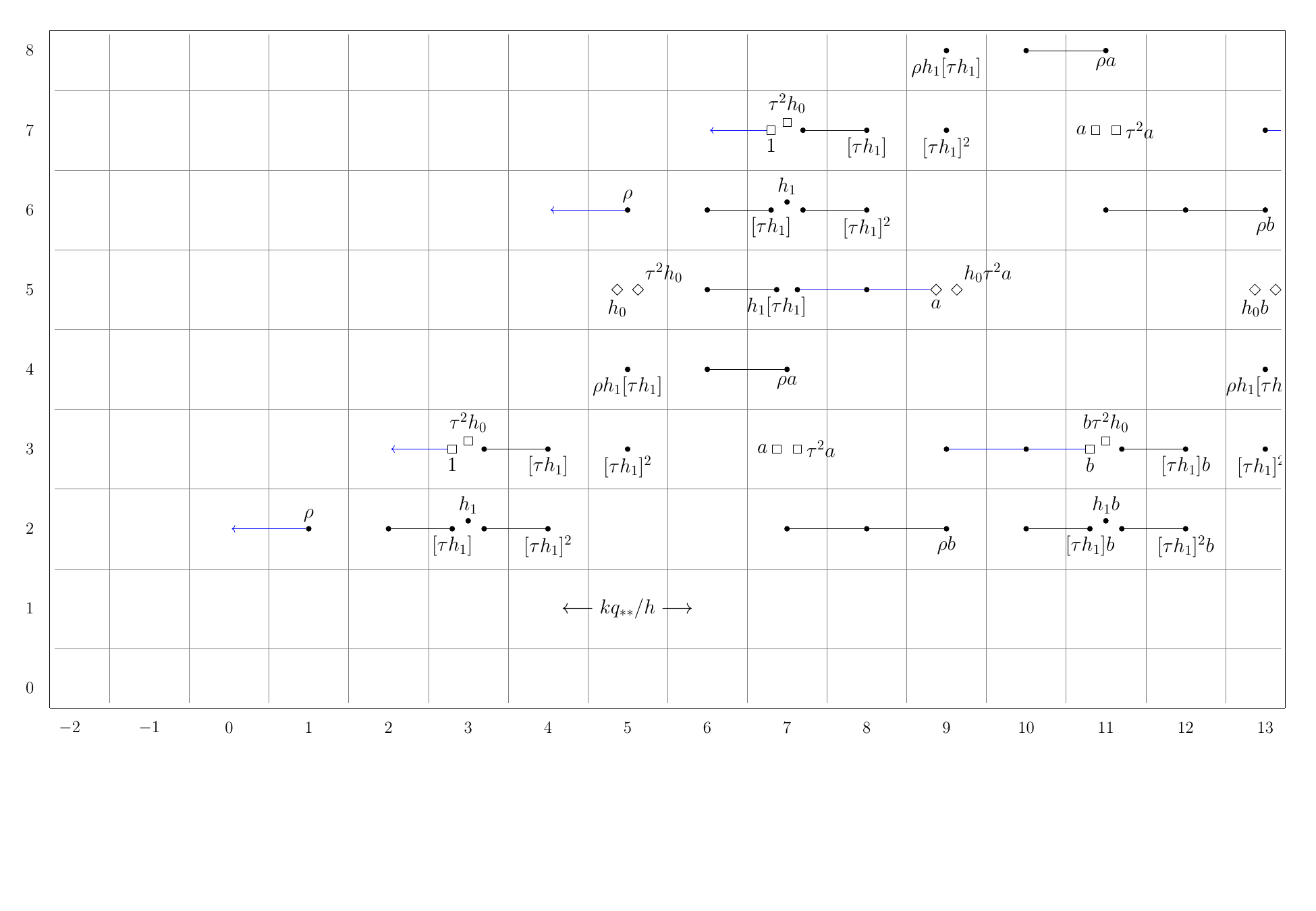}
	\vspace{-2.5cm} \caption{Chart representing the whole $E_4$-page of the Atiyah--Hirzebruch spectral sequence converging to $gr^{*}_{AH} kq_{**}^\r (B\mu_{2^n})$. Here $\square = \z/2^n[\tau^4],$\\ $\diamond = \z/2^{n-1}[\tau^4] $ and $\bullet = \f_2[\tau^4]$.}
	\label{fig:kq_AHSS_BC_{2^n}_E4all}
\end{figure}

\vspace{0.5cm}
To solve extension problems we use the motivic Adams spectral sequence, which has $E_2$-page: 
\[
E_2 = \Ext_{\ca^{\r}(1)}^{***}(H^{**}(B\mu_{2^n})) \Rightarrow kq_{**}^{\r}( B\mu_{2^n}).
\]
From \cref{cor: a1} we have
\[
H^{**}(B\mu_{2^n}) = \Sigma^{1,1}\m_2^{\r} \bigoplus_{k\geq 0} \Big( \Sigma^{4k+2,2k+1}\ca^{\r}(1)\ddash \ce^{\r}(1) \oplus \Sigma^{4k+3,2k+2}\ca^{\r}(1)\ddash \ce^{\r}(1)\Big).
\]
Consequently,
\[
E_2^{***} = \Sigma^{1,1} \Ext_{\ca^\r (1)}^{***}(\m_2^\r ) \oplus \bigoplus_{k\geq 0} \Big( 
\Sigma^{4k+2,2k+1}\Ext_{\ce^\r (1)}^{***}(\m_2^\r ) \oplus \Sigma^{4k+3,2k+2}\Ext_{\ce^\r (1)}^{***}(\m_2^\r )\Big).
\]

These summands are drawn in the charts below, see \cref{fig:E_n kq}.

\begin{thm}\label{thm:r-kq}
	The \(E_\infty\)-page of the motivic Adams spectral sequence converging to \(kq_{**}^\r (B\mu_{2^n})\) as a module over the \(E_\infty(kq^{\r}_{**}\))  is given by the charts in \cref{fig:kq ASS over R}.
\end{thm}
\begin{proof}
Since we have already determined the $E_2$-page, we now determine the differentials.
First, for the cell in degree~1, the group \(kq_{\ast\ast}^{\r}/h\) appears, corresponding to the black chart for ranks.
\begin{figure}[H]
   \includegraphics[width=5.9in]{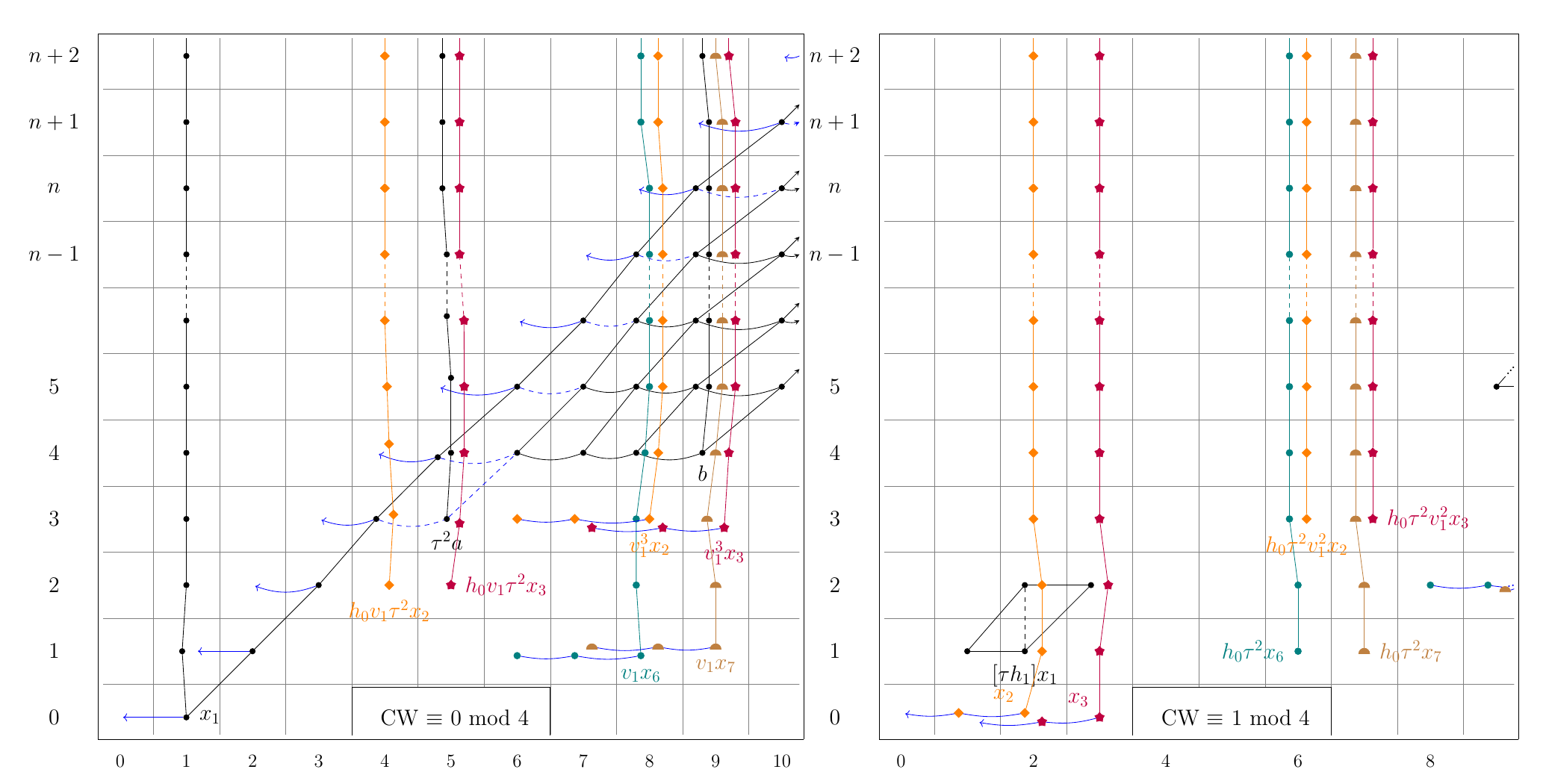}
     \includegraphics[width=6.1in]{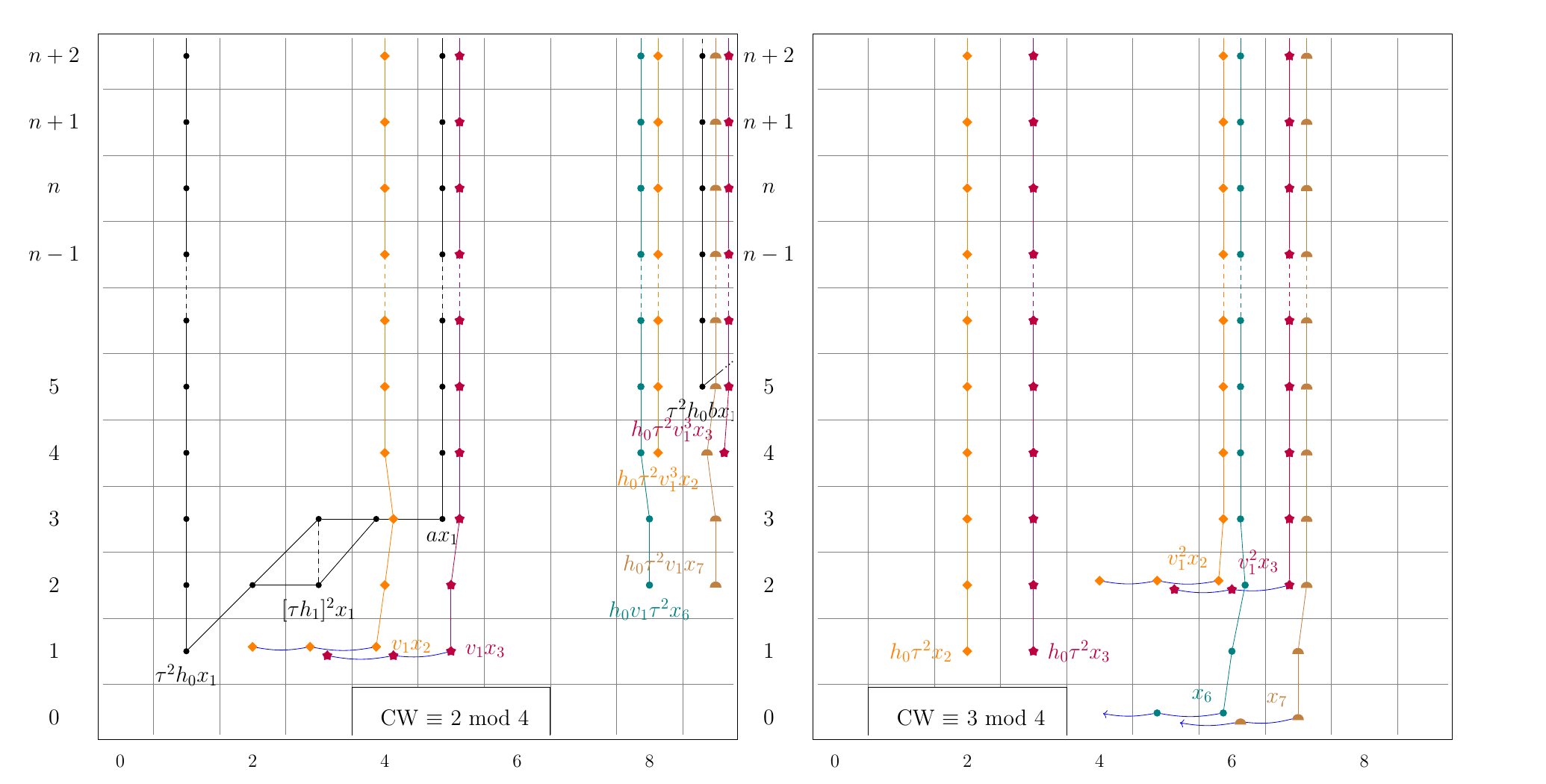} 
     \vspace{-0.9cm}
   \caption{Chart representing the $E_n$-page of the Adams spectral sequence converging to $kq_{**}^\r 
   		B\mu_{2^n}$, organized by coweights modulo 4.
		Each $\bullet$ represents $\f_2[\tau^4]$, and green lines 
		indicate multiplication by $\rho$.}
   \label{fig:E_n kq}
\end{figure}

For arbitrary \(n\), the nontrivial Adams differentials through stem \(14\) are as follows:

The differentials in stem{1-4} are obtained by comparison with the Atiyah--Hirzebruch spectral sequence where the $\rho$-towers correspond to $\rho$-towers in the filtration 0 in the Adams spectral sequence. However, in filtration~1, there is a copy of \(\z/2^n[\tau^4]\) in coweight congruent to \(1\) modulo~\(4\), a copy of \(\f_2[\tau^4]\) in coweight congruent to \(2\) modulo~\(4\), and a copy of \(\z/2^{n-1}[\tau^4]\) in coweight congruent to \(3\) modulo~\(4\). These correspond to the differentials
\[
d_n(x_2)=h_0^n x_1,
\qquad
d_n(h_0\tau^2x_2)=\tau^2h_0^{n+1}x_1,
\]
since \(d_n\) increases the filtration  by \(n\) and decreases the stem by \(1\).
Also,the differentials are \(h_0\tau^2\)-linear as \(h_0\tau^2\) is a permanent class. Therefore, \((\mathrm{i})\) gives 
\(
d_n(h_0\tau^2x_2)=\tau^2h_0^{n+1}x_1,
\) directly.

The differential
\(
d_n(v_1x_2)=h_0^{n+1}x_3
\)
is obtained by comparing with the Atiyah--Hirzebruch spectral sequence in filtrations~3 and~4; see the AHSS chart in \cref{fig:kq_AHSS_BC_{2^n}_E4all}.

In filtration~3, there remain two copies of \(\z/2\) and four copies of \(\f_2\), together with classes from the \(\rho\)-towers. In filtration~4, only two copies of \(\f_2\) remain, along with the \(\rho\)-tower classes. This forces the entire \(v_1x_2\)-tower to kill the \(x_3\)-tower, respecting the degree.

The \(h_0\tau^2\)-linearity of the spectral sequence then gives
\[
d_n(h_0\tau^2v_1x_2) =h_0\tau^2d_n(v_1x_2)=h_0^{n+2}\tau^2x_3.
\]
This is also reflected in the Atiyah--Hirzebruch spectral sequence: the entire tower on \(h_0v_1\tau^2x_2\) must kill the tower on \(\tau^2x_3\), beginning with the class \(h_0^{n+1}\tau^2x_3\).

The extension-of-scalars functor again induces a morphism of Adams spectral sequences compatible with all differentials. Under this morphism, \(\rho\) maps to zero, while \(h_0\), \(v_1\), and the classes \(x_j\) map to their corresponding \(\c\)-motivic classes.

The corresponding \(\c\)-motivic differentials were established in \cref{thm:c-kq}. Naturality therefore gives
\[
d_n(x_{4k+2})
=h_0^{n-1}v_1x_{4k-1} \quad\text{and}\quad
d_n(v_1x_{4k+2})
=h_0^{n+1}x_{4k+3}
+h_0^{n-1}v_1^2x_{4k-1}.
\]
Thus, in the \(\r\)-motivic Adams spectral sequence, the differentials must have the displayed terms, possibly together with $\rho$-multiple terms that vanish after extension of scalars.
However, since the target classes are $\rho$-torsion, the differentials coincide with the \(\c\)-motivic differentials upto linearity.

These differentials exhibit \(v_1^2\)-linearity, which follows as in the classical case.

\medskip
For \(n=2\), the corresponding \(\c\)-motivic differentials continue in all stems. Applying the same extension-of-scalars argument lifts these differentials to the \(\r\)-motivic Adams spectral sequence. Comparison with the \(\r\)-motivic Atiyah--Hirzebruch spectral sequence and the structure of the \(\rho\)-towers again rule out any additional \(\rho\)-multiple terms. Hence, the families in \((\mathrm{v})\) and \((\mathrm{vi})\) hold for every \(k>0\) when \(n=2\).

This completes the computation of the $E_n$-term of the Adams spectral sequence. The spectral sequence collapses for degree reasons, and the final $E_{n+1}$-term is illustrated in \cref{fig:kq ASS over R}. 
\end{proof}

\begin{figure}[H] \vspace{-0.5cm}
	\centering    
	\includegraphics[width=5.9in]{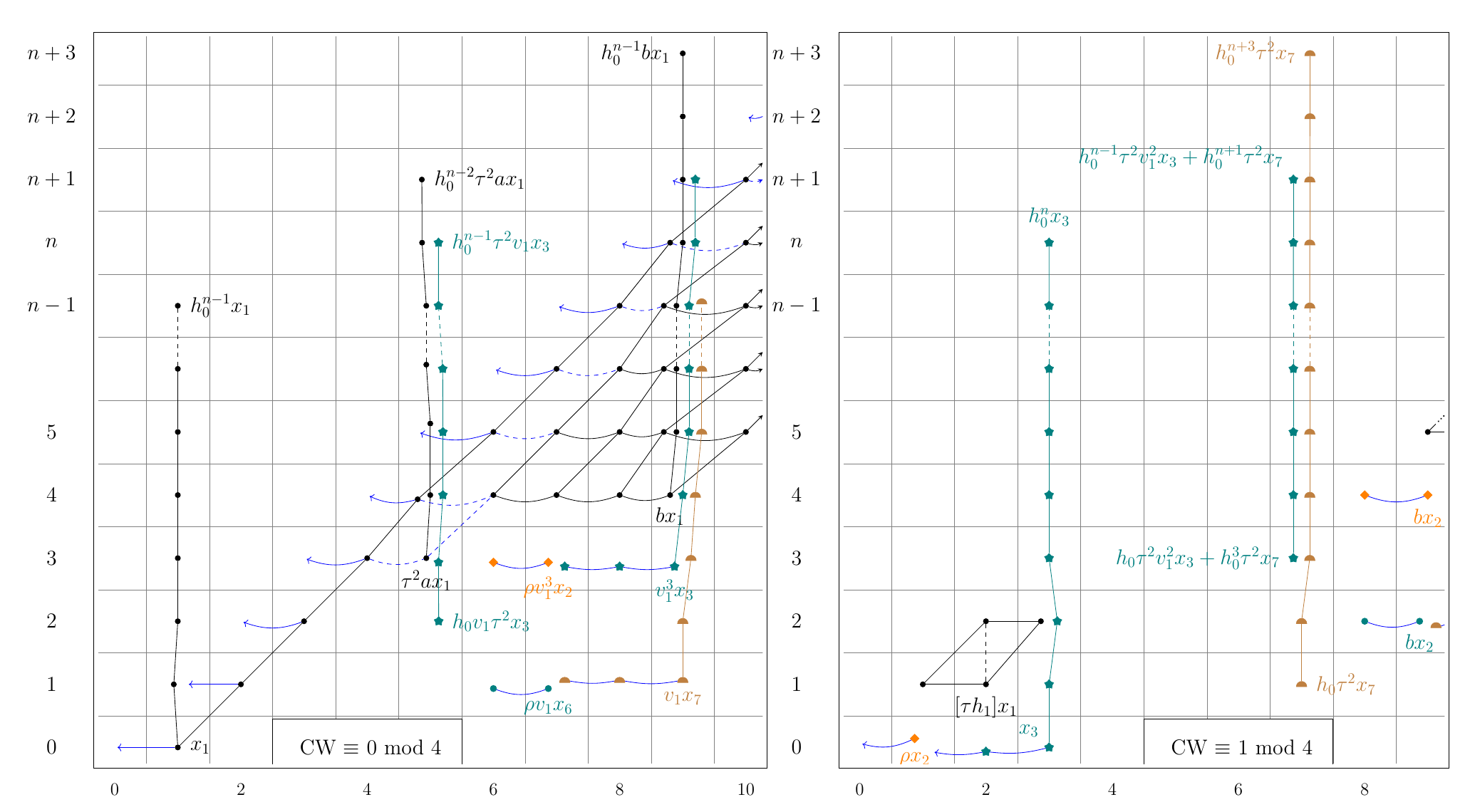}	 \vspace{-0.5cm}
	\includegraphics[width=6.0 in]{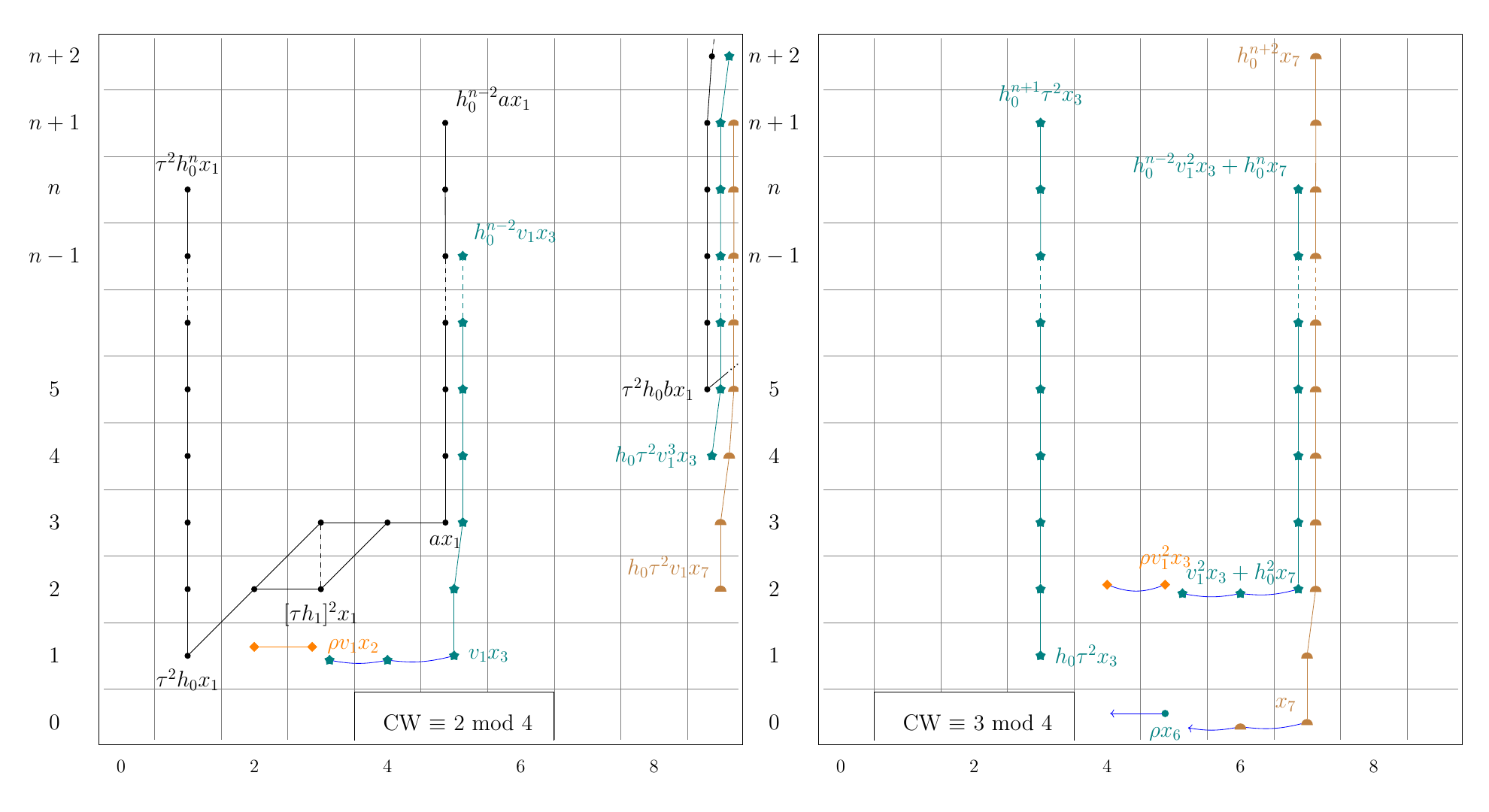}
	\caption{
		The chart represents $E_\infty(kq_{k,w}^\r (B\mu_{2^n}))$, organized by coweights modulo 4: 
		Each bullet represents $\f_2[\tau^4]$.
	}

	\label{fig:kq ASS over R}
\end{figure}


\appendix

\section{An eta-invariant approach to differentials on indecomposable classes}
\label{app:eta-invariant}

In this appendix, we use the eta invariant to detect an element in
\(ko_7(BC_{2^n})\) whose order is divisible by \(2^{n+3}\). The idea is
to construct a class from a seven-dimensional lens space and then use
the eta-invariant homomorphism of \cite[Theorem~4.1]{BGS97} to estimate
its order.

\begin{proposition}
For \(n\geq 2\), the group \(ko_7(BC_{2^n})\) contains an element whose
order is divisible by \(2^{n+3}\).
\end{proposition}

Before proving the proposition, we recall the necessary background on
lens spaces and eta invariants.

Let \(k\geq 2\), and let \(q_1,\ldots,q_\ell\) be integers that are
coprime to \(k\). The lens space
\[
L(k;q_1,\ldots,q_\ell) = S^{2\ell-1}/C_k,
\]
is given by a fixed generator acting by the action
\(
(z_1,\ldots,z_\ell)
\longmapsto
\Bigl(
e^{2\pi i q_1/k}z_1,\ldots,
e^{2\pi i q_\ell/k}z_\ell
\Bigr).
\)
The universal covering determines a classifying map \( c\colon L(k;q_1,\ldots,q_\ell)\longrightarrow BC_k. \)

If \(\ell\) is even or \(k\) is odd, then the lens space is spin; see
\cite[the discussion preceding Theorem~4.6]{BGS97}. After choosing a
spin structure, the pair
\[
\Bigl(L(k;q_1,\ldots,q_\ell),c\Bigr)
\]
determines a class
\[
\Bigl[L(k;q_1,\ldots,q_\ell),c\Bigr]
\in
\Omega_{2\ell-1}^{\mathrm{Spin}}(BC_k).
\]
Applying the Atiyah--Bott--Shapiro orientation gives a class in real
connective \(K\)-homology,
\[
\Omega_{2\ell-1}^{\mathrm{Spin}}(BC_k)
\longrightarrow
ko_{2\ell-1}(BC_k),
\]
which is $7$-connected; see also \cite{ABP67} for details.

To detect the order of this class, we use the periodicization map
\[
p\colon
ko_{2\ell-1}(BC_k)
\longrightarrow
KO_{2\ell-1}(BC_k),
\]
followed by an eta invariant on periodic \(KO\)-homology. By
\cite[Theorem~4.1]{BGS97}, every virtual representation \(\rho\) of
\(C_k\) of virtual dimension zero determines a homomorphism
\[
\eta(\rho)\colon
KO_{2\ell-1}(BC_k)
\longrightarrow
\mathbb{R}/\mathbb{Z}.
\]
Composing with periodicization gives a homomorphism
\[
\overline{\eta}(\rho)
:=
\eta(\rho)\circ p
\colon
ko_{2\ell-1}(BC_k)
\longrightarrow
\mathbb{R}/\mathbb{Z}.
\]
The main idea of the proof proceeds as follows: 

Since a group homomorphism cannot increase the order of an element, if
\(\overline{\eta}(\rho)(x)\) has order \(m\), then the order of
\(x\in ko_{2\ell-1}(BC_k)\) is divisible by \(m\).

\begin{proof}
Set \(k=2^n\), and let
\(
\rho_s\colon C_k\longrightarrow U(1)
\)
denote the one-dimensional representation for which a fixed generator
of \(C_k\) acts by \(e^{2\pi i s/k}\). Since
\(
\rho_s-\rho_{s'}
\)
has virtual dimension zero, it determines a homomorphism
\[
\overline{\eta}_{s,s'}
:=
\overline{\eta}\Bigl(\rho_s-\rho_{s'}\Bigr)
\colon
ko_{2\ell-1}(BC_k)
\longrightarrow
\mathbb{R}/\mathbb{Z}.
\]
Writing \(\eta_s^{\mathrm{Dirac}}\) for the eta invariant of the Dirac
operator twisted by the flat bundle associated to \(\rho_s\), we have
\[
\overline{\eta}_{s,s'}
=
\Bigl(
\eta_s^{\mathrm{Dirac}}
-
\eta_{s'}^{\mathrm{Dirac}}
\Bigr)\circ p.
\]

We now specialize to stem \(7\), so take \(\ell=4\), and consider the
lens space
\(
L_n:=L\Bigl(2^n;1,1,1,1\Bigr).
\)
Also, let
\(
c_n\colon L_n\longrightarrow BC_{2^n}
\)
denote the classifying map of its universal covering. Since \(\ell=4\)
is even, \(L_n\) is spin, and the Atiyah--Bott--Shapiro orientation sends its spin bordism class
\[
[L_n,c_n]\in\Omega_7^{\mathrm{Spin}}(BC_{2^n})
\mapsto
x_n\in ko_7(BC_{2^n}).
\]
We detect the order of \(x_n\) using the eta-invariant homomorphism
described above.

For
\(
\mathbf{q}=(q_1,\ldots,q_\ell),
\)
Gilkey's formula \cite[Theorem~2.5(c)]{Gilkey84} gives
\[
\eta_s^{\mathrm{Dirac}}
\Bigl(L(2^n;\mathbf{q})\Bigr)
=
-\frac{1}{2^n}
\widehat{A}_\ell
\Bigl(
s+2^{n-1};q_1,\ldots,q_\ell
\Bigr)
\in\mathbb{R}/\mathbb{Z},
\]
where
\[
\widehat{A}_\ell(s;\mathbf{x})
=
\sum_{a+2b=\ell}
\frac{s^a}{a!}\widehat{A}_b(\mathbf{x}).
\]
Here $\widehat A$ is the Hirzebruch $\widehat A$-genus, described in \cite[Section~23.1]{BH59-II}.

For the calculation in stem \(7\), we use
\[
\widehat{A}_1(\mathbf{x})
=
-\frac{1}{24}\sum_i x_i^2
\qquad \text{and} \qquad
\widehat{A}_2(\mathbf{x})
=
\frac{1}{5760}
\left(
-4\sum_{i<j}x_i^2x_j^2
+
7\left(\sum_i x_i^2\right)^2
\right).
\]

For \(\ell=4\) and \(\mathbf{q}=(1,1,1,1)\), these become
\[
\widehat{A}_1(1,1,1,1)
=
-\frac{1}{6}
\qquad \text{and} \qquad 
\widehat{A}_2(1,1,1,1)
=
\frac{1}{5760}
\left(
-4\sum_{i<j}1
+
7\left(\sum_{i=1}^{4}1\right)^2
\right)
=
\frac{11}{720}.
\]
It follows that
\[
\widehat{A}_4(s;1,1,1,1)
=
\frac{s^4}{4!}
+
\frac{s^2}{2!}\widehat{A}_1(1,1,1,1)
+
\widehat{A}_2(1,1,1,1) \\
=
\frac{s^4}{24}
-
\frac{s^2}{12}
+
\frac{11}{720}.
\]
Consequently,
\[
\widehat{A}_4
\Bigl(
s+2^{n-1};1,1,1,1
\Bigr)
=
\frac{\Bigl(s+2^{n-1}\Bigr)^4}{24}
-
\frac{\Bigl(s+2^{n-1}\Bigr)^2}{12}
+
\frac{11}{720},
\]
and hence
\[
\eta_s^{\mathrm{Dirac}}
\Bigl(L\Bigl(2^n;1,1,1,1\Bigr)\Bigr)
=
-\frac{1}{2^n}
\left(
\frac{\Bigl(s+2^{n-1}\Bigr)^4}{24}
-
\frac{\Bigl(s+2^{n-1}\Bigr)^2}{12}
+
\frac{11}{720}
\right)
\in\mathbb{R}/\mathbb{Z}.
\]

We now evaluate the relative eta invariant
\(
\eta_1^{\mathrm{Dirac}}-\eta_0^{\mathrm{Dirac}}
\)
on \(L_n\). The constant terms cancel, giving
\[
\begin{aligned}
\Bigl(
\eta_1^{\mathrm{Dirac}}
-
\eta_0^{\mathrm{Dirac}}
\Bigr)(L_n)
&=
-\frac{1}{2^n}
\Bigl(
\frac{\Bigl(2^{n-1}+1\Bigr)^4}{24}
-
\frac{\Bigl(2^{n-1}+1\Bigr)^2}{12}
-
\frac{2^{4n-4}}{24}
+
\frac{2^{2n-2}}{12}
\Bigr) \\
&=
-\frac{
\Bigl(2^{n-1}+1\Bigr)^4
-
2\Bigl(2^{n-1}+1\Bigr)^2
-
2^{4n-4}
+
2^{2n-1}
}{
24\cdot 2^n
}
\in\mathbb{R}/\mathbb{Z}.
\end{aligned}
\]

Set
\[
u_n
=
\Bigl(2^{n-1}+1\Bigr)^4
-
2\Bigl(2^{n-1}+1\Bigr)^2
-
2^{4n-4}
+
2^{2n-1}.
\]
Then
\[
\Bigl(
\eta_1^{\mathrm{Dirac}}
-
\eta_0^{\mathrm{Dirac}}
\Bigr)(L_n)
=
-\frac{u_n}{3\cdot 2^{n+3}}
\in\mathbb{R}/\mathbb{Z}.
\]
For \(n\geq 2\), all terms in the definition of \(u_n\), except the
first, are even, while
\[
\Bigl(2^{n-1}+1\Bigr)^4\equiv 1\pmod 2.
\]
Thus \(u_n\) is odd. Therefore, no factor of \(2\) cancels from the
denominator, and
\(
\overline{\eta}_{1,0}(x_n)\in\mathbb{R}/\mathbb{Z}
\)
has order divisible by \(2^{n+3}\).

The Atiyah--Hirzebruch spectral sequence shows that
\(ko_7(BC_{2^n})\) is a finite \(2\)-group. We may therefore write
\[
ko_7(BC_{2^n})
\cong
\bigoplus_{i=1}^{r}\mathbb{Z}/2^{j_i}.
\]
Suppose, for contradiction, that
\[
j_i\leq n+2
\qquad
\text{for every }1\leq i\leq r.
\]
Then multiplication by \(2^{n+2}\) annihilates every cyclic summand, so
\(
2^{n+2}x_n=0.
\)
Applying \(\overline{\eta}_{1,0}\) gives
\(
2^{n+2}\overline{\eta}_{1,0}(x_n)=0,
\)
contradicting the fact that
\(\overline{\eta}_{1,0}(x_n)\) has order divisible by \(2^{n+3}\).

Therefore,
\(
j_i\geq n+3
\)
for some \(i\). The corresponding cyclic summand contains an element
whose order is divisible by \(2^{n+3}\).
\end{proof}

\begin{remark}
The eta-invariant method described above also applies in higher stems.
However, by itself it provides limited information about the
differentials supported by indecomposable classes.
\end{remark}

\begin{remark}
Another possible approach uses the Bockstein homology spectral sequence.
This approach is used in \cite{BB96} and \cite{Bay94}, where this spectral sequence is combined
with the May--Milgram theorem to relate Bockstein differentials to
differentials between \(h_0\)-towers. Applying this comparison in the
present setting requires certain duality identifications that we do not
discuss here.
\end{remark}

\bibliographystyle{alpha}
\bibliography{master}

\end{document}